\documentclass[11pt]{amsart}
\usepackage{amssymb,amsmath,epsfig,mathrsfs, enumerate, xparse}
\usepackage[nodisplayskipstretch]{setspace}
\usepackage{graphicx}
\usepackage[normalem]{ulem}
\usepackage{fancyhdr}
\usepackage{tikz}
\usepackage{esint}
\usepackage{circuitikz}
\usepackage[margin=2.5 cm]{geometry}
\usepackage{hyperref}
\hypersetup{
 colorlinks   = true,
 urlcolor     = blue,
 linkcolor    = blue,
 citecolor   = red ,
 bookmarksopen=true
}
\usepackage[noabbrev]{cleveref}
\theoremstyle{plain}
\newtheorem{thm}{Theorem}
\newtheorem{prop}{Proposition}[section]
\newtheorem{lem}[prop]{Lemma}
\newtheorem{cor}[prop]{Corollary}

\newtheorem{defi}[prop]{Definition}
\newtheorem{rmk}[prop]{Remark}

\newcommand {\R} {\mathbb{R}}

\newcommand {\N} {\mathbb{N}}
\newcommand {\G} {\mathcal{G}}

\newcommand {\p} {\partial}
\newcommand {\dt} {\partial_t}

\newcommand {\supp} {\text{supp}}

\DeclareMathOperator{\di}{div}

\DeclareMathOperator {\Ree} {Re}

\numberwithin{equation}{section}

\allowdisplaybreaks

\begin{document}
\title[parabolic Signorini problem]{Optimal regularity for the variable coefficients parabolic Signorini problem}
\author{Vedansh Arya}
\address{Department of Mathematics and Statistics,
University of Jyväskylä, Finland}
\email{vedansh.v.arya@jyu.fi}
\author{Wenhui Shi}
\address{Lehrstuhl für Angewandte Analysis, RWTH Aachen, 52056 Aachen, Germany}
\email{shi@math1.rwth-aachen.de}

\begin{abstract}
In this paper we discuss the optimal regularity  of the variable coefficient parabolic Signorini problem with $W^{1,1}_p$ coefficients and $L^p$ inhomogeneity, where $p>n+2$ with $n$ being the space dimension. Relying on an parabolic Carleman estimate and an epiperimetric inequality, we show the optimal regularity of the solutions as well as the regularity of the regular free boundary. 
\end{abstract}

\maketitle

\tableofcontents

\section{Introduction}
In this paper, we study the local regularity of solutions and free boundaries to the parabolic Signorini problem with variable coefficients and nonzero inhomogeneities. To describe the problem at hand,  let $Q_1:=B_1\times (-1,0]\subset \R^{n+1}$, $n\geq 2$, be the parabolic cylinder, where $B_1\subset \R^n$ is the open unit ball centered at the origin; $Q_1^+:=Q_1\cap \{x_n>0\}$ and $Q'_1:=Q_1\cap \{x_n=0\}$. Let $u:Q_1^+\rightarrow \R$ be a solution to 
\begin{equation}\label{eq:main}
\begin{split}
\p_t u - \p_i(a^{ij}(x,t)\p_j u)= f(x,t) &\text{ in } Q_1^+,\\
u\geq 0, \quad \p_\nu^A u \geq 0,\quad u\p_\nu^A u =0 &\text{ on } Q'_1.
\end{split}
\end{equation}
Here $A=(a^{ij})_{i,j\in\{1,\cdots,n\}}$ is symmetric and uniformly elliptic; given the unit outer normal $\nu=(\nu_1,\cdots, \nu_n)$, $\p^A_{\nu} u := a^{ij}\nu_i\p_ju$ denotes the conormal derivatives of $u$, and $f:Q_1^+\rightarrow \R$ is a given inhomogeneity. We assume that  
\begin{equation}\label{eq:assumption_f}
\text{ for some }p\in (n+2,\infty],\quad f\in L^p(Q_1^+)
\end{equation}
 and furthermore, the coefficients $a^{ij}$ satisfy the following assumptions: 
\begin{itemize}
\item [(i)] (Uniform ellipticity) $a^{ij}(0,0)=\delta^{ij}$  and 
\begin{align*}
\frac{7}{8} |\xi|^2 \leq a^{ij}(x,t)\xi_i\xi_j\leq \frac{9}{8}|\xi|^2, \quad \forall \xi\in \R^n, \ \forall (x,t)\in Q_1^+\cup Q'_1;
\end{align*}
\item [(ii)](Sobolev regularity) $a^{ij}$ are in the parabolic Sobolev class $W^{1,1}_p(Q_1^+)$, i.e. $a^{ij}$, $\p_ta^{ij}$ and $\nabla_x a^{ij}\in L^p(Q_1^+)$. 
\item [(iii)] (Off-diagonal) $a^{in}(x',0,t)=0$ for $i\in \{1,\cdots, n-1\}$ and for all $(x',t)\in Q'_1$.
\end{itemize}
We make some remarks about assumptions (i)--(iii): by Sobolev embedding assumption (ii)  implies that $a^{ij}$ are H\"older continuous: 
\begin{align*}
|a^{ij}(x,t)-a^{ij}(y,s)|\leq C_n (|x-y|+|t-s|)^{1-\frac{n+1}{p}},\quad \forall (x,t), (y,s)\in Q_1^+\cup Q'_1.
\end{align*}
In particular, $a^{ij}$ is in the parabolic H\"older class $H^{\gamma,\gamma/2}(Q_1^+\cup Q'_1)$ with $\gamma=1-\frac{n+2}{p}$. Assumption (i) then can be achieved by a linear transformation and rescaling, which posts no restrictions as we only care about the local regularity properties. Assumption (iii) can be achieved by means of a change of variables, cf. Proposition \ref{prop:off_diag}. This assumption allows us to extend the problem to the full cylinder $Q_1$ by an even reflection in $x_n$.

Given suitable initial and boundary data, it is well-known that there exists a unique solution to \eqref{eq:main} in the class $\{v\in W^{1,0}_2(Q'_1): v\geq 0 \text{ a.e.  on }Q'_1\}$ in the sense of variational inequalities, cf. \cite{DL76}.
Due to the work of Arkhipova and Uraltseva \cite{AU}, in this paper we may assume that solutions to \eqref{eq:main} have the following Sobolev and pointwise regularity:
\begin{align}\label{eq:assump_u}
u\in W^{2,1}_2(Q_1^+), \quad \nabla_x u\in  H^{\alpha, \frac{\alpha}{2}}(Q_1^+\cup Q'_1)
\end{align} 
 for some $\alpha\in (0,1)$, so that the boundary conditions in \eqref{eq:main} are satisfied in the classical sense. We refer to Section \ref{subsec:notation} for the definition of the parabolic Sobolev and H\"older classes. 
In the sequel, we aim to obtain the optimal regularity of the solution $u$, as well as the regularity of the free boundary 
\begin{equation}\label{eq:fb}
\Gamma_u:=\p_{Q'_1}\{(x',0,t)\in Q'_1: u(x',0,t)>0\}.
\end{equation}
\subsection{Main results}
Our first main result states that solutions to \eqref{eq:main} has the optimal regularity:
\begin{thm}\label{thm:opt}
Let $u\in W^{1,0}_2(Q_1^+)$ be a solution to the parabolic Signorini problem \eqref{eq:main}. Assume that $a^{ij}\in W^{1,1}_p(Q_1^+)$ and $f\in L^p(Q_1^+)$ with $p>n+2$. Let $\gamma:=1-\frac{n+2}{p}$. Then there exits a constant $C=C(n,p,\|f\|_{L^p(Q_1^+)}, \|a^{ij}\|_{W^{1,1}_p(Q_1^+)},\|u\|_{L^2(Q_1^+)})$ such that the following statements hold true:
\begin{itemize}
\item [(i)] If $p\in (2(n+2),\infty]$, then 
\begin{equation*}
\|u\|_{H^{\frac{3}{2},\frac{3}{4}}(\overline{Q_{1/4}^+})}\leq C.
\end{equation*} 
\item [(ii)] If $p\in (n+2, 2(n+2)]$, then  $u$ is almost $H^{1+\gamma, \frac{1+\gamma}{2}}$-regular in the sense that
\begin{align*}
|\nabla u(x,t)-\nabla u(y,s)|&\leq C|\ln (|x-y|+\sqrt{|t-s|})|^2 (|x-y|+\sqrt{|t-s|})^\gamma,\\
|u(x,t)-u(x,s)|&\leq C(\ln |t-s|)|^2 |t-s|^{\frac{1+\gamma}{2}}
\end{align*}
for all $(x,t), (y,s)\in \overline{Q_{1/4}^+}$.
\end{itemize}
\end{thm}

\begin{rmk}\label{rmk:opti}
When $p\in (2(n+2),\infty]$, the $H^{3/2,3/4}$-regularity is optimal in space due to the explicit stationary solution $u_0(x)=\Ree(x_{n-1}+ix_n)^{\frac{3}{2}}$.  When $p\in (n+2, 2(n+2)]$, the regularity result is almost optimal because $H^{1+\gamma,\frac{1+\gamma}{2}}$ is the optimal regularity for the linear parabolic equations away from the free boundary. We also remark that under our assumptions of the coefficients and inhomogeneity, the optimal regularity of the solution in time remains unknown (this is unknown even for the heat operator with $L^\infty$ inhomogeneity). 
\end{rmk}

\begin{rmk}\label{rmk:zero_inhomo}
When $f=0$ in \eqref{eq:main}, we obtain the optimal growth estimate of the solutions around the free boundary points when $p\in (n+2, 2(n+2)]$:
\begin{align*}
\sup_{Q_r^+} |u(x,t)|\leq C r^{\frac{3}{2}},\quad \forall r\in (0,1/4).
\end{align*}
This together with the regularity for the linear problem away from the free boundary gives that $u\in H^{1+\gamma,\frac{1+\gamma}{2}}(\overline{Q_{1/4}^+})$.
\end{rmk}

The main ingredient in our proof for Theorem \ref{thm:opt} is the parabolic Carleman estimate. The Carleman approach was first developed in \cite{KRS16,KRS17} to treat the  elliptic Signorini problem with $W^{1,p}$ coefficients. In the simplest form the elliptic Carleman estimate reads, for $\tau\geq 1$,
\begin{align*}
\tau^{\frac{3}{2}}\| e^{\tau \phi(-\ln |x|)} |x|^{-1}\phi'(\phi'')^{\frac{1}{2}} u\|_{L^2} \lesssim \|e^{\tau\phi(-\ln |x|)}|x| \Delta u\|_{L^2},
\end{align*}
where $u\in C_c(B_1\setminus \{0\})$ satisfies the Signorini boundary condition on $B'_1$, and the weight function $\phi:\R_+\rightarrow \R$ is chosen to be slightly convex and asymptotically linear at infinity, see \cite{KRS16}. The convexity of the weight function makes the above estimate robust under the variable coefficient perturbation. 

In the parabolic case,  a generalized Poon's monotonicity formula was used to treat the Signorini problem for the heat operator, cf. \cite{DGPT}:
\begin{align*}
r\mapsto \frac{\int_{-r^2}^0\int_{\R^n_+} |t||\nabla u|^2 G\ dxdt }{\int_{-r^2}^0 \int_{\R^n_+} |u|^2 G\ dxdt} \quad \nearrow \text{ in } (0,1),
\end{align*}
where $u:\R^n_+\times (-1,0]$ is a global solution to the Signorini problem for the heat operator with zero inhomogeneity, and $G$ is the standard Gaussian. The Poon's monotonicity formula was first introduced by Poon to prove the unique continuation principle for the heat equation \cite{Poon96}, and was later generalized to study the nodal set properties and unique continuation principle for the fractional heat operator, cf. \cite{AT20,BG18}.   Unlike in the elliptic case, generalization of the Poon's frequency function for the variable coefficient operator $\p_t - \p_i(a^{ij}(x,t)\p_j)$ is unknown. This motivates us to explore the Carleman approach to treat the variable coefficient perturbation. 

Similar as in the elliptic case in \cite{KRS16}, our parabolic Carleman estimate is derived based on a perturbation argument to the constant coefficient case:
\begin{align*}
\tau^{\frac{1}{2}}\|e^{\tau\phi(-\ln|t|)+ \frac{|x|^2}{8t}}|t|^{-\frac{1}{2}} \phi'(\phi'')^{\frac{1}{2}} u\|_{L^2} \lesssim \|e^{\tau\phi(-|\ln|t|)+ \frac{|x|^2}{8t}}|t|^{\frac{1}{2}} (\p_t-\Delta u)\|_{L^2},
\end{align*}
where $\phi$ is the same weight function as in the elliptic case and $u$ is a global solution compactly supported in $\R^n\times [-1,-\delta^2]$.  
  However, compared to the elliptic case in \cite{KRS16}, the proof for the parabolic Carleman estimate is much more involved: firstly, due to the lack of regularity in $\p_t u$, instead of the original problem we have to work with a family of penalized problems; Secondly, we need to derive a new weighted $H^2$-estimate to deal with the perturbation, because of the Gaussian weight in the parabolic Carleman estimate.  
  
The parabolic Carleman estimate allows us to study 
qualitative properties of the vanishing order:
 \begin{align*}
 \kappa_{(x_0,t_0)}:= \limsup_{r\rightarrow 0} \frac{ \ln \left(\frac{1}{r^2} \int_{t_0-(2r)^2}^{t_0-r^2}\int_{\R^n}|u|^2 G_{(x_0,t_0)}\right)^{\frac{1}{2}}}{\ln r},
 \end{align*} 
 where $G_{(x_0,t_0)}$ is the parametrix defined in \eqref{eq:parametrix}. 
Furthermore, it gives the crucial compactness properties needed in the blow-up analysis.    
Consequently, we obtain an almost optimal regularity  of the solution (up to a logarithmic loss), which is crucial in the study of the regular set of the free boundary. The up to an arbitrary small $\epsilon$ loss almost optimal regularity may also be obtained via a compactness argument (cf. \cite{And13} and \cite{RS17} for the elliptic counterpart). However, the Carleman estimate provides more: it yields a refined growth estimate of solutions around the free boundary points. This allows us to study the singular set of the free boundary, which will be carried out in the forthcoming paper.  

\medskip

To improve the almost optimal regularity to the optimal regularity,  we adopt the parabolic epiperimetric inequality approach, cf. \cite{Shi20}. This approach is based on a compactness argument, which is robust under the perturbation. We obtain the parabolic H\"older regularity of the regular free boundary at the same time. This is our second main result:

\begin{thm}\label{thm:reg_fb}
Let $u\in W^{1,0}_2(Q_1^+)$ be a solution to the parabolic Signorini problem \eqref{eq:main}. Assume that $a^{ij}\in W^{1,1}_p(Q_1^+)$ and $f\in L^p(Q_1^+)$ with $p>2(n+2)$. Let $\mathcal{R}_u:=\{(x_0,t_0)\in \Gamma_u\cap Q_{1}:\kappa_{(x_0,t_0)}=\frac{3}{2}\}$ be the regular free boundary. Then $\mathcal{R}_u$ is a relatively open (possibly empty) subset of $\Gamma_u$. Moreover, given $(x_0,t_0)\in \mathcal{R}_u$, there is $\delta>0$ small such that up to a rotation of coordinates,
\begin{align*}
\mathcal{R}_u\cap Q_{\delta}(x_0,t_0)=\{(x'',x_{n-1},0,t)\in Q_\delta(x_0,t_0): x_{n-1}=g(x'', t)\}
\end{align*}
for some function $g\in H^{1,1/2}$ and $\nabla_{x''}g\in H^{\beta, \beta/2}$ for some $\beta\in (0,1)$.
\end{thm}
We remark that if $f=0$ in \eqref{eq:main}, we are able to relax the restriction on $p$ to $p>n+2$. The regularity of $g$ in $t$ is not necessarily optimal. This is closely related to the regularity of the solution in $t$.

The epiperimetric inequality method goes  back to the work of Weiss \cite{W99} in the context of the classical obstacle problem, and \cite{GPG,FS16,CSV20,CSV202} for the elliptic Signorini problem. The approach crucially relies on the minimality (or almost minimality) of the solution to the so-called Weiss energy, thus is usually applied to the stationary problems. It has also been used to study the parabolic problems once one shows the boundedness of the time derivative of the solution (such that the equation on each time slice can be viewed as an elliptic equation with bounded inhomogeneity), cf. \cite{BDGP20, BDGP21}. 
The main consequence of the epiperimetric inequality is the decay rate (in $r$) of the Weiss energy applied to the blow-up sequence $u_r$. This in turn gives the convergence rate of the  distance of $u_r$ to its blow-up limit as $r\rightarrow 0$.

In the parabolic case, without relying on the minimiality of the solution, one can still derive a decay estimate of the Weiss energy around regular free boundary points, based on the observation that the problem can be viewed (formally) as a constrained $L^2$-gradient flow of the Weiss energy, and a compactness argument, cf. \cite{Shi20}. The argument is robust under the perturbation, and allows us to treat the Signorini problem with variable coefficients and $L^p$ inhomogeneities.  As a consequence we obtain the uniqueness of the blow-up limit at the regular points as well as the following asymptotic expansion of the solution: 
\begin{align*}
\|u-c_{(x_0,t_0)}\Ree(x'\cdot \nu_{(x_0,t_0)}+ix_n)^{\frac{3}{2}}\|_{H^{1,0}(Q_{r}^+(x_0,t_0))}\leq C r^{\frac{3}{2}+\beta}
\end{align*} 
for some $c_{(x_0,t_0)}>0$ and $\nu_{(x_0,t_0)}\in \mathbb{S}^{n-1}\cap \{x_n=0\}$. This in turn yields the H\"older continuity of the map $(x_0,t_0)\mapsto \nu_{(x_0,t_0)}$, hence the desired regularity of $g$. 


\subsection{Background and known results}
The parabolic Signorini problem arises naturally in modelling fluids passing through a semi-permeable membrane, cf. \cite{DL76}. More precisely, let $\Omega$ be a region occupied by a porous medium and let $u$ be the pressure of a viscous fluid which is only slightly compressible. Let $\mathcal{M}$ be a semi-permeable membrane with zero thickness and $h$ be a given a fluid pressure on $\mathcal{M}$. Then it can be shown that $u$ solves 
\begin{align*}
\p_t u -\Delta u = 0 \text{ in } \Omega\times (0,T]
\end{align*}
and satisfies the Signorini type boundary conditions on $\mathcal{M}_T:=\mathcal{M}\times (0,T]$:
\begin{align*}
u> h &\Rightarrow \p_{\nu } u =0,\\
u=h &\Rightarrow \p_{\nu} u \geq 0,
\end{align*}
where $\nu$ is the outer normal at $\mathcal{M}_T$. Given suitable initial and boundary data, the well-posedness of this problem in the sense of variational inequalities haven been shown in \cite{DL76}. 
To study the pointwise regularity of the solutions, one applies local coordinate transformation to straighten $\mathcal{M}$ and consider $u-h$ instead of $u$. Then after a rescaling and translation, one is naturally led to the variable coefficient parabolic Signorini problem in \eqref{eq:main}. Note that \eqref{eq:main} also includes the case when a time-dependent source term is present and that the semi-permeable membrane $\mathcal{M}$ is moving in time. 

Concerning the pointwise H\"older regularity results, Arkhipova and Uraltseva showed that spacial derivatives of solutions to \eqref{eq:main} are in the parabolic Hölder class $H^{\alpha, \frac{\alpha}{2}}_{loc}(Q_1^+\cup Q'_1)$ for some $\alpha\in (0,1)$, cf. \cite{AU88, AU}. Recently, similar regularity properties of  solutions were shown for the fully nonlinear parabolic Signorini problem, cf. \cite{C22} and \cite{HT22}, and for almost minimizers with parabolic H\"older continuous coefficients $a^{ij}$, cf. \cite{JP23}.   

When $a^{ij}=\delta^{ij}$, we have the following results concerning the optimal regularity of the solution and regularity of the free boundary: When $f\in L^\infty$, the optimal $H^{3/2,3/4}$ regularity of the solutions was proved by Danielli, Garofalo, Petrosyan and To in \cite{DGPT} using Almgren's frequency function approach, and also by Athanasopoulos, Caffarelli and Milakis \cite{ACM18, ACM19} by showing quasi-convexity properties of solutions under some global assumptions of the initial data. When $f\in H^{1,1/2}$ or better, the Hölder regularity properties of the regular free boundary were studied in \cite{DGPT, PS14, ACM19, PZ19}. When $f$ is smooth, the regular free boundary is smooth in space and time \cite{BSZ17}.
 The structure of the singular free boundary was studied in \cite{DGPT}. It is known by an extension technique that the parabolic Signorini problem is closely related to the obstacle problem for the fractional heat operator $(\p_t-\Delta)^{1/2}$. The above well-posedness and regularity results haven been generalized to the obstacle problem for the fractional heat operator $(\p_t-\Delta)^s$, $s\in (0,1)$, in \cite{BDGP20, BDGP21}. 

There are not many results on the optimal regularity  for variable coefficients parabolic Signorini problem with rough coefficients or rougher than $L^\infty$ source terms. In a recent paper Torres-Latorre shows when $a^{ij}=\delta^{ij}$ and $f\in W^{1,p}$ in space and time for some $p>n+2$, the regular free boundary is $C^{1,\alpha}_{x,t}$ for some $\alpha\in (0,1)$, under the assumption that the regular free boundary is a  parabolic Lipschitz graph with sufficiently small Lipschitz constant \cite{TL23}. The proof is based on establishing a boundary Harnack inequality with nonzero right hand side. Our results generalize the known results in the sense that we only require $f\in L^p$ and $a^{ij}\in W^{1,1}_p$ ($p>2(n+2)$) to obtain the optimal regularity of the solution. Furthermore, we get the H\"older regularity of the regular free boundary without assuming any Lipschitz regularity. The techniques we use in the proof are quite flexible when dealing with perturbations. They can be adapted to study the variable coefficients parabolic thin obstacle problem, where the obstacle is sitting in a hypersurface which divides the domain. The regularity assumption of $\p_t a^{ij}$ can be relaxed, if $a^{ij}$ satisfies the off-diagonal condition (iii). For  presentation simplicity, in this paper we do not pursue these variations.


\subsection{Plan of the paper}
The rest of the paper is planned as follows: In section \ref{sec:pre} we introduce the solution class and notations used in the paper. In section \ref{sec:Carleman} we show a parabolic Carleman estimate for global solutions and derive a three sphere's inequality from it. In Section \ref{sec:al_opt} we introduce the notion of the  vanishing order and  study its properties using the three sphere's inequality. We also obtain an almost optimal growth estimate for solutions. In Section \ref{sec:opt} we introduce the Weiss energy and show a discrete decay estimate of the Weiss energy at the regular free boundary points. This gives the optimal regularity of the solution stated in Theorem \ref{thm:opt}. The regularity of the regular free boundary (Theorem \ref{thm:reg_fb}) is shown in Section \ref{sec:fb_reg}, cf. Theorem \ref{thm:opt_fb2}. 

\section{Preliminaries}\label{sec:pre}
 \subsection{Notations}\label{subsec:notation}
We will use the following notations: $\R^n_+:= \{x=(x_1,\cdots, x_n)\in \R^n: x_n> 0\}$; For $r>0$, we define 
\begin{align*}
B_r(x_0)&:=\{x \in \R^n : \; |x-x_0|<r\},\quad (\text{Euclidean balls})\\
 B_r^+(x_0)&:=B_r(x_0)\cap \R^n_+, \\
  B'_r(x_0)&:=B_r(x_0)\cap \{x_n=0\},\\
  Q_r(x_0,t_0)&:=B_r(x_0)\times (t_0-r^2,t_0],\quad (\text{Parabolic cylinders})\\
  Q_r^+(x_0,t_0)&:=Q_r(x_0,t_0)\cap \{x_n> 0\},\\
  Q'_r(x_0,t_0)&:=Q_r(x_0,t_0)\cap \{x_n=0\},\\
  S_r(t_0)&:=\R^n\times (t_0-r^2,t_0],\quad (\text{parabolic strips})\\
  S_r^+(t_0)&:=\R^n_+\times (t_0-r^2,t_0],\\
  S'_r(t_0)&:=S_r(t_0)\cap \{x_n=0\},\\
  A_{r_1,r_2}(t_0)&:=S_{r_2}(t_0)\setminus S_{r_1}(t_0),\quad (\text{parabolic annuli})\\
  A_{r_1,r_2}^+(t_0)&:=S_{r_2}^+(t_0)\setminus S_{r_1}^+(t_0),\\
  A'_{r_1,r_2}(t_0)&:=S'_{r_2}(t_0)\setminus S'_{r_1}(t_0),\quad \text{ for } 0<r_1<r_2.
\end{align*}
If $x_0$ is the origin (or $(x_0,t_0)=(0,0)$ or $t_0=0$) we drop the dependence in the notations on $x_0$ (or $(x_0,t_0)$ or $t_0$). 
We use $\nabla u=(\p_1u,\cdots, \p_n u)$ to denote the spacial gradient of $u$.

We recall the definition of the parabolic H\"older spaces and Sobolev spaces, cf. \cite{DGPT}. Let $\Omega\subset \R^n$ be bounded and open and $\Omega_T:=\Omega\times (0,T]$. Given $\ell>0$, $\ell=m+\gamma$ for $m\in \N$ and $\gamma\in (0,1]$, we say $u\in H^{\ell,\ell/2}(Q_r)$ if the following norm is finite:
\begin{align*}
\|u\|_{H^{\ell,\ell/2}(\Omega_T)}&:= \sum_{k=0}^{m}\sum_{|\alpha|+2j=k}\|\p_x^\alpha\p_t^ju\|_{L^\infty(\Omega_T)}\\
&+\sum_{|\alpha|+2j=m} \langle \p_x^\alpha\p_t^j u \rangle^{(\gamma)}_{\Omega_T} + \sum_{|\alpha|+2j=m-1,\ m\geq 1}\langle \p_x^\alpha\p_t^j u \rangle^{(\frac{1+\gamma}{2})}_{t,\Omega_T},
\end{align*}
 where for $\beta\in (0,1)$,
 \begin{align*}
 \langle u\rangle^{(\beta)}_{\Omega_T}&:=\sup_{(x,t), (y,s)\in \Omega_T}\frac{|u(x,t)-u(y,s)|}{(|x-y|+\sqrt{|t-s})^\beta},\\
 \langle u \rangle ^{(\beta)}_{t,\Omega_T}&:=\sup_{(x,t), (x,s)\in \Omega_T}\frac{|u(x,t)-u(x,s)|}{|t-s|^{\beta}}.
 \end{align*}
 In this paper, we mainly use the following H\"older spaces: for $\beta\in (0,1)$, then 
 \begin{align*}
 \|u\|_{H^{\beta,\beta/2}(\Omega_T)}&=\|u\|_{L^\infty(\Omega_T)}+\langle u \rangle^{(\beta)}_{\Omega_T},\\
 \|u\|_{H^{1+\beta, (1+\beta)/2}(\Omega_T)}&=\|u\|_{L^\infty(\Omega_T)}+\|\p_x u\|_{L^\infty(\Omega_T)}+\langle \p_x u \rangle^{(\beta)}_{\Omega_T}+ \langle u\rangle_{t,\Omega_T}^{(\frac{1+\beta}{2})}.
 \end{align*}
 The parabolic Sobolev spaces $W^{2m,m}_p(\Omega_T)$, $m\in \N$, consists of those functions $u$ such that $\|u\|_{W^{2m,m}_{p}(\Omega_T)}<\infty$, where
 \begin{align*}
 \|u\|_{W^{2m,m}_{p}(\Omega_T)}:=\sum_{|\alpha|+2j\leq m}\|\p^\alpha_x\p_t^{2j}u\|_{L^p(\Omega_T)}.
 \end{align*}
 We will denote by $W^{1,1}_p(\Omega_T)$ by functions with finite $\|\cdot \|_{W^{1,1}_p(\Omega_T)}$ norm, where
 \begin{align*}
 \|u\|_{W^{1,1}_p(\Omega_T)}:=\|u\|_{L^p(\Omega_T)}+\|\nabla u\|_{L^p(\Omega_T)}+\|\p_t u \|_{L^p(\Omega_T)}.
 \end{align*}
 
In the paper we write $A\lesssim B$ ($A\gtrsim B$) if $A\leq CB$ ($A\geq CB$) for some absolute constant $C$ or a constant $C$ depending only on the dimension.

 \subsection{Classes of solutions}
The techniques used in the paper require working with global solutions. For that we use cut-offs to extend \eqref{eq:main} into a global problem. 
More precisely, let  $\tilde{u}:=u\psi$, where $\psi=\psi(|x|)$ is a smooth cut-off function such that 
\begin{align*}
0\leq \psi\leq 1, \quad \supp(\psi)\subset B_{2/3},\quad \psi=1 \text{ in } B_{1/2}.
\end{align*}
 Let $\tilde{a}^{ij}:=\eta a^{ij}+(1-\eta)\delta^{ij}$, where $\eta=\eta(|x|)$ is a smooth cut-off function such that
 \begin{align*}
0\leq \eta\leq 1, \quad \supp(\eta)\subset B_1,\quad \eta=1 \text{ in } B_{2/3}.
\end{align*}
 Note that $\tilde{u}=u$ in $\overline{B_{1/2}^+}\times (-1,0]$, $\supp (\tilde{u})\subset Q_1$ and 
 the supports of $\psi$ and $\eta$ are chosen such that $\supp(\nabla\eta)\cap \supp (\psi)=\emptyset$ and $\eta=1$ on $\supp(\psi)$. It then follows from a direct computation that $\tilde{u}:S_1^+\rightarrow \R$ solves the global Signorini  problem
\begin{equation*}
\begin{split}
\p_t\tilde{u}-\p_i(\tilde{a}^{ij}\p_j\tilde{u})=\tilde{f}\quad &\text{ in } S_1^+,\\
\tilde{u}\geq 0,\ \p^{\tilde{A}}_{\nu}\tilde{u}\geq 0,\ \tilde{u}\p_\nu^{\tilde{A}}\tilde{u}=0 \quad &\text{ on } S'_1,
\end{split}
\end{equation*}
where $\tilde{A}=(\tilde{a}^{ij})_{ij}$ and 
\begin{align*}
\tilde{f}:=f\psi- \p_i\psi a^{ij}\p_ju -\p_i(a^{ij}u)\p_j\psi -a^{ij} u \p_{ij}\psi.
\end{align*}
Here we have used that support conditions for the cut-offs. Note that $\supp(\tilde{f})\subset \overline{Q_1^+}$, and moreover $\tilde{f}\in L^p(S_1^+)$, since $u, \nabla u\in L^\infty(Q_1^+)$. Moreover, it is not hard to see that $\tilde{a}^{ij}$ satisfies (i)--(iii) in the whole strip $S_1^+$.

\begin{rmk}
When the inhomogeneity $f=0$ in the local problem \eqref{eq:main}, we still have in general $\tilde{f}\neq 0$  after the extension. But in this case, $\tilde{f}$ satisfies the support condition 
 \begin{equation}\label{eq:support}
 \supp(\tilde{f})\subset (B_{2/3}\setminus B_{1/2})\times (-1,0].
 \end{equation}
\end{rmk}

\begin{defi}[Global solutions] 
\label{defi:global}
Let $p\in (n+2,\infty]$. Let $f\in L^p(S_1^+)$, $A=(a^{ij})\in W^{1,1}_p(S^+_1; \mathcal{M}^{n\times n})$, $\supp(\nabla a^{ij})\subset \overline{Q_1^+}$ and $A$ satisfies the assumptions (i)--(iii) in $S_1^+$. We say that $u\in \mathcal{G}^{A,f}_p(S_1^+)$, if  $\supp (u)\subset \overline{Q_1^+}$, $u\in W^{2,1}_2(S_1^+)\cap L^\infty(S_1^+)$, $\nabla u \in H^{\alpha,\frac{\alpha}{2}}(S_1^+\cup S'_1)$ for some $\alpha\in (0,1)$ and $u$ solves
\begin{equation}\label{eq:global_sol}
\begin{split}
\p_tu-\p_i(a^{ij}\p_ju)=f\quad &\text{ in } S_1^+,\\
u\geq 0,\ \p^{A}_{\nu}u\geq 0,\ u\p_\nu^{A}u=0 \quad &\text{ on } S'_1.
\end{split}
\end{equation}
In Section \ref{sec:al_opt} we will also use the solution class $\mathcal{G}^{A,f}_p(S_1)$. We say $u\in \mathcal{G}_p^{A,f}(S_1)$ if $u, A, f$ are even about $x_n$ and $u\big|_{S_1^+}\in \mathcal{G}^{A,f}_p(S_1^+)$.
\end{defi}


In the paper we work with solutions defined on  lower half cylinders $Q_r(x_0,t_0)=B_r(x_0)\times (t_0-r^2, t_0]$, which do not contain any information of the solutions beyond $t_0$. Thus the topological free boundary $\Gamma_u$, which is defined in \eqref{eq:fb}, may not be preserved when we restrict our solutions to sub-cylinders, i.e. $(x_0,t_0)\in \Gamma_u$ is a free boundary point for $u:Q_1^+\rightarrow \R$, but $(x_0,t_0)$ may not be a free boundary point for $u\big|_{Q_{r}(x_0,t_0)}$ (see also the comment below \cite[Definition 4.1]{DGPT}). Because of this, similar as in \cite{DGPT}, we will work with the \emph{extended free boundary}, which resolves the restriction issue. Note that by doing so, we may include more points with higher vanishing order in the interior of the contact set (cf. \cite[Proposition 10.8]{DGPT}).

\begin{defi}[Extended free boundary]
\label{defi:ext_fb}
Let $u\in \mathcal{G}^{A,f}(S_1^+)$. We define the extended free boundary to be
\begin{align*}
\Gamma_u^\ast:=\p_{Q'_1}\{(x',0,t)\in Q'_1: u(x',0,t)=0,\quad \p_n u(x',0,t)=0\}.
\end{align*}
\end{defi}
It is immediate from the H\"older continuity of $\nabla u$ that $\Gamma_u\subset \Gamma_u^\ast$. On the other hand, it is easy to construct a solution whose extended free boundary is not empty while the classical free boundary is empty, cf. \cite[Remark 10.9]{DGPT}.

 \section{Parabolic Carleman estimate}\label{sec:Carleman}
In this section we derive a parabolic Carleman estimate for a global solution $u\in \mathcal{G}^{A,f}_p(S_1^+)$, where $\supp(u)\subset \overline{B_1^+}\times [-\rho^2, -r^2]$, $0<r<\frac{\rho}{2}<\rho<\frac{1}{2}$. The main result of the section states as follows:

\begin{prop}\label{prop:Carleman}
Let $u\in \mathcal{G}^{A,f}_p(S_1^+)$ with $\supp(u)\subset \overline{B_1^+}\times [-\rho^2, -r^2]$, $0<r<\frac{\rho}{2}<\rho<\frac{1}{2}$. Let
\begin{align*}
\delta:=\sup_{\tilde{r}\in [r,\frac{1}{2}]} \|\p_ta^{ij}\|_{L^{\frac{n+2}{2}}(A_{\tilde r, 2\tilde r})}+\|\nabla a^{ij}\|_{L^{n+2}(A_{\tilde{r}, 2\tilde{r}}^+)}+\|a^{ij}-\delta^{ij}\|_{L^\infty(A_{\tilde{r}, 2\tilde{r}}^+)}. 
\end{align*} 
Let $\phi:(-1,0)\rightarrow \R$ be the following weight function:
\begin{align*}
\phi(t):=\tilde\phi(|\ln(-t)|),\quad \tilde\phi(s):=s+c_0(s\arctan s- \frac{1}{2}\ln(1+s^2)).
\end{align*}
where $c_0\in (0,1/4)$ is a small fixed constant. 
Then there are constant $\delta_0=\delta_0(n)>0$ and $C=C(n,\|a^{ij}\|_{W^{1,1}_p})>0$ such that if $\delta\leq\delta_0$, then for any $\tau>1$ and for $\tilde\gamma:=\frac{1}{n+2}-\frac{1}{p}$ we have
\begin{align*}
&\quad \tau^{\frac{1}{2}} \left\| e^{\tau \phi(t)+\frac{|x|^2}{8t}} (-t)^{-\frac{1}{2}}(1+|\ln(-t)|^2)^{-\frac{1}{2}} u\right\|_{L^2(S_1^+)}+ \left\|e^{\tau \phi(t)+\frac{|x|^2}{8t}} (1+|\ln(-t)|^2)^{-\frac{1}{2}}\nabla u\right\|_{L^2(S_1^+)}\\
&\leq Cc_0^{-1}\left(\tau \left\|e^{\tau \phi(t)+\frac{|x|^2}{8t}}(-t)^{-\frac{1}{2}+\tilde\gamma}u\right\|_{L^2(S_1^+)}+\left\|e^{\tau \phi(t)+\frac{|x|^2}{8t}}(-t)^{\frac{1}{2}}f\right\|_{L^2(S_1^+)}\right).
\end{align*}
\end{prop}

\begin{rmk}[Smallness of $\delta$]\label{rmk:smallness}
We note that for $a^{ij}\in W^{1,1}_p(S_1^+)$, $p>n+2$, with bounded support, an application of H\"older's inequality and Sobolev embedding yields
\begin{align*}
\|\nabla a^{ij}\|_{L^{n+2}(A_{\tilde{r}, 2\tilde{r}})}&\leq \|\nabla a^{ij}\|_{L^p(A_{\tilde{r}, 2\tilde{r}})} \tilde r^{2(\frac{1}{n+2}-\frac{1}{p})},\\
\|a^{ij}-\delta^{ij}\|_{L^\infty(A_{\tilde{r}, 2\tilde{r}}^+)}&\leq C \tilde{r}^{1-\frac{n+2}{p}},\\
\|\p_ta^{ij}\|_{L^{\frac{n+2}{2}}(A_{\tilde{r}, 2\tilde{r}}^+)}&\leq \|\p_t a^{ij}\|_{L^p(A_{\tilde{r}, 2\tilde{r}}^+)} \tilde{r}^{2(\frac{1}{n+2}-\frac{1}{2p})}.
\end{align*}
Thus the smallness assumption on $\delta$ in Proposition \ref{prop:Carleman} dose not post any issue in our application. 
\end{rmk}

Compared to the elliptic Carleman estimate in \cite{KRS16}, one additional complication in the parabolic case is that $\p_t u$ is only in $L^2(S_1^+)$. Thus the boundary contribution $\int_{S'_1}\p_tu \p_n u $, which arises in the computation of the commutator in the proof for the Carleman estimate, is not well-defined.   To overcome this technical difficulty, we will derive Carleman estimate for a family of approximating problems and then pass to the limit. More precisely, given $u\in \mathcal{G}^{A,f}_p(S_1^+)$ with $\supp(u)\subset \overline{B_1^+}\times [-\rho^2,-r^2]$. For $R\geq 3$ and $\epsilon\in (0,1/4)$, let $u^\epsilon: B_R^+\times [-1,0]\rightarrow \R$ be a solution to the  penalized problem
\begin{equation}\label{eq:approx}
\begin{split}
\p_t  u^\epsilon -\p_i (a^{ij}_\epsilon\p_j  u^\epsilon)&=f^\epsilon \text{ in } B_R^+\times (-1,0],\\
a^{nn}_\epsilon\p_n  u^\epsilon &= \beta_\epsilon( u^\epsilon) \text{ on } B'_R\times (-1,0],\\
 u^\epsilon &= 0 \text{ on } (\p B_R)^+\times (-1,0],\\
 u^\epsilon(\cdot, -1)&=0 \text{ on } B_R^+\times \{-1\},
 \end{split}
\end{equation}
where $\beta_\epsilon:\R\rightarrow \R_-$  is a smooth function with
\begin{align*}
\beta_\epsilon\leq 0, \quad \beta'_\epsilon\geq 0, \quad \beta_\epsilon(s) = \epsilon+\frac{s}{\epsilon} \text{ for } s\leq -2\epsilon^2, \quad \beta_\epsilon(s)=0 \text{ for } s\geq 0,
\end{align*}
and $a^{ij}_\epsilon$, $f^\epsilon$ are smooth approximations of $a^{ij}$ and $f$, respectively, which are chosen such that $\|a^{ij}_\epsilon\|_{W^{1,1}_p(S_1^+)}$ and $\|f^\epsilon\|_{L^p(S_1^+)}$ have uniform upper bound, $a^{ij}_\epsilon$ satisfies the off-diagonal condition and $\supp(\nabla a^{ij}_\epsilon)$, $f^\epsilon$ are essentially supported in $\overline{Q_1^+}$.
It follows from \cite{AU88} that there is a unique solution $u^\epsilon$ to the penalized problem \eqref{eq:approx}, and $u^\epsilon$ converges weakly to $u$ in $W^{1,1}_2(B_R^+\times (-1,0])$ as $\epsilon \to 0.$ The solutions $u^\epsilon$ are smooth and satisfy the global energy estimate
\begin{align}\label{eq:global_energy}
\sup_{t\in(-1,0]}\left(\|u^\epsilon(\cdot, t)\|_{L^2(B_R^+)} + \|\nabla u^\epsilon(\cdot, t)\|_{L^2(B_R^+)}\right) + \|\p_tu^\epsilon\|_{L^2(B_R^+\times (-1,0])} \lesssim \|f\|_{L^2(Q_1^+)}.
\end{align}
Furthermore, $D^2u^\epsilon$ are locally uniformly (in $\epsilon$ and $R$) bounded in $L^2(B_R^+\times (-1,0])$.

We will also need the following weighted $H^2$-estimate for solutions to the penalized problems:

\begin{lem}\label{lem:app_H2}
Let $u^\epsilon$ be the solution to the penalized problem \eqref{eq:approx}. 
 Let $\delta$ be defined in Proposition \ref{prop:Carleman}. For $m\in \N$, we consider the dyadic time-slabs 
\begin{equation}\label{eq:dyadic}
A_m:=\R^n_+\times [-2^{-2m}, -2^{-(2m+2)}], \quad N_m:=A_{m-1}\cup A_m.
\end{equation}
Let $\omega:=\omega(x,t)=1+\frac{|x|}{\sqrt{|t|}}$. 
Then there is $\delta_0=\delta_0(n)>0$, such that if $\delta \leq \delta_0$, then for all $m\in [1,\lfloor \frac{|\ln r|}{\ln 2}\rfloor +1]$, $s\in [0,2]$ and $\tau\geq 1$ we have
\begin{align*}
&\quad \left\|e^{\tau\phi+\frac{|x|^2}{8t}}\omega^s D^2 u^\epsilon \tilde\zeta\right\|_{L^2(A_m)}+ 2^m \left\|e^{\tau\phi +\frac{|x|^2}{8t}}\omega^s  \nabla u^\epsilon \tilde\zeta\right\|_{L^2(A_m)}\\
&\quad +\sup_{t\in [-2^{-2m}, -2^{-(2m+2)}]} \left\|e^{\tau\phi(t) +\frac{|x|^2}{8t}}\omega^s\nabla u^\epsilon(\cdot, t)\tilde\zeta\right\|_{L^2(\R^n_+)}\\
&\lesssim \tau 2^{2m}\left\|e^{\tau\phi + \frac{|x|^2}{8t}}\omega^s u^\epsilon\tilde\zeta\right\|_{L^2(N_m)}+ \left\|e^{\tau \phi + \frac{|x|^2}{8t}}\omega^s f^\epsilon \tilde\zeta\right\|_{L^2(N_m)}+ E(\epsilon,R,m,\tau),
\end{align*}
where the error term $E(\epsilon,R,m,\tau)\rightarrow 0$ if $\epsilon\rightarrow 0$ and then $R\rightarrow \infty$, $\tilde{\zeta}(x,t)=\zeta(\frac{|x|}{\sqrt{|t|}})$ with
$\zeta:[0,\infty)\rightarrow [0,1]$ being a smooth cut-off function, such that 
\begin{equation}\label{eq:cut-off}
\zeta'\leq 0,\quad \zeta(s)=1 \text{ if } |s|\leq R-1, \quad \zeta(s)=0 \text{ if } |s|\geq R.
\end{equation}
\end{lem}
\begin{proof}
Let $v^\epsilon:=e^{\tau\phi(t)}u^\epsilon$. Then $v^\epsilon$ solves
\begin{align*}
\p_t v^\epsilon -\p_i(a^{ij}_\epsilon \p_j v^\epsilon )&= f^\epsilon + \tau \phi'(t) v^\epsilon \text{ in } B_R^+\times (-1, 0],\\
a^{nn}_\epsilon \p_n v^\epsilon &= e^{\tau \phi(t)} \beta_\epsilon(u^\epsilon) \text{ on } B'_R \times (-1,0].
\end{align*}
Then the proof for the lemma follows along the lines of that for Lemma \ref{lem:H2} (without passing to the limit $\epsilon\rightarrow 0$ and $R\rightarrow\infty$), cf. \eqref{eq:H2-eps}, and we do not repeat it here. 
\end{proof}


\begin{proof}[Proof for Proposition \ref{prop:Carleman}]
Let $u^\epsilon$ be the solution to the penalized problem \eqref{eq:approx}.  Below we will derive a Carleman estimate for $u^\epsilon$ and pass to the limit $\epsilon\rightarrow 0$. For that we extend $u^\epsilon$ to the whole space $\R^n_+\times [-1,0]$ by considering 
\begin{equation}\label{eq:cutoff}
\tilde{u}^\epsilon (x,t):=u^\epsilon(x,t)\zeta(\frac{|x|}{\sqrt{|t|}}),
\end{equation} 
where $\zeta:[0,\infty)\rightarrow [0,1]$ is the smooth cut-off function which satisfies \eqref{eq:cut-off}. The proof for the Carleman estimate is divided into three steps:

\medskip

\emph{Step 1: Carleman inequality for the heat operator. } Let $\tilde u^\epsilon$ be defined in \eqref{eq:cutoff}. Let 
\begin{equation*}
\psi:=\psi(x,t;\tau):=\tau \phi(t)+\frac{|x|^2}{8t}.
\end{equation*}
 We will show that 
\begin{equation}\label{eq:Carleman_heat}
\begin{split}
&\quad \tau \left\| e^{\psi} (-t)^{-\frac{1}{2}}(1+|\ln(-t)|^2)^{-\frac{1}{2}} \tilde u^\epsilon\right\|_{L^2(A_{r,\rho}^+)}^2+ \left\|e^{\psi}  (1+|\ln(-t)|^2)^{-\frac{1}{2}}\nabla \tilde u^\epsilon\right\|_{L^2(A_{r,\rho}^+)}^2\\
&\lesssim c_0^{-1}\left\|e^{\psi}(-t)^{\frac{1}{2}} (\p_t-\Delta)\tilde u^\epsilon\right\|_{L^2(A_{r,\rho}^+)}^2+ E,
\end{split}
\end{equation}
where the error term $E:=E(\epsilon,R,\tau,r,\rho)\rightarrow 0$ if we first let $\epsilon \rightarrow 0$ and then $R\rightarrow \infty$.
The proof for this step is inspired by \cite{KT09}. 

We first write the operator in the new coordinates 
\begin{align}\label{eq:cov}
(y,s)\in \R^n_+\times [0,\infty);\qquad s=-\ln (-t), \quad y=\frac{x}{\sqrt{-t}}.
\end{align}
In the new coordinates, the heat operator reads
\begin{align*}
-t(\p_t-\Delta )= \p_s+\frac{y}{2}\cdot \nabla_y - \Delta_y.
\end{align*}
Conjugating the operator by $(-t)^{\frac{n}{4}}e^{\frac{|x|^2}{8t}}=e^{-\frac{ns}{4}}e^{-\frac{|y|^2}{8}}$ (which corresponds to setting $\tilde{u}^\epsilon=(-t)^{-\frac{n}{4}}e^{-\frac{|x|^2}{8t}}\hat u^\epsilon$), we obtain
\begin{align*}
(-t)^{\frac{n}{4}+1}e^{\frac{|x|^2}{8t}}(\p_t-\Delta)(-t)^{-\frac{n}{4}}e^{-\frac{|x|^2}{8t}}=\p_s +H, \quad H:=- \Delta_y+\frac{|y|^2}{16}.
\end{align*}
In the sequel, we will prove the Carleman estimate for the operator $\p_s+H$. We conjugate the operator $\p_s+H$ by $e^{\tau\tilde{\phi}(s)}$ and define $w^\epsilon=e^{\tau\tilde{\phi}}\hat u^\epsilon$. This leads to
\begin{align*}
e^{\tau\tilde\phi(s)}(\p_s+H)e^{-\tau\tilde{\phi}(s)}=\p_s+H-\tau\tilde{\phi}'=:A+S,
\end{align*}
where 
\begin{align*}
A:=\p_s ,\quad S:=H-\tau\tilde{\phi}'
\end{align*}
are the anti-symmetric and symmetric part of the operator (up to boundary contributions), respectively.  Also, since $$\tilde{u}^\epsilon(x,t)=e^{-\tau\tilde{\phi}(-\ln (-t))}(-t)^{-\frac{n}{4}}e^{-\frac{|x|^2}{8t}}w^\epsilon \left(\frac{x}{\sqrt{-t}},-\ln (-t)\right), $$
we have 
\begin{align*}
\nabla \tilde{u}^\epsilon = e^{-\psi}(-t)^{-\frac{n}{4}}\left[-\frac{x}{4t}w^\epsilon \left(\frac{x}{\sqrt{-t}},-\ln (-t)\right)+\frac{1}{\sqrt{-t}}\nabla_y w^{\epsilon}\left(\frac{x}{\sqrt{-t}},-\ln (-t)\right)\right].
\end{align*}
Now using the change of variables \eqref{eq:cov} (with $dsdy = (-t)^{\frac{n}{2}+1}dt dx$), we find 
\begin{align*}
\left\|e^{\psi}  (1+|\ln(-t)|^2)^{-\frac{1}{2}}\nabla \tilde u^\epsilon\right\|_{L^2}^2 \le 2\| (1+s^2)^{-\frac{1}{2}} |y| w^\epsilon\|_{L^2}^2 +2 \| (1+s^2)^{-\frac{1}{2}} \nabla_y w^\epsilon\|_{L^2}^2.
\end{align*}
Thus, by combining the above estimates, we find that, in order to prove the desired Carleman inequality \eqref{eq:Carleman_heat}, it suffices to show that
\begin{align*}
\tau \| (1+ s^2)^{-\frac{1}{2}}w^\epsilon\|_{L^2}^2 +\| (1+s^2)^{-\frac{1}{2}} |y| w^\epsilon\|_{L^2}^2+ \| (1+s^2)^{-\frac{1}{2}} \nabla_y w^\epsilon\|_{L^2}^2 \lesssim c_0^{-1}\|(A+S)w^\epsilon\|_{L^2}^2.
\end{align*}
Here and in the rest of \emph{step 1}, the $L^2$ norm is taken over $\R^n_+\times [s_2,s_1]$, where $s_2:=2|\ln\rho|$ and $s_1:=2(|\ln r|+ \ln 2)$. In order to get the desired Carleman estimate we compute $\|(A+S)w^\epsilon\|_{L^2}^2$.  Using $(a+b)^2=a^2+b^2+2ab$ followed by integration by parts in $y$-variable and then in $s$-variable, we find
\begin{equation}\label{eq:L2}
\begin{split}
\|(A+S)w^\epsilon\|_{L^2}^2&=\|Aw^\epsilon\|_{L^2}^2 + \|Sw^\epsilon\|_{L^2}^2 + 2\int (\p_sw^{\epsilon})(Hw^{\epsilon}-\tau \tilde{\phi}'w^{\epsilon})\\
&=\|Aw^\epsilon\|_{L^2}^2 + \|Sw^\epsilon\|_{L^2}^2 + 2\int \nabla_y\p_sw^{\epsilon} \cdot \nabla_y w^{\epsilon} \\
&+2\int (\frac{|y|^2}{16}-\tau\tilde{\phi}')(\p_sw^{\epsilon})w^{\epsilon} - \int_{\R^{n-1}\times\{0\}\times [s_2,s_1]}\p_{y_n} w^\epsilon \p_sw^\epsilon \ dy' ds\\
&= \|Aw^\epsilon\|_{L^2}^2 + \|Sw^\epsilon\|_{L^2}^2 +\tau \int \tilde{\phi}''|w^{\epsilon}|^2 +\text{Bdry},
\end{split}
\end{equation}
where the boundary contribution is 
\begin{align*}
\text{Bdry}&:=-\int_{\R^{n-1}\times\{0\}\times [s_2,s_1]}(\p_{y_n} w^\epsilon) (\p_sw^\epsilon) \ dy' ds  + \left[\int_{\R^n_+} |\nabla w^\epsilon|^2 + (\frac{|y|^2}{16}-\tau\tilde\phi') |w^\epsilon|^2  dy \right]^{s_1}_{s_2}.
\end{align*}
\emph{Claim:} We have $|\text{Bdry}|\leq E$, where $E:=E(\epsilon,R,\tau,r,\rho)\rightarrow 0$ if we first let $\epsilon\rightarrow 0$ and then $R\rightarrow \infty$. \\
\emph{Proof of the claim:} Recall that $w^\epsilon (y,s)= (-t)^{\frac{n}{4}}e^{\psi} u^\epsilon(ye^{-\frac{s}{2}}, -e^{-s})\zeta(|y|).$ Then, using $\p_n u^{\epsilon}(x',0,t)=\beta_\epsilon(u^\epsilon)$ and the change of variables \eqref{eq:cov}, we write the boundary contribution in terms of $u^\epsilon$ and in $(x,t)$ variables:
\begin{align*}
&-\int_{A'_{r/2,2\rho}}e^{2\psi} (-t)\beta_\epsilon(u^\epsilon) (\p_t u^\epsilon + \frac{x'}{2t}\cdot \nabla u^\epsilon)\zeta^2 \ dx'dt -\int_{A'_{r/2,2\rho}}e^{2\psi}(\tau\p_s\tilde\phi-\frac{n}{4}) \beta_\epsilon(u^\epsilon) u^\epsilon \zeta^2\ dx'dt \\
&+\left[\int_{\R^n_+} e^{2\psi} (-t)\left|\zeta\nabla  u^\epsilon+ \frac{x}{|x|\sqrt{|t|}}u^\epsilon \zeta'+ \frac{x}{4t} u^\epsilon \zeta\right|^2 \ dx \right]^{t=-\frac{r^2}{4}}_{t=-\rho^2}+ \left[\int_{\R^n_+} e^{2\psi}  \left(\frac{|x|^2}{16|t|}-\tau \p_s\tilde{\phi}\right) |u^\epsilon|^2 \zeta^2\ dx\right]^{t=-\frac{r^2}{4}}_{t=-\rho^2}\\
&=: I_1+I_2 + I_3 + I_4.
\end{align*}
For $I_1$, using that $\beta_\epsilon(u^\epsilon)\p_t u^\epsilon = \p_t B_\epsilon(u^\epsilon)$ and $\beta_\epsilon(u^\epsilon)x'\cdot \nabla u^\epsilon=x'\cdot \nabla B_\epsilon(u^\epsilon)$, where $B_\epsilon'=\beta_\epsilon$ with $B_\epsilon(0)=0$ is an anti-derivative of $\beta_\epsilon$, and an integration by parts in $t$ and $x$, we get
\begin{align*}
I_1&=\int_{A'_{r/2,\rho}}\left[\p_t (e^{2\psi} (-t)\zeta^2)-\nabla\cdot (e^{2\psi}x'\zeta^2)\right] B_\epsilon(u^\epsilon)\ dx'dt - \left[\int_{\R^{n-1}\times \{0\}} e^{2\psi}(-t) B_\epsilon(u^\epsilon)) \zeta^2\ dx'\right]^{t=-(\frac{r}{2})^2}_{t=-\rho^2}.
\end{align*} 
We infer from the estimate
\begin{align*}
\left\|\beta_\epsilon(u^\epsilon)\right\|_{L^\infty(B'_R\times [-\rho^2,-(\frac{r}{2})^2])}\leq C,
\end{align*}
where $C=C(n, p, \|f\|_{L^p(Q_1^+)}, \|a^{ij}\|_{W^{1,1}_p(Q_1)^+})$ (cf. Lemma 4 and Lemma 5 of \cite{AU88}) that $u^\epsilon\geq -C\epsilon$. Thus
\begin{align*} 
\|B_\epsilon(u^\epsilon)\|_{L^\infty(B'_R\times [-\rho^2,-(\frac{r}{2})^2])}\leq C\epsilon.
\end{align*}
From this we conclude that $|I_1|\leq C_1\epsilon$, where $C_1=C_1(n, p, \|f\|_{L^p}, \|a^{ij}\|_{W^{1,1}_p}, r,\rho,\tau)$.  To estimate $I_2$ we note that 
\begin{align*}
0\leq \beta_\epsilon(u^\epsilon)u^\epsilon \leq C\epsilon \text{ on } B'_R\times [-\rho^2,-(\frac{r}{2})^2],
\end{align*}
which follows from the $L^\infty$ bound on $\beta_\epsilon(u_\epsilon)$, $\beta_\epsilon(s)=0$ when $s\geq 0$ and $u^\epsilon\geq -C\epsilon$. This implies that $|I_2| \leq C_1 \epsilon$. Here  $C_1$ may vary from lines but is always with the same dependence of the parameters.
To estimate $I_3$ we first use \cite[Claim A.1]{DGPT} with $v=u^\epsilon \zeta$ and then apply the $H^2$-estimate in Lemma \ref{lem:app_H2}  to find 
\begin{align*}
|I_3|&\lesssim \sum_{t=-\frac{r^2}{4},-\rho^2}\int_{\R^n_+} e^{2\psi} (-t)\zeta^2|\nabla  u^\epsilon|^2+ e^{2\psi}|u^\epsilon \zeta'|^2+e^{2\psi}|u^\epsilon \zeta|^2\ dx \\
& \lesssim \tau r^{-2}\int_{A_{r/2,r}^+}e^{2\psi} |u^\epsilon|^2 \zeta^2 + r^{2}\int_{A_{r/2,r}^+} e^{2\psi} |f^\epsilon|^2 \zeta^2 +\tau \rho^{-2}\int_{A_{\rho,2\rho}^+}e^{2\psi} |u^\epsilon|^2 \zeta^2\\
& + \rho^{2}\int_{A_{\rho, 2\rho}^+} e^{2\psi} |f^\epsilon|^2 \zeta^2+ E+ e^{-\frac{(R-1)^2}{4}}\sum_{t=-\frac{r^2}{4},-\rho^2}\int  e^{2\tau \phi(t)} |u^{\epsilon}|^2,
\end{align*}
where $E=E(\epsilon, R, r,\tau)\rightarrow 0$ as $\epsilon\rightarrow 0$ and then $R\rightarrow \infty$.  Here, in the term involving $|u^\epsilon \zeta'|$, we have used the fact that $\supp \zeta' \subset \{(x,t): R-1\leq \frac{|x|}{\sqrt{|t|}}\leq R\}.$  Since by assumption $u$ and $f$ have compact support, it follows that $u^\epsilon,f^\epsilon \to 0$ in $L^2(B_R^+\times [-r^2,-(\frac{r}{2})^2])$ and $L^2(B_R^+\times [-4\rho^2, -\rho^2])$ if we first let  $\epsilon \to 0$ and then $R \to \infty.$ Also, from the global energy estimate \eqref{eq:global_energy} and the fact that $e^{-\frac{(R-1)^2}{4}} \to 0$ as $R \to \infty,$ the last term on the right-hand side of the above inequality goes to $0.$  
Thus, we have that $I_3\rightarrow 0$ if we first let $\epsilon\rightarrow 0$ and then $R\rightarrow \infty$. For the term $I_4$ we use similar estimate and get that $|I_4|\rightarrow 0$ as $\epsilon\rightarrow 0$ and then $R\rightarrow \infty$. This completes the proof for the claim.

We infer from \eqref{eq:L2} that 
\begin{align}\label{eq:Carleman1}
\|\p_s w^\epsilon\|_{L^2}^2+ \|Sw^\epsilon\|_{L^2}^2 + \tau\|(\tilde{\phi}'')^{1/2}w^\epsilon\|_{L^2}^2 \leq \|(\p_s+H-\tau\tilde\phi')w^\epsilon\|_{L^2}^2 + E.
\end{align}
 To obtain the positive contribution involving $\nabla_y w$, we make use of the symmetric part of the conjugated operator. More precisely, an integration by parts in $y$ yields (recall that $S=H-\tau\tilde\phi'$)
\begin{align*}
&\quad \int \tilde\phi''|\nabla w^\epsilon|^2\ dyds +\frac{1}{16}\int \tilde\phi''|y|^2(w^\epsilon)^2 \ dyds\\
&= \int \tilde\phi''Hw^\epsilon w^\epsilon \ dyds - \int_{\R^{n-1}\times \{0\}\times [s_2,s_1]} \tilde{\phi}''\p_n w^\epsilon w^\epsilon \ dy'ds\\
&\leq \int \tilde\phi''Sw^\epsilon w^\epsilon+\tau\int \tilde\phi''\tilde\phi'(w^\epsilon)^2 \ dyds\\
&\leq \frac{1}{\tau}\|Sw^\epsilon\|_{L^2}^2 + \tau\|\tilde\phi''w^\epsilon\|_{L^2}^2 + \tau \|(\tilde\phi'')^{1/2}(\tilde\phi')^{1/2} w^\epsilon\|_{L^2}^2,
\end{align*}
where in the first inequality we have used that the boundary integral is nonnegative, which follows from the facts that $\p_nw^\epsilon w^\epsilon \geq 0$ and that $\tilde\phi''\geq 0$. Since $1/2<1-c_0\pi/2<\tilde\phi'(s)< 1+c_0\pi/2$, $\tilde\phi''(s)=\frac{c_0}{1+s^2}\leq \frac{1}{2}$ for $c_0\leq \frac{1}{4}$ and $\tau\geq 1$,  combining the above inequality with \eqref{eq:Carleman1} we have 
\begin{equation*}
\begin{split}
&\quad \tau \left\|(\tilde\phi'')^{1/2}w^\epsilon\right\|_{L^2}^2 + \left\|(\tilde\phi'')^{1/2}\nabla w^\epsilon\right\|_{L^2}^2 + \left\|(\tilde{\phi}'')^{1/2}|y|w^\epsilon\right\|_{L^2}^2 \lesssim \left\|(\p_s+H-\tau\tilde\phi')w^\epsilon\right\|_{L^2}^2+ E,
\end{split}
\end{equation*}
where the error term $E\rightarrow 0$ as first $\epsilon\rightarrow 0$ and then $R\rightarrow \infty$. Writing the above inequality in terms of $\tilde u^\epsilon$ and recalling $w^\epsilon (y,s)= (-t)^{\frac{n}{4}}e^{\frac{|x|^2}{8t}}e^{\tau\tilde\phi(|\ln(-t)|)}\tilde u^\epsilon(ye^{-\frac{s}{2}}, -e^{-s})$, we conclude the claimed inequality \eqref{eq:Carleman_heat}.

\medskip

\emph{Step 2: Estimate the RHS of \eqref{eq:Carleman_heat}.} We note that $\tilde{u}^\epsilon$ solves
\begin{align}\label{eq:tilde_u}
\p_t \tilde u^\epsilon -\Delta \tilde u^\epsilon = f^\epsilon\zeta +\p_i((a^{ij}_\epsilon-\delta^{ij})\p_ju^\epsilon)\zeta-2\nabla u^\epsilon\nabla\zeta+u^\epsilon(\p_t-\Delta )\zeta.
\end{align}
Since $\supp (\zeta') \subset \{(x,t): R-1\leq \frac{|x|}{\sqrt{|t|}}\leq R\}$, then
\begin{align*}
&\quad \|e^{\psi}(-t)^{\frac{1}{2}} \left(\nabla u^\epsilon\cdot \nabla\zeta + u^\epsilon (\p_t-\Delta)\zeta\right)\|_{L^2(A_{r,\rho}^+)}\leq C_{\tau, \rho, r}e^{-\frac{R^2}{4}}\left(\|u^\epsilon\|_{L^2(B_R^+\times (-1,0])} + \|\nabla u^\epsilon\|_{L^2(B_R^+\times (-1,0])}\right).
\end{align*}
By the uniform global energy estimate the above right hand side is uniformly bounded in $\epsilon$, thus goes to zero as $R\rightarrow \infty$. 
Then in order to estimate the RHS of \eqref{eq:Carleman_heat} we need to estimate
\begin{equation*}
\begin{split}
X_1:=\left\|e^{\psi}(-t)^{\frac{1}{2}}(a^{ij}_\epsilon-\delta^{ij})\p_{ij}u^\epsilon\zeta\right\|_{L^2(A_{r/2,\rho}^+)}^2,\ X_2:=\left\|e^{\psi}(-t)^{\frac{1}{2}} \p_i a^{ij}_\epsilon\p_ju^\epsilon\zeta\right\|_{L^2(A_{r/2,\rho}^+)}^2.
\end{split}
\end{equation*}
For that we introduce the following quantity: for $\gamma:=1-\frac{n+2}{p}$,
\begin{align*}
C(a^{ij}):= \sup_{A_{r/2,\rho}^+} \frac{|a^{ij}(x,t)-\delta^{ij}|}{(|x|+\sqrt{-t})^\gamma} + \sup_{\tilde{r}\in [\rho, \frac{1}{2}]}\tilde{r}^{-\frac{2\gamma}{n+2}}\|\nabla a^{ij}\|_{L^{n+2}(A_{\tilde{r},2\tilde{r}}^+)}.
\end{align*}
Note that $C(a^{ij})\lesssim \|a^{ij}\|_{W^{1,1}_p(A_{r,\rho}^+)}$ under our assumptions on $a^{ij}$.

\emph{Estimate for $X_1$:} We use the dyadic decomposition in time together with the weighted $H^2$-estimate for $u^\epsilon$ in Lemma \ref{lem:app_H2}. More precisely, let $A_m$ and $N_m$, $m\in \mathcal{N}:=\{\lfloor -\frac{\ln \rho}{\ln 2}\rfloor ,\cdots, \lfloor -\frac{\ln r}{\ln 2}\rfloor +1\}$, be defined in \eqref{eq:dyadic}. Then with $\omega(x,t):=1+\frac{|x|}{\sqrt{-t}}$ we have
\begin{align*}
X_1&\leq  C(a^{ij})^2 \left\|e^\psi (-t)^{\frac{1}{2}}(|x|+\sqrt{-t})^{\gamma} \p_{ij}u^\epsilon \zeta\right\|_{L^2(A_{r/2,\rho}^+)}^2\\
&\lesssim C(a^{ij})^2 \sum_{m\in \mathcal{N}} 2^{-2m(1+\gamma)}\left\|e^{\psi}\omega^{\gamma}\p_{ij}u^\epsilon \zeta\right\|_{L^2(A_m)}^2.
\end{align*}
 By Lemma \ref{lem:app_H2}, if $\delta\leq \delta_0(n)$ small, then 
\begin{align*}
X_1&\lesssim C(a^{ij})^2\left(\tau^2 \sum_{m\in \mathcal{N}} 2^{-2m(-1+\gamma)}\left\|e^{\psi} \omega^{\gamma}u^\epsilon \zeta\right\|_{L^2(N_m)}^2+ 2^{-2m(1+\gamma)}\left\|e^\psi \omega^\gamma f^\epsilon\zeta\right\|_{L^2(N_m)}^2+ E\right),
\end{align*}
where $E:=(\epsilon,R,\tau,\rho, r)\rightarrow 0$ as $\epsilon\rightarrow 0$ and then $R\rightarrow \infty$. 
Using $\gamma \leq 1$, \cite[Claim A.1]{DGPT} and the gradient estimate in Lemma \ref{lem:H2}, we can estimate the first term on the above right hand side by
\begin{align*}
\|e^\psi \omega^\gamma u^\epsilon\zeta\|_{L^2(N_m)}^2 &\leq \|e^{\psi}\omega u^\epsilon \zeta\|_{L^2(N_m)}^2 \\
&\lesssim  \|e^{\psi}  u^\epsilon \zeta\|_{L^2(A_m\cup A_{m-1}\cup A_{m-2})}^2 + \|e^\psi |t|f^\epsilon\zeta\|_{L^2(A_m\cup A_{m-1}\cup A_{m-2}}^2).
\end{align*}
Thus summing up and using that $f^\epsilon$ is essentially supported in $Q_1^+\cup Q'_1$ (thus $|t|^{\gamma/2}\omega^\gamma\lesssim 1$ in $A_{r/2,2\rho}^+$), we obtain 
\begin{align*}
X_1 &\lesssim C(a^{ij})^2 \left(\tau^2 \left\|e^{\psi}(-t)^{-\frac{1}{2}+\frac{\gamma}{2}} u^\epsilon \zeta\right\|_{L^2(A_{r/4,2\rho}^+)}^2 + \left\|e^{\psi}(-t)^{\frac{1}{2}}f^\epsilon\zeta\right\|_{L^2(A_{r/4,2\rho}^+)}^2 + E\right),
\end{align*}
where $E\rightarrow 0$ if $\epsilon\rightarrow 0$ and then $R\rightarrow \infty$.

\emph{Estimate for $X_2$:} By H\"older's inequality and Sobolev embedding, 
\begin{align*}
X_2&\lesssim \sum_{m\in \mathcal{N}} 2^{-2m}\left\|e^{\psi} \p_ia^{ij}\p_j u^\epsilon \zeta\right\|_{L^2(A_m)}^2\lesssim \sum_{m\in \mathcal{N}} 2^{-2m}\|\p_ia^{ij}\|_{L^{n+2}(A_m)}^2 \|e^{\psi}\p_j u^\epsilon \zeta\|_{L^{\frac{2(n+2)}{n}}(A_m)}^2\\
&\lesssim \sum_{m\in \mathcal{N}} 2^{-2m}\|\p_ia^{ij}\|_{L^{n+2}(A_m)}^2\cdot  \left(\sup_{t\in [2^{-2m}, 2^{-2m-2}]}\|e^{\psi}\nabla u^\epsilon (\cdot, t)\zeta\|_{L^2(\R^n_+)}^2\right)^{\frac{2}{n+2}}\|\nabla(e^{\psi}\nabla u^\epsilon \zeta)\|_{L^2(A_m)}^{\frac{2n}{n+2}}.
\end{align*}
Invoking the weighted $H^2$-estimate for $u^\epsilon$, cf. Lemma \ref{lem:app_H2}, and using that
\begin{align*}
\|\nabla a^{ij}\|_{L^{n+2}(A_m)}^2\leq C(a^{ij})2^{-\frac{4m\gamma}{n+2}},
\end{align*}
and arguing as in the estimate for $X_1$, we get for $\tilde\gamma:=\frac{\gamma}{n+2}$,
\begin{align*}
X_2&\lesssim C(a^{ij})^2 \left(\tau^2 \left\|e^\psi(-t)^{-\frac{1}{2}+\tilde\gamma} u^\epsilon \zeta\right\|_{L^2(A_{r/4,\rho}^+)}^2  + \left\|e^\psi (-t)^{\frac{1}{2}+\tilde\gamma}f^\epsilon \zeta\right\|_{L^2(A_{r/4,2\rho}^+)}^2+E\right),
\end{align*}
where $E\rightarrow 0$ as $\epsilon\rightarrow 0$ and then $R\rightarrow \infty$.
Combining the estimates for $X_1$ and $X_2$ together we have
\begin{equation}\label{eq:x12}
\begin{split}
X_1+X_2&\lesssim C(a^{ij})^2 \left(\tau^2\left\|e^\psi (-t)^{-\frac{1}{2}+\tilde\gamma}u^\epsilon \zeta\right\|_{L^2(A_{r/4,2\rho}^+)}^2 + \left\|e^\psi (-t)^{\frac{1}{2}} f^\epsilon \zeta\right\|_{L^2(A_{r/4,2\rho}^+)}^2 + E\right).
\end{split}
\end{equation}

\medskip

\emph{Step 3: Conclusion.} In view of \eqref{eq:Carleman_heat}, \eqref{eq:tilde_u} as well as \eqref{eq:x12}, we conclude
\begin{align*}
&\tau \left\| e^{\psi} (-t)^{-\frac{1}{2}}(1+|\ln(-t)|^2)^{-\frac{1}{2}} \tilde u^\epsilon\right\|_{L^2(A_{r,\rho}^+)}^2+ \left\|e^{\psi}(1+|\ln(-t)|^2)^{-\frac{1}{2}}\nabla \tilde u^\epsilon\right\|_{L^2(A_{r,\rho}^+)}^2\\
&\lesssim  C(a^{ij})^2c_0^{-1}\left(\tau^2\left\|e^{\psi}(-t)^{-\frac{1}{2}+\tilde\gamma} \tilde{u}^\epsilon\right\|_{L^2(A_{r/4,2\rho}^+)}^2 +\left\|e^{\psi}(-t)^{\frac{1}{2}} f^\epsilon\zeta\right\|_{L^2(A_{r/4,2\rho}^+)}^2
+ E\right).
\end{align*}
We pass to the limit by first letting $\epsilon\rightarrow 0$ and then increase the support of $\zeta$ by letting $R\rightarrow \infty$. Since $\|\tilde{u}^\epsilon\|_{L^2(A_{r/4,r}^+)},  \|\tilde{u}^\epsilon\|_{L^2(A_{\rho,2\rho}^+)}\rightarrow 0$ as $\epsilon\rightarrow 0$ due to the support assumption of $u$ (and the same applies to the integral involving $f^\epsilon$),  we obtain the desired Carleman estimate. 
\end{proof}

We derive an immediate consequence of the Carleman inequality, the three "sphere" inequality for solutions to the parabolic Signorini problem. 
\begin{cor}\label{3sphere}
Let $u\in \mathcal{G}^{A,f}_p(S_1^+)$ with $p>n+2$. Assume that $1<\tau\leq \tau_0<\infty$. Then there is $R_0=R_0(\tau_0,c_0, n,p,\|f\|_{L^p}, \|a^{ij}\|_{W^{1,1}_p})\in (0,\frac{1}{2})$ and $C=C(n,\|a^{ij}\|_{W^{1,1}_p})$, such that for any $r_1,r_2,r_3\in (0,R_0)$ with $r_1<r_2<r_3/2$ and $2r_1<r_2$ we have
\begin{equation}
\label{eq:three_sphere}
\begin{split}
&\quad \tau^{\frac{1}{2}}(r_2)^{-1} |\ln (r_2)|^{-1}e^{\tau\phi(-r_2^2)}\|ue^{\frac{|x|^2}{8t}}\|_{L^2(A^+_{r_2,2r_2})}\\
&\leq Cc_0^{-1}\left((r_1)^{-1} e^{\tau\phi(-r_1^2)}\|u e^{\frac{|x|^2}{8t}}\|_{L^2(A^+_{r_1,2r_1})}+(r_3)^{-1} e^{\tau\phi(-r_3^2)}\|u e^{\frac{|x|^2}{8t}}\|_{L^2(A^+_{r_3,2r_3})}\right.\\
&\quad +\left.\|(-t)^{\frac{1}{2}}e^{\tau\phi(t)}e^{\frac{|x|^2}{8t}}f\|_{L^2(A^+_{r_1,2r_3})}\right).
\end{split}
\end{equation}
Here $\phi$ is the weight function defined in Proposition \ref{prop:Carleman}. 
\end{cor}

\begin{proof}
Let $\eta=\eta(t)\geq 0$ be a smooth cut-off function such that
\begin{align*}
\eta=1 \text{ in } A_{\frac{3r_1}{2}, \frac{5r_3}{4}}, \quad \eta=0 \text{ outside } A_{\frac{5r_1}{4}, \frac{3r_3}{2}},\\
|\p_t\eta|\lesssim \frac{1}{r_1^2} \text{ in } A_{\frac{5r_1}{4}, \frac{3r_1}{2}}, \quad |\p_t\eta|\lesssim \frac{1}{r_3^2} \text{ in } A_{\frac{5r_3}{4}, \frac{3r_3}{2}}.
\end{align*}
Then $\tilde{u}=u\eta$ is supported in $A_{\frac{5r_1}{4}, \frac{3r_3}{2}}^+$ and satisfies
\begin{align*}
\p_t\tilde{u}-\p_i(a^{ij}\p_j \tilde u)= f\eta +u\p_t\eta &\text{ in } S_1^+\\
\tilde u\geq 0, \ \p_\nu \tilde u\geq 0, \ u\p_{\nu} \tilde u =0 & \text{ on }  S'_1.
\end{align*}
Note that by Remark \ref{rmk:smallness} there is $\tilde{R}_0=\tilde{R}_0(n,p,\|a^{ij}\|_{W^{1,1}_p})$ small, such that the Carleman estimate holds true if $r_3\leq \tilde{R}_0$.
We apply the Carleman estimate Proposition \ref{prop:Carleman} to $\tilde u$ and get
\begin{align*}
&\tau^{\frac{1}{2}} \left\| e^{\tau\phi(t)+\frac{|x|^2}{8t}} (-t)^{-\frac{1}{2}}(1+|\ln(-t)|^2)^{-\frac{1}{2}} \tilde u\right\|_{L^2(A^+_{r_1, 2r_3})}+ \left\|e^{\tau\phi(t)+\frac{|x|^2}{8t}}  (1+|\ln(-t)|^2)^{-\frac{1}{2}}\nabla\tilde u\right\|_{L^2(A^+_{r_1,2r_3})}\\
&\leq Cc_0^{-1}\left(\tau\left\|e^{\tau\phi(t)+\frac{|x|^2}{8t}}(-t)^{-\frac{1}{2}+\tilde\gamma}\tilde u\right\|_{L^2(A^+_{r_1,2r_3})}+\left\|e^{\tau\phi(t)\frac{|x|^2}{8t}}(-t)^{\frac{1}{2}}f\eta\right\|_{L^2(A^+_{r_1,2r_3})}\right.\\
&\quad \left. +\left\|e^{\tau\phi(t)+\frac{|x|^2}{8t}}(-t)^{\frac{1}{2}}u\p_t\eta\right\|_{L^2(A^+_{r_1,2r_3})}\right).
\end{align*}
We now choose $R_0=R_0(\tau_0, c_0,n,p,\|f\|_{L^p}, \|a^{ij}\|_{W^{1,1}_p}) <\tilde{R}_0$ to be small enough such that when $r_3\leq R_0,$  the first term on the right hand side above can be absorbed by the left hand side. This is possible due to the presence of the additional factor $(-t)^{\tilde{\gamma}}$ in the first term on the right hand side above. For the third term on the right hand side we infer from the support of $\eta$ that
\begin{align*}
\left\|(-t)^{\frac{1}{2}}e^{\tau\phi(t)+\frac{|x|^2}{8t}} u\p_t\eta\right\|_{L^2(A^+_{r_1,2r_3})}&\lesssim \left\|(-t)^{-\frac{1}{2}}e^{\tau\phi(t)+\frac{|x|^2}{8t}} u\right\|_{L^2(A^+_{\frac{5r_1}{4}, \frac{3r_1}{2}})}+ \left\|(-t)^{-\frac{1}{2}}e^{\tau\phi(t)+\frac{|x|^2}{8t}}u\right\|_{L^2(A^+_{\frac{5r_3}{4}, \frac{3r_3}{2}})}.
\end{align*}
Thus we have
\begin{align*}
&\tau^{\frac{1}{2}}(r_2)^{-1} |\ln (r_2)|^{-1}e^{\tau\phi(-r_2^2)}\|ue^{\frac{|x|^2}{8t}}\|_{L^2(A^+_{r_2,2r_2})}\\
&\leq C c_0^{-1}\left((r_1)^{-1} e^{\tau\phi(-r_1^2)}\|u e^{\frac{|x|^2}{8t}}\|_{L^2(A^+_{r_1,2r_1})}+(r_3)^{-1} e^{\tau\phi(-r_3^2)}\|u e^{\frac{|x|^2}{8t}}\|_{L^2(A^+_{r_3,2r_3})}\right.\\
&\quad +\left.\|(-t)^{
\frac{1}{2}}e^{\tau\phi(t)}e^{\frac{|x|^2}{8t}}f\eta\|_{L^2(A^+_{r_1,2r_3})}\right).
\end{align*}
This completes the proof for the corollary.
\end{proof}
\begin{rmk}[Strengthened three-sphere]
	We can strengthen Corollary \ref{3sphere} by using the antisymmetric part of the conjugated operator in \emph{Step 1} of the proof for Proposition \ref{prop:Carleman}, see also \cite[Remark 14]{R14} for the elliptic case:
We use the Poincar\'e inequality to estimate $\|\partial_s w^\epsilon\|_{L^2}$ on the left hand side of \eqref{eq:Carleman1} to get
\begin{align*}
\|w^\epsilon-w^\epsilon(\cdot, \tilde s_i)\|_{L^2(\R^n_+\times (\tilde s_2,\tilde s_1))} \lesssim (\tilde s_2-\tilde s_1) \|\partial_s w^\epsilon\|_{L^2(\R^n_+\times (\tilde s_2,\tilde s_1))},\quad i=1,2
\end{align*}
for any $\tilde s_1$ and $\tilde s_2$ such that $-\ln \rho^2\leq \tilde s_2<\tilde s_1\leq -\ln r^2$. Given $0<r<\tilde r<\rho<1$, we now apply the above inequality with $s_2=-\ln\rho^2$ and $s_1=-\ln \tilde r^2$ (or $s_2=-\ln \tilde r^2$ and $s_1=-\ln r^2$) and then write $w^\epsilon$ in terms of $u^\epsilon$. Let $\epsilon\rightarrow 0$ as well as the support of the cut-off function swipe $\R^n$ (i.e. let $R\rightarrow \infty$). Since  $u(\cdot, -r^2)=0$ and $u(\cdot, -\rho^2)=0$ by our assumption and $D_{x,t} u^\epsilon$ is uniformly bounded in $L^2$, by the trace theorem we thus have that $u^\epsilon(\cdot, -r^2)\rightarrow 0$ and $u^\epsilon(\cdot, -\rho^2)\rightarrow 0$ in $L^2(\R^n_+)$. Thus we get an additional contribution on the left hand side of the Carleman estimate: for $\tilde r\in (r,\rho)$, 
\begin{align*}
\ln(\rho/\tilde r)^{-1}\|e^{\tau\phi(t)+\frac{|x|^2}{8t}}(-t)^{-\frac{1}{2}} u \|_{L^2(A^+_{\tilde r,\rho})}^2 + \ln (\tilde r/r)^{-1}\|e^{\tau\phi(t)+\frac{|x|^2}{8t}}(-t)^{-\frac{1}{2}} u \|_{L^2(A^+_{r,\tilde r})}^2
\end{align*} 
This leads to an improvement of the three-sphere inequality \eqref{eq:three_sphere}:
		\begin{equation}
			\label{eq:three_sphere_strong}
			\begin{split}
				&\quad \tau^{\frac{1}{2}}(r_2)^{-1} |\ln (r_2)|^{-1}e^{\tau\phi(-r_2^2)}\|ue^{\frac{|x|^2}{8t}}\|_{L^2(A^+_{r_2,2r_2})}\\
				&\quad +\max\{(\operatorname{ln}(r_2/r_1))^{-1}, (\ln(r_3/r_2))^{-1}\}(r_2)^{-1}e^{\tau\phi(-r_2^2)}\|ue^{\frac{|x|^2}{8t}}\|_{L^2(A^+_{r_2,2r_2})}\\
				&\leq Cc_0^{-1}\left((r_1)^{-1} e^{\tau\phi(-r_1^2)}\|u e^{\frac{|x|^2}{8t}}\|_{L^2(A^+_{r_1,2r_1})}+(r_3)^{-1} e^{\tau\phi(-r_3^2)}\|u e^{\frac{|x|^2}{8t}}\|_{L^2(A^+_{r_3,2r_3})}\right.\\
&\quad +\left.\|(-t)^{\frac{1}{2}}e^{\tau\phi(t)}e^{\frac{|x|^2}{8t}}f\|_{L^2(A^+_{r_1,2r_3})}\right).
			\end{split}
		\end{equation}
		The strengthened inequality will be crucial later in the proof for  the doubling inequality, cf. Lemma \ref{doubling}.
\end{rmk}

\section{Almost optimal regularity}\label{sec:al_opt}
In this section, we introduce a notion of vanishing order and derive several properties of it, i.e. well-definiteness, upper semi-continuity and the gap of the vanishing order beyond $3/2$. We establish a doubling inequality which is crucial later for the compactness in the blowup arguments. Furthermore, we derive an almost optimal (up to a logarithmic loss) growth estimate of the solution. For the proofs we  rely on the three sphere's inequality  obtained in the previous section.

\medskip

Throughout this section, for notational simplicity we will work with $u\in \mathcal{G}^{A,f}_p(S_1)$ with $p>n+2$ and $u, A$ and $f$ even about $x_n$. 

\begin{defi}[Vanishing order]\label{defi:vanishing_order}
Given $u\in \mathcal{G}^{A,f}_p(S_1)$ and $(x_0,t_0)\in Q_{1/4}$, we define the vanishing order of $u$ at $(x_0,t_0)$ as 
\begin{align*}
\kappa_{(x_0,t_0)}:=\limsup_{r\rightarrow 0+}\frac{\ln \left(\fint_{A_{r,2r}(x_0,t_0)}u^2e^{\frac{\sigma(x-x_0;x_0,t_0)}{4(t-t_0)}} \ dxdt \right)^{\frac{1}{2}}}{\ln r}\in [0,\infty],
\end{align*}
where
$\sigma(y;x_0,t_0):=\langle A(x_0,t_0)^{-1}y, y\rangle$ for any $y\in \R^n$ and $A(x_0,t_0)^{-1}$ is the inverse of the matrix $A(x_0,t_0)=(a^{ij}(x_0,t_0))$. Here we define $\ln 0:=-\infty$, so the vanishing order of the trivial solution is $+\infty$.

When $p>n+2$, it follows from the regularity results in \cite{AU} that if $(x_0,t_0)\in \Gamma_u^\ast$, then $\kappa_{(x_0,t_0)}\in [ 1+\alpha,\infty]$ for some $\alpha\in (0,1)$.
\end{defi}
We  note that $\kappa_{(x_0,t_0)}$ can also be written as
\begin{align*}
\kappa_{(x_0,t_0)} = \limsup_{r\rightarrow 0+} \frac{\ln \left(\fint_{A_{r,2r}(x_0,t_0)}u^2 G_{(x_0,t_0)} \ dxdt \right)^{\frac{1}{2}}}{\ln r},
\end{align*}
where $G_{(x_0,t_0)}$ is the parametrix
\begin{align}\label{eq:parametrix}
G_{(x_0,t_0)}(x,t)=\frac{(\det(A(x_0,t_0)))^{\frac{1}{2}}}{(4\pi (t_0-t))^{\frac{n}{2}}}e^{\frac{\sigma(x-x_0;x_0,t_0)}{4(t-t_0)}} \text{ for } t<t_0,
\end{align}
 and $G_{(x_0,t_0)}(x,t)=0$ if $t\geq t_0$. When $(x_0,t_0)=(0,0)$, $G_{(0,0)}$ is the standard Gaussian and we write $G:=G_{(0,0)}$. Here and in Definition \ref{defi:vanishing_order} the averaged integral is defined as
\begin{align*}
\fint_U f w\ dxdt:=\frac{\int_{U} f w \ dxdt}{\int_U w\ dxdt},
\end{align*}
for the weight $w$ and $U\subset \R^n\times (-\infty,0)$. 

\begin{rmk}\label{rmk:sigma}
The weight function arises naturally from the change of variables: fix any $(x_0,t_0)\in Q'_{1/4}$, we apply the change of variable
\begin{align*}
T_{(x_0,t_0)}:A_{r,2r}(x_0,t_0)&\rightarrow A_{r,2r}\\
(x,t)&\mapsto (y,s)=T_{(x_0,t_0)}(x,t):=(B(x_0,t_0)(x-x_0),t-t_0),
\end{align*}
 where  $B^T(x_0,t_0)B(x_0,t_0)=A(x_0,t_0)^{-1}$. Note that $B(x_0,t_0)$ also satisfies the off-diagonal condition, i.e. $b^{in}(x'_0,0,t_0)=0$ for $i\in \{1,\cdots, n-1\}$. Thus $A'_{r,2r}(x_0,t_0)$ is mapped onto $A'_{r,2r}$. Let $v(y, s):=u(T^{-1}_{(x_0,t_0)}(y,s))$. Then $v$ solves the equation
\begin{align*}
\p_s v(y,s) - \nabla\cdot (C(y,s)\nabla v(y,s))=\tilde{f}(y,s) \text{ in } A_{r,2r},
\end{align*}
where with $B:=B(x_0,t_0)$
\begin{align*}
\quad C(y,s):=B^TA(T_{(x_0,t_0)}^{-1}(y,s))B,\quad \tilde{f}(y,s):=f(T_{(x_0,t_0)}^{-1}(y,s)),
\end{align*}
and  $v(y,s)$ satisfies the Signorini boundary condition on $A'_{r,2r}$. 
It is not hard to see that $C(0,0)=id$, 
\begin{align*}
\frac{7}{9}|\xi|^2 \leq \langle C(y,s)\xi, \xi\rangle \leq \frac{9}{7}|\xi|^2, \quad \forall \xi\in \R^n, \quad \forall (y,s)\in A_{r,2r},
\end{align*} 
and $C$ satisfies assumptions (ii) and (iii). 
Thus the three sphere's inequality \eqref{eq:three_sphere_strong} holds for $v$. Writing it back in term of $u$ and the original variables, we then obtain a three sphere's inequality for $\|u e^{\frac{\sigma(x-x_0;x_0,t_0)}{8(t-t_0)}}\|_{L^2(A_{r,2r})}$.
\end{rmk}

We first prove that in certain cases, the $\limsup$ is actually a limit in the definition of the vanishing order. 

\begin{prop}\label{prop:limit}
If  $\kappa_{(x_0,t_0)}\leq 2-\frac{n+2}{p}$, then 
\begin{align*}
\limsup_{r\rightarrow 0+} \frac{\ln \|ue^{\frac{\sigma(x-x_0; x_0,t_0)}{8(t-t_0)}}\|_{L^2(A_{r,2r}(x_0,t_0))}}{\ln r} =\liminf_{r\rightarrow 0+} \frac{\ln \|ue^{\frac{\sigma(x-x_0;x_0,t_0)}{8(t-t_0)}}\|_{L^2(A_{r,2r}(x_0,t_0))}}{\ln r}.
\end{align*}
When $f$ satisfies the support condition \eqref{eq:support}, then the above equality holds true as long as $\kappa_{(x_0,t_0)}<\infty$. 
\end{prop}
\begin{proof}
Without loss of generality assume that $(x_0,t_0)=(0,0)$ and for simplicity let $\kappa:=\kappa_{(0,0)}$. 
For any $0<\epsilon\ll 1$ fixed, define
$\tau=\frac{\kappa-\epsilon+n/2}{2+c_0\pi}>1$ ($\tau >1$ can be achieved by taking $c_0$ sufficiently small depending on $\alpha$). Noting that
\begin{align}\label{eq:asymp}
\frac{\tau\tilde\phi(2|\ln r|)}{|\ln r|}\to \tau(2+c_0\pi)\ \text{ as } r\rightarrow 0+
\end{align}
and that 
\begin{align*}
\kappa=\limsup_{r\rightarrow 0}\frac{\ln \|ue^{\frac{|x|^2}{8t}}\|_{L^2(A_{r,2r})}}{\ln r}-\frac{n+2}{2},
\end{align*} 
we have a sequence $r_j \rightarrow 0$ such that 
$$(r_j)^{-1} e^{\tau\phi(-r_j^2)}\|u e^{\frac{|x|^2}{8t}}\|_{L^2(A_{r_j,2r_j})} \le C.$$ 
For given $r_2$ with $0<r_2 \ll 1.$ Choose $r_j$ such that $r_j <r_2/2.$ We now set $r_1=r_j$ and $r_3\sim 1$ in the three
sphere’s inequality \eqref{eq:three_sphere}. Then, the first  term on the right hand side of \eqref{eq:three_sphere} is uniformly bounded. For the third term on the right hand side of \eqref{eq:three_sphere}, we first use the fact that $\frac{\tau\tilde\phi(2|\ln r|)}{|\ln r|}$ is non-increasing as a function of $r$, along with \eqref{eq:asymp}, and then apply H\"older's inequality to estimate
\begin{equation}\label{eq:est_f_1}
\begin{split}
&\quad \|(-t)^{\frac{1}{2}}e^{\tau\phi}e^{\frac{|x|^2}{8t}}f\|_{L^2(A_{r_j,2r_3})}\lesssim \|\sqrt{-t} (\sqrt{-t})^{-(\kappa-\epsilon +\frac{n}{2})} e^{\frac{|x|^2}{8t}}f\|_{L^2(A_{r_j,2r_3})}\\
&\lesssim \|f\|_{L^p(A_{r_1,2r_3})}\|(\sqrt{-t})^{-(\kappa-\epsilon +\frac{n-2}{2})}e^{\frac{|x|^2}{8t}}\|_{L^{\frac{2p}{p-2}}(A_{r_j,2r_3})}\lesssim \|f\|_{L^p(A_{r_j,2r_3})}r_3^{\delta}
\end{split}
\end{equation}
with $\delta=-(\kappa-\epsilon)+(2-\frac{n+2}{p}) \ge 0,$ by our choice of $\kappa$. More precisely, we have that $\delta\geq 0$ for all $\epsilon\geq \kappa-(2-\frac{n+2}{p})$ (in particular for all $\epsilon\geq 0$, if $\kappa\leq  2-\frac{n+2}{p}$). Thus when $\kappa\leq 2-\frac{n+2}{p}$, the right hand side of \eqref{eq:three_sphere} is uniformly bounded and we get
\begin{align*}
(r_2)^{-1} |\ln (r_2)|^{-1}e^{\tau\phi(-r_2^2)}\|ue^{\frac{|x|^2}{8t}}\|_{L^2(A_{r_2,2r_2})} \leq C, \quad \forall 0<r_2\ll r_3. 
\end{align*}
Taking the logarithm on both sides and dividing by $\ln r_2$ we get
\begin{align*}
\frac{\ln \|ue^{\frac{|x|^2}{8t}}\|_{L^2(A_{r_2,2r_2})}}{\ln r_2}\geq \frac{\tau \tilde\phi(2|\ln r_2|)}{|\ln r_2|}+1+\frac{\ln |\ln r_2|}{\ln r_2} +\frac{\ln C}{\ln r_2}.
\end{align*}
Thus by \eqref{eq:asymp} we have
\begin{align*}
\liminf_{r_2\rightarrow 0+} \frac{\ln \|ue^{\frac{|x|^2}{8t}}\|_{L^2(A_{r_2,2r_2})}}{\ln r_2} \geq \kappa +\frac{n+2}{2}-\epsilon. 
\end{align*}
Since $\epsilon>0$ is arbitrary, we obtain the desired estimate.

If $f$ is supported away from $(0,0)$, then for any $\tau<\infty$
\begin{align*}
\|(-t)^{\frac{1}{2}}e^{\tau\phi}e^{\frac{|x|^2}{8t}}f\|_{L^2(A_{r_1,2r_3})}\leq C_\tau.
\end{align*}
Thus the conclusion is true as long as $\kappa<\infty$. 
\end{proof}

\begin{rmk}\label{growth}
An immediate consequence of Proposition \ref{prop:limit} is the following growth estimate for $u\in \mathcal{G}^{A,f}_p(S_1)$:
	For every $\epsilon\in (0,1/4)$ and for every $(x_0,t_0) \in Q_{1/4}$ with $\kappa_{(x_0,t_0)}\leq \bar \kappa$, where $\bar \kappa=2-\frac{n+2}{p}$ and $\bar\kappa\in (0,\infty)$ if $f$ satisfies additionally the support condition \eqref{eq:support},  there exists $r_{\epsilon}=r(\epsilon,u,\bar\kappa, (x_0, t_0))>0$ such that for all $0<r<r_\epsilon$, we have
	\begin{equation*}
	r^{\kappa_{(x_0,t_0)}+\frac{n+2}{2}+\epsilon} \le \big\|ue^{\frac{\sigma(x-x_0;x_0,t_0)}{8(t-t_0)}}\big\|_{L^2(A_{r,2r}(x_0,t_0))} \le r^{\kappa_{(x_0,t_0)}+\frac{n+2}{2}-\epsilon}.
	\end{equation*}
	While the lower bound is in general not uniform in $u$ and $(x_0,t_0)$, the upper bound can be improved, as we will show in Lemma \ref{lem:growth}.
\end{rmk}

In the next lemma we establish the upper semi-continuity of the map $(x_0,t_0)\mapsto \kappa_{(x_0,t_0)}$. 

\begin{lem}[Upper semi-continuity]
\label{lem:upper_semi}
	Let $u\in \mathcal{G}^{A,f}_p(S_1)$. Assume that $\kappa_{(x_0,t_0)}<2-\frac{n+2}{p}$.  Then the mapping $(x,t) \mapsto \kappa_{(x,t)}$ is upper semi-continuous at $(x_0,t_0)$, i.e. 
	\begin{align*}
	\limsup_{\substack{(x,t)\rightarrow (x_0,t_0),\\(x,t)\in  Q_{\delta}(x_0,t_0)}} \kappa_{(x,t)}\leq \kappa_{(x_0,t_0)}.
	\end{align*}
	If in addition $f$ satisfies \eqref{eq:support}, then the upper semi-continuity is satisfied at all points with  $\kappa_{(x_0,t_0)}<\infty$. 
	\end{lem}
\begin{proof}
	Without loss of generality we will show the upper semi-continuity at  $(0,0)$ and we let $\kappa:=\kappa_{(0,0)}$ for simplicity. Assume that $\kappa<2-\frac{n+2}{p}$. It is easy to see from the proof that this assumption can be replaced by $\kappa<\infty$ if $f$ satisfies the support condition \eqref{eq:support}, as the constraint for $\kappa$ only comes from estimating the inhomogeneity term $\|f e^{\frac{|x|^2}{8t}}\|_{L^2}$. Let $0<\epsilon<2-\frac{n+2}{p}-\kappa$ be arbitrary.
%
%
%
%
Fix $r_2>0$ small enough, which will be chosen later and will depend on $n,p,\epsilon$ and $u$. Let $(x_0,t_0) \in Q_{r_2^2/16}$. We apply \eqref{eq:three_sphere} of Corollary \ref{3sphere} at $(x_0,t_0)$ with $\tau=\frac{\kappa+\epsilon+n/2}{2+c_0\pi}$ and three radii $r_1,r_2,r_3$, where $4r_1<r_2$, $2r_2<r_3=\frac{R_0}{4}$ with $R_0\in (0,1)$ being the radius from Corollary \ref{3sphere}, to get
\begin{equation}\label{eq:upp_sem}
\begin{split}
&\quad \tau^{\frac{1}{2}}(r_2)^{-1} |\ln (r_2)|^{-1}e^{\tau\phi(-r_2^2)}\big\|ue^{\frac{\sigma(x-x_0;x_0,t_0)}{8(t-t_0)}}\big\|_{L^2(A_{r_2/2,2r_2}(x_0,t_0))}\\
&\leq C(r_1)^{-1} e^{\tau\phi(-r_1^2)}\big\|u e^{\frac{\sigma(x-x_0;x_0,t_0)}{8(t-t_0)}}\big\|_{L^2(A_{r_1,2r_1}(x_0,t_0))}+ X,
\end{split}
\end{equation}
where $C=C(n,p,\|f\|_{L^p}, \|a^{ij}\|_{W^{1,1}_p})$ and 
\begin{equation*}
\begin{split}
X&:=C\left((r_3)^{-1} e^{\tau\phi(-r_3^2)}\big\|u e^{\frac{\sigma(x-x_0;x_0,t_0)}{8(t-t_0)}}\big\|_{L^2(A_{r_3,2r_3}(x_0,t_0))}\right.\\
&\left. \quad \quad +\big\|(t_0-t)^{\frac{1}{2}}e^{\tau\phi(t-t_0)}e^{\frac{\sigma(x-x_0;x_0,t_0)}{8(t-t_0 )}}f\big\|_{L^2(A_{r_1,2r_3}(x_0,t_0))} \right).
\end{split}
\end{equation*}
We are going to show that for $r_2$ chosen sufficiently small,  $X$ will be absorbed by the left hand side of \eqref{eq:upp_sem}. Clearly, for $(x_0,t_0) \in Q_{r_2^2/16}$ we have $A_{r_3,2r_3}(x_0,t_0) \subset A_{r_3,3r_3}$, $A_{r_1,2r_3}(x_0,t_0) \subset A_{r_1,3r_3}.$ Therefore, by H\"older's inequality and a change of variable
\begin{align*}
X &\leq C \left( (r_3)^{-1} e^{\tau\phi(-r_3^2)}\|u\|_{L^2(A_{r_3,3r_3})} \|e^{\frac{|x|^2}{8t}}\|_{L^2(A_{r_3,2r_3})}\right.\\
&\quad +\left.\|(-t)^{\frac{1}{2}}e^{\tau\phi(t)}e^{\frac{|x|^2}{8t }}\|_{L^{\frac{2p}{p-2}}(A_{r_1,2r_3})}\|f\|_{L^p(A_{r_1,3r_3})}\right) \\
& =: M=M(R_0,\epsilon,\|u\|_{L^2(Q_{R_0})},\|f\|_{L^p(Q_{R_0})}),
\end{align*}
where the second term on the right hand side is estimated in a similar way as \eqref{eq:est_f_1}, i.e.
\begin{align*}
\|(-t)^{\frac{1}{2}}e^{\tau\phi(t)}e^{\frac{|x|^2}{8t }}\|_{L^{\frac{2p}{p-2}}(A_{r_1,2r_3})}
\lesssim \|(-t)^{-(\kappa+\epsilon+\frac{n-2}{2})}e^{\frac{|x|^2}{8t}}\|_{L^{\frac{2p}{p-2}}(A_{r_1,2r_3})}\lesssim r_3^{\delta'},
\end{align*}
with $\delta'=\frac{-\kappa p -\epsilon p + 2p -(n+2)}{p}>0$ by our choice of $\epsilon $ and $\kappa$. 

Now we estimate the left hand side of \eqref{eq:upp_sem} from below. We first note that for $(x_0,t_0)\in Q_{r_2^2/16}$ we have
\begin{align}\label{eq:lb_lhs}
\big\|ue^{\frac{\sigma(x-x_0;x_0,t_0)}{8(t-t_0)}}\big\|_{L^2(A_{r_2/2,2r_2}(x_0,t_0))} \gtrsim \|ue^{(1+Cr_2^\gamma)\frac{|x|^2}{8t}}\|_{L^2(A_{r_2,2r_2})}
\end{align}
for some absolute constant $C>0$. Indeed, for all $(x,t) \in \supp(u)\cap A_{r_2/2,2r_2}(x_0,t_0)$, using the $H^{\gamma,\gamma/2}$-regularity of $A=(a^{ij})$ we have 
\begin{align*}
	\frac{\sigma(x-x_0;x_0,t_0)}{t-t_0}-\frac{|x_0-x|^2}{t-t_0}
	\gtrsim r_2^\gamma\frac{|x-x_0|^2}{t-t_0}.
\end{align*}
 Furthermore,  if $r_2\ll 1$, it is not hard to see that for  $(x,t)$ as above,
\begin{align*}
\frac{|x-x_0|^2}{t-t_0}-\frac{|x|^2}{t}\gtrsim -1.
\end{align*}
Combining the above two inequalities together we have 
\begin{align*}
\frac{\sigma(x-x_0;x_0,t_0)}{t-t_0}\ge (1+Cr_2^\gamma) \frac{|x|^2}{t}-C
\end{align*}
for some absolute constant $C>0$, which implies 
	$e^\frac{\sigma(x-x_0;x_0,t_0)}{8(t-t_0)} \gtrsim e^{(1+Cr_2^\gamma)\frac{|x|^2}{8t}}$ for all $(x,t) \in \supp(u)\cap A_{r_2/2,2r_2}(x_0,t_0)$.
Also,  $(x_0,t_0) \in Q_{r_2^2/16}$ gives that $A_{r_2,2r_2} \subset A_{r_2/2,2r_2}(x_0,t_0)$. Hence \eqref{eq:lb_lhs} holds true.

\emph{Claim:} We have
\begin{align*}
\lim_{r\rightarrow 0+}\frac{\ln \|u e^{\frac{|x|^2}{8t}}\|_{L^2(A_{r,2r})}}{\ln r} = \lim_{r\rightarrow 0+}\frac{\ln \|u e^{(1+Cr^\gamma)\frac{|x|^2}{8t}}\|_{L^2(A_{r,2r})}}{\ln r}.
\end{align*}
\emph{Proof of the claim:} The inequality $\leq $ is trivial. To see the other direction, we note that $e^{Cr^\gamma|x|^2/t}\gtrsim 1$ when $|x|\lesssim r^{1-\frac{\gamma}{2}}$ and $t\sim -r^2$. Then it follows
\begin{align}\label{eq:est1}
\int_{-(2r)^2}^{-r^2}\int_{\R^n}u^2 e^{(1+Cr^\gamma)\frac{|x|^2}{4t}}\gtrsim \int_{-(2r)^2}^{-r^2}\int_{|x|\lesssim r^{1-\frac{\gamma}{2}}} u^2 e^{\frac{|x|^2}{4t}}.
\end{align}
Since 
\begin{align*}
\int_{-(2r)^2}^{-r^2}\int_{|x|\gtrsim r^{1-\frac{\gamma}{2}}} u^2 e^{\frac{|x|^2}{4t}} &\lesssim \|u\|_{L^\infty(A_{r,2r})}^2 r^{n+2}e^{-r^{-\gamma}}\leq C(n,p,\|u\|_{L^2(Q_1)}, \|f\|_{L^p(Q_1)}) r^{n+2}e^{-r^{-\gamma}}
\end{align*}
and 
$$\int_{-(2r)^2}^{-r^2}\int u^2e^{\frac{|x|^2}{4t}}\geq r^{n+2+2\kappa+\epsilon}$$ for $0<r<r_\epsilon\ll 1$,
which follows from the growth estimate in Remark \ref{growth}, we conclude that for $r\ll 1$, 
\begin{align}\label{eq:est2}
\int_{-(2r)^2}^{-r^2}\int_{|x|\lesssim r^{1-\frac{\gamma}{2}}} u^2 e^{\frac{|x|^2}{4t}}\gtrsim \int_{-(2r)^2}^{-r^2}\int_{\R^n} u^2 e^{\frac{|x|^2}{4t}}.
\end{align}
Combining \eqref{eq:est1} and \eqref{eq:est2} together then yield the desired inequality. This completes the proof for the claim.

Now we can choose $r_2=r_2(n,p,\epsilon, u)$ sufficiently small such that 
\begin{align*}
&\quad\tau^{\frac{1}{2}}(r_2)^{-1} |\ln (r_2)|^{-1}e^{\tau\phi(-r_2^2)} \|ue^{(1+Cr_2^\gamma)\frac{|x|^2}{8t}}\|_{L^2(A_{r_2,2r_2})}\\
&\ge r_2^{-\kappa-\frac{n+2}{2}-\frac{\epsilon}{2}}\|ue^{\frac{|x|^2}{8t}}\|_{L^2(A_{r_2,2r_2})}\geq 2M,
\end{align*}
		where the first inequality is true from \eqref{eq:asymp}, our choice of $\tau$ and the claim,  and the second inequality is possible by the lower bound in Remark \ref{growth}. Combining this with \eqref{eq:upp_sem} and the upper bound for $X$ we get
\begin{align*}
M\leq C( r_1)^{-1} e^{\tau\phi(-r_1^2)}\|u e^{\frac{\sigma(x-x_0;x_0,t_0)}{8(t-t_0)}}\|_{L^2(A_{r_1,2r_1}(x_0,t_0))}.
\end{align*}
Taking the logarithm both sides and dividing by $\operatorname{ln}r_1$, we get 
\begin{align*}
\frac{\operatorname{ln}(C/ M)}{\operatorname{ln}r_1} 
\ge -1+ \frac{\tau\phi(-r_1^2)}{\operatorname{ln}r_1} +\frac{\ln \|u e^{\frac{\sigma(x-x_0;x_0,t_0)}{8(t-t_0)}}\|_{L^2(A_{r_1,2r_1}(x_0,t_0))}}{\operatorname{ln}r_1}.
\end{align*}
We now take $r_1\rightarrow 0+$ and use \eqref{eq:asymp} to find
$\tau (2+c_0\pi) +1 \ge \kappa_{(x_0,t_0)}+\frac{n+2}{2}$.
From the definition of $\tau$ we can re-write the above inequality as 
\begin{align*}
	\kappa_{(0,0)}+\epsilon \ge \kappa_{(x_0,t_0)}.
\end{align*}
Since this is true for all $(x_0,t_0) \in Q_{r_2^2/16}$ and $\epsilon \in (0,2-\frac{n+2}{p}-\kappa)$ is arbitrary, we complete the proof of the lemma.
\end{proof}

The upper semi-continuity of the vanishing order immediately gives a uniform lower bound for the growth of solutions around points varying in a compact set:

\begin{cor}\label{cor:lower_bound}
		Assume that $u\in \mathcal{G}^{A,f}_p(S_1)$. Let $K\subset S_1$ be a compact set  (in the topology generated by the parabolic cylinders $\{B_r(x)\times (t-r^2, t]\cap S_1\}_{(x,t)\in \R^n\times \R}$) with $\kappa_{(x,t)} < \bar{\kappa}$ for all $(x,t) \in K$, for some  $\bar{\kappa} \in (0,2-\frac{n+2}{p})$. Then there exists a neighborhood $U$ of compact set $K$ and $r_0=r_0(K,\bar{\kappa},u)>0$ such that for all $(x_0,t_0) \in U$ the following holds 
		\begin{align*}
			 r^{-\frac{n+2}{2}}\|ue^{\frac{\sigma(x-x_0;x_0,t_0)}{8(t-t_0)}}\|_{L^2(A_{r,2r}(x_0,t_0))} \ge r^{\bar{\kappa}} \hspace{3mm}\text{for all} \hspace{3mm} 0<r<r_0.
		\end{align*}
		If $f$ satisfies \eqref{eq:support}, then $\bar \kappa $ can be taken in $(0,\infty)$. 
\end{cor} 

The proof of the above corollary is straightforward, using Remark \ref{growth}, Lemma \ref{lem:upper_semi}, and the basic properties of compact sets. It is similar to the elliptic case (see \cite[Corollary 4.6]{KRS16}), so we skip the details here.

 While the constant in the lower bound in Corollary \ref{cor:lower_bound} depends on the solution $u$, for the upper bound the constant can be made uniform in $u$. This is shown in the following lemma:
 
\begin{lem}
\label{lem:growth}
	Let $u\in \mathcal{G}^{A,f}_p(S_1)$. Let $R_0$ be the radius in Corollary \ref{3sphere}. Let $\bar \kappa:=2-\frac{n+2}{p}$ for general $f\in L^p(S_1^+)$, and $\bar\kappa$ be an arbitrary number in $(0,\infty)$ if $f$ additionally satisfies the support condition \eqref{eq:support}. Then there exists a constant $C=C(n, p, \bar \kappa, \|u\|_{L^2(S_1)}, \|f\|_{L^p(S_1)},\|a^{ij}\|_{W^{1,1}_p(S_1)})$ such that the following statements hold:
	\begin{itemize}
	\item [(i)] For all $(x_0,t_0) \in Q_{1/4}$ and for all $0<r<R_0/4$
	\begin{align*}
		r^{-\frac{n+2}{2}}\|ue^{\frac{\sigma(x-x_0;x_0,t_0)}{8(t-t_0)}}\|_{L^2(A_{r,2r}(x_0,t_0))} \le C |\operatorname{ln}{(r)}|^2r^{\min\{\kappa_{(x_0,t_0)}, \bar \kappa\}}.
	\end{align*}
	\item [(ii)] Suppose $\kappa_{(x_0,t_0)}<\bar \kappa$. Then there is $r_u =r(u, n, \bar\kappa, (x_0,t_0), R_0)\in (0,1/4)$ such that for all $r\in (0,r_u)$ and $\delta\in (0,1/2)$ we have
\begin{align*}
\frac{\|ue^{\frac{\sigma(x-x_0;x_0,t_0)}{8(t-t_0)}}\|_{L^2(A_{\delta r,2\delta r}(x_0,t_0))}}{\|ue^{\frac{\sigma(x-x_0;x_0,t_0)}{8(t-t_0)}}\|_{L^2(A_{r,2r}(x_0,t_0))}}\leq C |\ln\delta|^2\delta^{\kappa_{(x_0,t_0)}+\frac{n+2}{2}}.
\end{align*}
\end{itemize}
\end{lem}
\begin{proof}
	Without loss of generality we shall assume $(x_0,t_0)=(0,0)$ and denote $\kappa:=\kappa_{(0,0)}$.  Assume that $\kappa<\bar \kappa$. From Remark \ref{growth}, given $\epsilon_{j} >0$ with $\epsilon_j\rightarrow 0$,  there exists $r_j$ with $r_j \rightarrow 0$ such that
	\begin{align}\label{eq:rj}
		r_j^{-(\kappa+\frac{n+2}{2}-\epsilon_j)}\|ue^{\frac{|x|^2}{8t}}\|_{L^2(A_{r_j,2r_j})}\rightarrow 0 \text{ as } j\rightarrow \infty.
	\end{align}
 Define $\tau_j:=\left(\kappa +\frac{n}{2}-\epsilon_j\right)\frac{-\operatorname{ln}r_j}{\tilde{\phi}(2|\operatorname{ln} (r_j)|)}.$ Then we have 
 \begin{align}\label{eq:tauj}
 	(r_j)^{-1} e^{\tau_j\tilde{\phi}(2|\operatorname{ln (r_j)}|)}=r_j^{-(\kappa+\frac{n+2}{2}-\epsilon_j)}.
 	\end{align}
 We now apply the strengthened three sphere's inequality \eqref{eq:three_sphere_strong} with radii $r_1=r_j,$ $r_2=\delta r$ with $\delta\in (0,1/2)$, $r_3=r$ and $\tau=\tau_j$ to obtain 
 \begin{equation}\label{eq:upp_bd1}
 		\begin{split}
 		&\quad (\delta r)^{-1} |\ln \delta |^{-1}e^{\tau_j\phi(-(\delta r)^2)}\|ue^{\frac{|x|^2}{8t}}\|_{L^2(A_{\delta r,2\delta r})}\\
 		&\leq C\left( (r_j)^{-1} e^{\tau_j\phi(-r_j^2)}\|u e^{\frac{|x|^2}{8t}}\|_{L^2(A_{r_j,2r_j})}+r^{-1} e^{\tau_j\phi(-r^2)}\|u e^{\frac{|x|^2}{8t}}\|_{L^2(A_{r,2r})}\right.\\
 		&\left.\quad +\|(-t)^{\frac{1}{2}}e^{\tau_j\phi(t)}e^{\frac{|x|^2}{8t}}f\|_{L^2(A_{r_j,2r})}\right).
 	\end{split}
 \end{equation}
\emph{Proof for (i):} By \eqref{eq:rj} and \eqref{eq:tauj}, the first term on the above right hand side goes to zero as $j\rightarrow \infty$. We apply H\"older's inequality to the third term and find that it is bounded by $C_{\bar\kappa}\|f\|_{L^p(Q_{2r})}$ with our choice of $\tau_j$.  In view of \eqref{eq:asymp}, $\tau_j \rightarrow \tau_\infty:=\frac{\kappa +n/2}{2+c_0\pi}$ as $j\rightarrow \infty.$ Therefore, by choosing $r=R_0/2$ and since $\delta\in (0,1/2)$ is arbitrary, we get from \eqref{eq:upp_bd1} that 
 \begin{equation*}
	\begin{split}
		s^{-1} |\ln s|^{-1}e^{\tau_\infty\phi(-s^2)}\|ue^{\frac{|x|^2}{8t}}\|_{L^2(A_{s,2s})}\leq C, \quad \forall 0<s<R_0/4,
	\end{split}
\end{equation*}
where $C=C(n,p,R_0, \|u\|_{L^2(Q_{R_0})}, \|f\|_{L^p(Q_{R_0})}, \|a^{ij}\|_{W^{1,1}_p})$. 
Recalling that $\phi(-r^2)=\tilde\phi(2|\ln r|)$ and the asymptotics of $\tilde\phi$ 
\begin{align*}
\tilde\phi(s)= (1+\frac{c_0\pi}{2})s - c_0-\frac{c_0}{2}\ln(1+s^2)+ o(s^{-1})\ \text{ as } s\rightarrow \infty,
\end{align*}
we get  
\begin{align*}
		& |\ln s|^{-1}s^{-\kappa- \frac{n+2}{2}}(1+4|\operatorname{ln}s|^2)^{-\frac{c_0(\kappa+n/2)}{4+2c_0\pi}}\|ue^{\frac{|x|^2}{8t}}\|_{L^2(A_{s,2s})}\leq C. 
\end{align*}
Since $\kappa <\bar\kappa$, by choosing $c_0=c_0(n,p,\bar\kappa)$ small enough a priori, we can ensure
$\frac{c_0(\kappa +n/2)}{4+2c_0\pi}<\frac{1}{2}$.
Consequently, we obtain
\begin{align*}
	s^{- \frac{n+2}{2}} \|ue^{\frac{|x|^2}{8t}}\|_{L^2(A_{s,2s})}\leq C_{\bar\kappa} s^{\kappa}|\ln s|^2.
\end{align*}
If $\kappa\geq \bar\kappa$, by replacing $\kappa$ by $\bar\kappa$ in the above argument, we obtain the desired growth estimate. 

\medskip

\emph{Proof for (ii):} By H\"older's inequality, if $\kappa<\bar\kappa = 2-\frac{n+2}{p}$, the third term on the right hand side  of \eqref{eq:upp_bd1} is actually bounded from above by $C(n,p,\epsilon_0)r^{\epsilon_0}\|f\|_{L^p(S_{2r})}$ for some $\epsilon_0\in (0,2-\frac{n+2}{p}-\kappa)$ fixed. In view of Remark \ref{growth}, if $r<r_0$ for some $r_0=r_0((x_0,t_0), u, n, p, \|f\|_{L^p}, \|u\|_{L^2})$, we have
 \begin{align*}
r^{-1} e^{\tau_\infty \phi(-r^2)}\|u e^{\frac{|x|^2}{8t}}\|_{L^2(A_{r,2r})}\geq r^{\epsilon_0/2}\geq C(n,p,\epsilon_0)r^{\epsilon_0}\|f\|_{L^p(Q_{2r})}.
\end{align*}
Thus we get from \eqref{eq:rj} and \eqref{eq:upp_bd1} (in the limit when $j\rightarrow \infty$) that 
\begin{equation*}
	\begin{split}
		 \|ue^{\frac{|x|^2}{8t}}\|_{L^2(A_{\delta r,2 \delta r})}
		\leq C \delta |\ln \delta| e^{\tau_\infty\phi(-r^2)-\tau_\infty\phi(-(\delta r)^2)}\|u e^{\frac{|x|^2}{8t}}\|_{L^2(A_{r,2r})}.
	\end{split}
\end{equation*}
From the asymptotics of $\phi(-r^2)$ with $r\ll 1$ we have
\begin{align*}
e^{\tau_\infty \phi(-r^2)} = e^{-\tau_\infty c_0}r^{-\tau_\infty(2+c_0\pi)}r^{o\left(\frac{1}{\operatorname{ln}(r)}\right)}(1+4|\ln r|^2)^{-\frac{\tau_\infty c_0}{2}}.
\end{align*}
Clearly, $r^{o\left(\frac{1}{\operatorname{ln}(r)}\right)} \backsim 1.$ Thus, after a few simplifications and using the explicit expression of $\tau_\infty$ we get
\begin{align}\label{eq:asymp_phi}
e^{\tau_\infty\phi(-r^2)-\tau_\infty\phi(-(\delta r)^2)}\leq C \delta^{\kappa+\frac{n}{2}} E(r,\delta)^{-1},
\end{align}
 where 
 \begin{equation}\label{eq:def_E}
\begin{split}
 E(r,\delta)&:=(1+4|\ln r|^2)^{\frac{\tau_\infty c_0}{2}}(1+4|\ln (\delta r)|^2)^{-\frac{\tau_\infty c_0}{2}}
\end{split}
 \end{equation}
Choosing $c_0=c_0(n,p)$ small such that $c_0\tau_\infty<1/2$, we have that $E(r,\delta)^{-1}\lesssim 1+ \frac{|\ln\delta|}{|\ln r|}$. Thus 
\begin{equation*}
	\begin{split}
		\|ue^{\frac{|x|^2}{8t}}\|_{L^2(A_{\delta r,2 \delta r})}
		\leq C  |\ln \delta|^2 \delta^{\kappa+\frac{n+2}{2}} \|u e^{\frac{|x|^2}{8t}}\|_{L^2(A_{r,2r})}.
	\end{split}
\end{equation*}
When $f$ satisfies \eqref{eq:support}, for all $\kappa\leq \bar \kappa\in (0,\infty)$ and some fixed $\epsilon_0\in (0,1)$, we have 
\begin{align*}
\|(-t)^{\frac{1}{2}}e^{\tau_j\phi(t)}e^{\frac{|x|^2}{8t}}f\|_{L^2(A_{r_j,2r})}\leq C_{\bar \kappa} r^{\epsilon_0}\|f\|_{L^p(S_{2r})},
\end{align*}
for all  $0<r<r_{\bar \kappa}$. Then arguing as above we obtain the desired inequality. 
\end{proof}

In the next lemma we derive a doubling inequality, which will play a crucial role in the blow-up analysis:

\begin{lem}[Doubling]\label{doubling}
Let $u\in \mathcal{G}^{A,f}_p(S_1)$. Let $(x_0,t_0) \in Q_{1/4}$ with $\kappa_{(x_0,t_0)}<\bar\kappa$, where $\bar\kappa=2-\frac{n+2}{p}$, and $\kappa\in (0,\infty)$ if $f$ additionally satisfies the support condition \eqref{eq:support}. Then given $\epsilon >0$, there exist $r_{\epsilon}=r_{\epsilon}(\epsilon,\bar\kappa,(x_0,t_0),u,n,p, \|f\|_{L^p})$ and $C=C(n,p,\bar \kappa, \|f\|_{L^p},\|u\|_{L^2})$, such that for all $0<r<r_{\epsilon}$ and $0<\sigma<\frac{1}{8}$ we have
	\begin{align*}
\|ue^{\frac{\sigma(x-x_0;x_0,t_0)}{8(t-t_0)}}\|_{L^2(S_r(x_0,t_0))} \leq C   \sigma^{-\kappa_{(x_0,t_0)}-\frac{n+2}{2}-\epsilon} \|ue^{\frac{\sigma(x-x_0;x_0,t_0)}{8(t-t_0)}}\|_{L^2(S_{\sigma r}(x_0,t_0))}.
	\end{align*}
	If $f$ satisfies the support condition \eqref{eq:support}, then for any $\bar \kappa\in (0,\infty)$, the above inequality holds at any $(x_0,t_0)$ such that $\kappa_{(x_0,t_0)}\in (0,\bar\kappa)$, where the constants $C$ and $r_\epsilon$ depend additionally on $\bar\kappa$.
\end{lem}
\begin{proof}
	Without loss of generality we assume that $(x_0,t_0)=(0,0)$ with $\kappa:=\kappa_{(0,0)}<2-\frac{n+2}{p}$. First note that, for a given $\epsilon >0$, by Remark \ref{growth} we have
\begin{align}\label{eq:growth}
r^{\kappa+\frac{n+2}{2}+\frac{\epsilon}{4}} \leq
 \|ue^{\frac{|x|^2}{8t}}\|_{L^2(A_{r,2r})} \leq r^{\kappa+\frac{n+2}{2}-\frac{\epsilon}{4}}\hspace{3mm}\text{for all}\hspace{3mm}0<r<r_{\epsilon}.
 \end{align}  
We use the strengthened three sphere's inequality \eqref{eq:three_sphere_strong} with radii $r_2=r<r_{\epsilon}$, $r_1=s<r/2$ and $r_3=R_0/4$, and weighting factor
$ \tau = \frac{\kappa +n/2+ \epsilon/2}{2+c_0\pi}$ 
 to find
 \begin{equation}\label{eq:3s_app}
\begin{split}
&\quad  
(\ln(r/s))^{-1}r^{-1}e^{\tau\phi(-r^2)}||ue^{\frac{|x|^2}{8t}}||_{L^2(A_{r,2r})}\\
&\leq C\left( s^{-1} e^{\tau\phi(-s^2)}\|u e^{\frac{|x|^2}{8t}}\|_{L^2(A_{s,2s})}
+(R_0)^{-1} e^{\tau\phi(-R_0^2/16)}\|u e^{\frac{|x|^2}{8t}}\|_{L^2(A_{R_0/4,R_0/2})}\right.\\
&\;\; \left.\quad + \|(-t)^{\frac{1}{2}} e^{\tau\phi(t)}e^{\frac{|x|^2}{8t}} f  \|_{L^2(A_{s,R_0/2})}\right).
			\end{split}
		\end{equation}
	Note that if $\frac{\epsilon}{2}\in (0,2-\frac{n+2}{p}-\kappa)$, the second and third term on the right hand side of  \eqref{eq:3s_app} can be bounded from above by $M=M(n,p,\|f\|_{L^p(S_{R_0})}, \|u\|_{L^2(S_{R_0})})$ (cf. the proof for Lemma \ref{lem:upper_semi} for a similar estimate). 
Further, we claim that by choosing $0<r<r_\epsilon$ with $r_\epsilon$ sufficiently small, the left hand side of \eqref{eq:3s_app} has the lower bound:
\begin{align*}
(\ln(r/s))^{-1}r^{-1}e^{\tau\phi(-r^2)}\|ue^{\frac{|x|^2}{8t}}\|_{L^2(A_{r,2r})}\geq 2M.
\end{align*}
To show this, we note that from the asymptotics of $\tilde{\phi}$ in \eqref{eq:asymp} 
\begin{align*}
\frac{\tau\tilde\phi(2|\ln r|)}{|\ln r|}> -\frac{\epsilon}{8}+ \tau(2+c_0\pi),\quad \forall r\in (0,r_\epsilon),
\end{align*}
for some $r_\epsilon$ sufficiently small. Consequently, we find
 $
 e^{\tau\phi(-r^2)} \ge r^{-\tau(2+c_0\pi)+\frac{\epsilon}{8}}$ for $0<r<r_\epsilon$.
Using the lower bound in the growth estimate \eqref{eq:growth},  the definition of $\tau$ and after possibly shrinking $r_\epsilon$, we obtain
\begin{align*}
(\ln(r/s))^{-1}r^{-1}e^{\tau\phi(-r^2)}\|ue^{\frac{|x|^2}{8t}}\|_{L^2(A_{r,2r})}&\geq (\ln(r/s))^{-1}r^{-1}e^{\tau\phi(-r^2)}r^{\kappa+\frac{n}{2}+1+\frac{\epsilon}{4}} \\
&\geq |\ln r|^{-1}r^{-\frac{\epsilon}{8}}\geq 2M.
\end{align*}
This combined with \eqref{eq:3s_app} gives
\begin{equation*}
\begin{split}
\|ue^{\frac{|x|^2}{8t}}\|_{L^2(A_{r,2r})}
\leq C \operatorname{ln}\left(\frac{r}{s}\right)  \left(\frac{r}{s}\right) e^{\tau\phi(-s^2)-\tau\phi(-r^2)}\|u e^{\frac{|x|^2}{8t}}\|_{L^2(A_{s,2s})}.
	\end{split}
\end{equation*}
Writing $s=\sigma r$ and invoking \eqref{eq:asymp_phi} we get
\begin{equation}\label{eq:3sph_doubling21}
\|ue^{\frac{|x|^2}{8t}}\|_{L^2(A_{r,2r})}
 	 \leq C |\operatorname{ln}\sigma| \sigma^{-\kappa-\frac{n+2}{2}-\frac{\epsilon}{2}}E(r,\sigma)\|ue^{\frac{|x|^2}{8t}}\|_{L^2(A_{\sigma r, 2\sigma r})},
\end{equation} 
where $E(r,\sigma)$ is defined in \eqref{eq:def_E}. It is not hard to see from the explicit expression  that $E(r,\sigma)\leq 1$. Thus we obtained the doubling inequality for $\|u e^{\frac{|x|^2}{8t}}\|_{L^2(A_{r,2r})}$.
To replace $A_{r,2r}$ and  $A_{\sigma r,2\sigma r}$ by $S_r$ and $S_{\sigma r}$ respectively, we apply \eqref{eq:3sph_doubling21} with $2^{-k}r$ and sum up in $k\in \N$. 
%
	%
%

When $f$ satisfies the support condition \eqref{eq:support}, the second and third term of \eqref{eq:3s_app} is bounded by $M=M(n,p,\|f\|_{L^p}, \|u\|_{L^2}, \bar \kappa)$ for $\kappa\leq \bar \kappa$. Thus the doubling inequality holds at all points with $\kappa_{(x_0,t_0)}\leq \bar \kappa$, with the constant $C$ and $r_\epsilon$ additionally depending on $\bar \kappa$. 
\end{proof}

An important consequence of the doubling inequality is the compactness of the blow-up sequences:

\begin{prop}\label{prop:compactness}
	Let $u\in \mathcal{G}^{A,f}_p(S_1)$, $p>n+2$, be a non-trivial solution. Let $(x_0,t_0)\in \Gamma_u^\ast\cap Q_{1/4}$ with $\kappa_{(x_0,t_0)}<2-\frac{n+2}{p}$,  or $\kappa_{(x_0,t_0)}<\infty$ if $f$ additionally satisfies \eqref{eq:support}. For $r>0$ define the rescalings 
\begin{align}\label{eq:u_r}
u_r^{(x_0,t_0)}(x,t) =\frac{u(x_0+r\sqrt{A(x_0,t_0)}x,t_0+r^2t)}{\left(H_r^{(x_0,t_0)}(u)\right)^{1/2}},
\end{align}
where
\begin{align*}
H_r^{(x_0,t_0)}(u)&:=\frac{1}{r^2}\int_{A_{\frac{r}{2},r}(x_0,t_0)} u(x,t)^2 G_{(x_0,t_0)}(x,t) \ dxdt
\end{align*}
and $G_{(x_0,t_0)}$ is the parametrix defined in \eqref{eq:parametrix}.		 Then there exists a sequence $r_j \rightarrow 0$ and $u_0^{(x_0,t_0)}:\R^n \times (-\infty,0]\rightarrow \R$, even about $x_n$, such that 
\begin{align*}
\int_{S_R} |u_{r_j}^{(x_0,t_0)}-u_0^{(x_0,t_0)}|^2 G\ dxdt \rightarrow 0,\\
\int_{S_R}|t||\nabla (u_{r_j}^{(x_0,t_0)}-u_0^{(x_0,t_0)})|^2 G\ dxdt \rightarrow 0
\end{align*}
for all $R\geq 1$. Moreover, $u_0^{(x_0,t_0)}\in W^{2,1}_2(Q_R)$ for each $R>0$ and $u_0^{(x_0,t_0)}$ is a nontrivial solution to the Signorini problem:
	 \begin{align*}
		 	\p_t v- \Delta v= 0 &\text{ in } \R^n \times (-\infty,0],\\
		 	v\geq 0, \quad -\p_n v \geq 0,\quad v \p_n v =0 &\text{ on } \R^{n-1}\times \{0\} \times (-\infty,0],
		 \end{align*}
in the sense that it solves the Signorini problem in every $Q_R$. The point $(0,0)\in \Gamma_{u_0^{(x_0,t_0)}}^\ast$ and the vanishing order of $u_0^{(x_0,t_0)}$ at $(0,0)$ is $\kappa_{(x_0,t_0)}$.
\end{prop}
\begin{proof}
The proof follows along the same lines as Theorem 7.3 in \cite{DGPT}. For completeness we provide an  outline of the proof here. For simplicity we assume that $(x_0,t_0)=(0,0)$ and furthermore we drop the dependence in the notation on $(0,0)$ and simply write $u_r$, $H_r(u)$. 
Clearly, $u_r$ solves the following:
\begin{equation}
	\begin{split}
		\p_t u_r - \p_i(a^{ij}(rx,r^2t)\p_j u_r)= f_r &\text{ in } B_{2/r}\times [-1/r^2,0],\\
		u_r\geq 0, \quad \p_\nu u_r \geq 0,\quad u_r\p_\nu u_r =0 &\text{ on } B'_{2/r}\times [-1/r^2,0],
	\end{split}
\end{equation}
where $f_r(x,t):=\frac{r^2f(rx,r^2t)}{H_r(u)^{1/2}}.$
Using Lemma \ref{lem:H2} with $s=0$, for all $0<R<(2r)^{-1}$ we obtain
\begin{equation}\label{eq:new_H2}
\begin{split}
&\int_{S_R} (u_r^2 +|t||\nabla u_r|^2 +|t|^2 |D^2u_r|^2 +|t|^2(\dt u_r)^2)G\\
 &\leq C \left(\int_{S_{2R}} |t|^2 f_r^2 G + \int_{S_{2R}}u_r^2G\right).
\end{split}
\end{equation}  
To bound the first term on the right hand side of \eqref{eq:new_H2}, we apply H\"older's inequality to get 
\begin{equation}\label{eq:inhomo}
\begin{split}
\int_{S_{2R}} |t|^2| f_r|^2 G &\leq \left(\int_{S_{2R}} |f_r|^p\right)^{\frac{2}{p}}\left( \int_{S_{2R}}(|t|^2 G)^{\frac{p}{p-2}}\right)^{\frac{p-2}{p}}\leq CR^{2(3-\frac{n+2}{p})}\|f_r\|_{L^p(S_{2R})}^2.
\end{split}
\end{equation}
Writing the above upper bound in terms of $f$ and using $2rR<1$ we get 
\begin{align*}
\int_{S_{2R}}|t|^2 |f_r|^2 G \leq C\frac{(rR)^{2(3-\frac{n+2}{p})}}{r^2 H_r(u)}\|f\|_{L^p(S_{1})}^2.
\end{align*}
From the lower bound in the growth estimate, cf. Remark \ref{growth}, given $\epsilon\in (0, 2-\frac{n+2}{p}-\kappa)$ fixed, there is $r_\epsilon=r_\epsilon(u)\ll 1$ such that $H_r(u) \geq r^{2(\kappa+\epsilon)}$ for all $r\in (0,r_\epsilon)$. Thus for all $r\in (0,r_\epsilon)$ we have
\begin{align*}
\int_{S_{2R}}|t|^2 |f_r|^2 G \leq Cr^{2(2-\frac{n+2}{p}-\kappa-\epsilon)} R^{2(3-\frac{n+2}{p})} \|f\|^2_{L^p(S_1)}\leq C(R).
\end{align*}
For the second term on the right hand side of \eqref{eq:new_H2} we apply the doubling inequality, cf. Lemma \ref{doubling}, to get for $r\leq r_u$,
\begin{align*}
\int_{S_{2R}} u_r^2 G \leq \frac{\int_{S_{2rR}} u^2 G }{\int_{S_{2r}} u^2 G} \leq C(R).
\end{align*}
Plugging in the above two inequalities to \eqref{eq:new_H2} we obtain the uniform Sobolev estimates for $u_r$: for all $0<r<r_\epsilon$,
\begin{align}\label{eq:uniform_H2}
\int_{S_R} (u_r^2 +|t||\nabla u_r|^2 +|t|^2 |D^2u_r|^2 +|t|^2(\dt u_r)^2)G &\leq C(R).
\end{align}
\emph{Claim:} There is a subsequence $u_{r_j}$ such that for all $R>0$,
\begin{align*}
\int_{S_R}|u_{r_j}-u_0|^2G \rightarrow 0
\end{align*}
In particular, $u_0$ is nontrivial. We note that in view of Lemma \ref{lem:energy_close}, the above claim also implies that 
\begin{align*}
\int_{S_R}|t||\nabla (u_{r_j}-u_0)|^2G \rightarrow 0,\quad \forall R>0.
\end{align*}
\emph{Proof of the claim:} By Lemma \ref{lem:growth}(ii), for any $r\in (0,r_u)$ with some $r_u$ sufficiently small and for any small $\delta>0$ fixed, we have
\begin{align*}
H_\delta(u_r)=\frac{H_{\delta r} (u)}{H_r(u)}\leq C |\ln \delta|^4 \delta ^{2\kappa}.
\end{align*}
This implies the smallness of the integral near $t=0$:
\begin{align*}
\int_{S_\delta} u_r^2 G\leq \sum_{i=0}^\infty (2^{-i}\delta)^2 H_{2^{-i}\delta}(u_r) \leq C\delta,\quad \forall 0<r<r_u.
\end{align*}
Here we have used that $\kappa\geq 0$. Next, we take $A=A(\delta,R)$ such that $\int_{\R^n \setminus B_{A-1}}G(x,s)dx \le e^{-1/\delta}$ for all $-R^2 <s<0.$ Let $\zeta_A \in C_c^{\infty}(B_A)$ be a smooth cut-off function such that $0 \le \zeta_A \le 1$ and $\zeta_A\equiv 1.$ in $B_{A-1}.$   Now, from an application of the log-Sobolev inequality to $u_r(1-\zeta_A)$, cf.  \cite[Lemma 7.7]{DGPT}, we find
\begin{align*}
\int_{S_R\cap \{|x|\geq A\}} u_r^2 G\leq \int_{S_R} u_r^2(1-\zeta_A)^2 G \le C\delta \int_{S_R} [u_r^2+|t||\nabla u_r|^2] G.
\end{align*} 
Finally, we use \eqref{eq:uniform_H2} to get 
\begin{align*}
\int_{S_R\cap \{|x|\geq A\}} u_r^2 G\leq \delta C(R), \quad \forall \  0<r<r_u.
\end{align*}
In $E:=E_{R,\delta, A}= B_A\times (-R^2,-\delta^2]$, the Gaussian is bounded from above and below, therefore \eqref{eq:uniform_H2} implies that $\{u_r\}_{0<r<r_u}$ is uniformly bounded in $W^{1,1}_2(E)$. Thus there is a subsequence $u_{r_j}=u_{r_j}^{\delta, A}$ such that $u_{r_j}\rightarrow u_0$ strongly in $L^2(E)$.
Let $\delta\rightarrow 0$ and $A\rightarrow \infty$, by a standard diagonal argument we obtain the desired sequence, which completes the proof of the claim.

Since $\nabla u_r$  is uniformly (in $r$) bounded in $H^{\alpha, \alpha/2}(B_R\times (-R^2, -\delta^2))$ for any fixed $\delta \in (0,R)$ and for some $\alpha\in (0,1)$ (this is due to the fact that $u_r$ is uniformly bounded in $W^{1,1}_2(B_R\times (-R^2,-\delta^2)$), and furthermore, $u_r$ satisfies the variational inequality 
\begin{align*}
\int_{t_1}^{t_2}\int_{B_R} a^{ij}(r\cdot, r^2\cdot) \p_i u_{r}\p_j (v-u_r) + \p_tu_r(v-u_r) \ dxdt \geq \int_{t_1}^{t_2}\int_{B_R}f_r (v-u_r) \ dxdt 
\end{align*}
for all $v\in \{ v\in W^{1,0}_2(B_R\times (-R^2,-\delta^2): v=u_r \text{ on } \p B_R\times (-R^2,-\delta^2), \ v\geq 0 \text{ on } B'_R\times (-R^2,-\delta^2)\},$ $v$ are even about $x_n$ and $-R^2<t_1<t_2<-\delta^2$. Passing to the limit along a possibly  further subsequence in the variational inequality we conclude that $u_0$ solves the constant coefficient Signorini problem in $B_R\times (-R^2, -\delta^2)$ for each $R$ and $\delta>0$. By the same argument as in \cite[Theorem 7.3(iii)]{DGPT} using the sub-mean value property (cf. Lemma \ref{lem:linfty}), we show that $u_0 \in L^\infty(Q_R)$ for each $R$, and thus by the energy estimate $u\in W^{2,1}_2(Q_R)$.

In the end, we show that the vanishing order of $u_0$ at $(0,0)$ is $\kappa$ and $(0,0)$ belongs to the extended free boundary of $u_0$. The vanishing order is preserved due to the refined growth estimate Lemma \ref{lem:growth}(ii) and the doubling inequality, cf. Lemma \ref{doubling}. More precisely, by Lemma \ref{lem:growth}(ii) and Lemma \ref{doubling} (or \eqref{eq:3sph_doubling21}), for given $\epsilon\in (0,1/8)$, there is $r_{u,\epsilon}=r(u,\epsilon, n, p, \|f\|_{L^p}, \kappa)>0$ such that for all $r\in (0,r_{u,\epsilon})$ and for all $\sigma\in (0,1/8)$ we have
\begin{align*}
C^{-1} \sigma^{\kappa+\epsilon}\leq \frac{\sigma^{-\frac{n+2}{2}}\|ue^{\frac{|x|^2}{8t}}\|_{L^2(A_{\sigma r, 2\sigma r})}}{\|ue^{\frac{|x|^2}{8t}}\|_{L^2(A_{\frac{r}{2}, r})}}\leq C |\ln \sigma|^2 \sigma ^\kappa.
\end{align*}
From the definition of $u_r$, cf. \eqref{eq:u_r}, we have
	\begin{equation}\label{eq:quotient_u}
	\|u_r e^{\frac{|x|^2}{8t}}\|_{L^2(A_{\sigma,2\sigma})}
= c_n\frac{\|ue^{\frac{|x|^2}{8t}}\|_{L^2(A_{\sigma r, 2\sigma r})}}{\|ue^{\frac{|x|^2}{8t}}\|_{L^2(A_{\frac{r}{2},r})}}.
\end{equation}
Therefore we obtain
	\begin{align*}
	C^{-1}   \sigma^{\kappa+\epsilon}\leq 	\sigma^{-\frac{n+2}{2}}\|u_re^{\frac{|x|^2}{8t}}\|_{L^2(A_{\sigma,2\sigma})} \leq C|\ln\sigma|^2\sigma^{\kappa} ,\quad \forall r\in (0,r_{u,\epsilon}), \sigma\in (0,\frac{1}{8}).
	\end{align*}
	By the strong $L^2$ convergence along a subsequence, the above inequality also holds for the limiting function $u_0$. Consequently, using $\epsilon>0$ is arbitrary, we get that the vanishing order of $u_0$ at $(0,0)$ is also $\kappa$. To see that $(0,0)\in \Gamma_{u_0}^\ast$, we first notice  that  $\kappa>1+\alpha$ for some $\alpha>0$ by  our assumption. This implies $u_0(0,0)=0$ and $\p_nu(0,0)=0$ (because otherwise either $\kappa=0$ or $\kappa=1$). Then it suffices to show $(0,0)$ is not in the interior of the set $\{(x',0,t)\in Q'_1:u_0(x',0,t)=0,\quad \p_n u_0(x',0,t)=0\}$. Assume by contradiction that if $(0,0)$ is an interior point, then $u_0=0$ in an open neighborhood of $(0,0)$ in $Q'_1$. Using that $u_0$ solves the heat equation in $\{x_n>0\}\times \R_-$, we have $D_x^\alpha u \p_t^k u_0(0,0)=0$ for all multi-index $\alpha$ and $k\in \N$, i.e. $u_0$ vanishes infinite order at $(0,0)$. This contradicts to our assumption that $\kappa<\infty$.
 This completes the proof of the proposition. 
\end{proof}

At the end of the section we  prove that if $p>2(n+2)$ or $p>n+2$ if $f$ satisfies \eqref{eq:support}, the lowest vanishing order at a free boundary point is $\frac{3}{2}$ and there is a gap around it. As a consequence we obtain an almost optimal growth estimate for $u$:
\begin{prop}\label{prop:almost_op}
Let $u\in\mathcal{G}^{A,f}_p(S_1)$ and let $(x_0,t_0)\in \Gamma_u^\ast\cap Q_{1/4}$. Then the following statements hold
\begin{itemize}
\item [(i)] If $p\in (2(n+2),\infty]$ and $\kappa_{(x_0,t_0)}< 2-\frac{n+2}{p}$, then $\kappa_{(x_0,t_0)}=\frac{3}{2}$. 
\item [(ii)] If $p\in (n+2,2(n+2)]$, then  $\kappa_{(x_0,t_0)}\geq 2-\frac{n+2}{p}$. 
\item [(iii)] If $p\in (n+2,\infty]$, $f$ satisfies \eqref{eq:support} and $\kappa_{(x_0,t_0)}<2$, then $\kappa_{(x_0,t_0)}=\frac{3}{2}$. 
\end{itemize}
As a consequence, there is a constant $C=C(n,p, \|u\|_{L^2}, \|f\|_{L^p})$ such that for all $(x_0,t_0)\in \Gamma_u^\ast\cap Q_{1/4}$  we have
\begin{align*}
r^{-\frac{n+2}{2}}\|u e^{\frac{\sigma(x-x_0;x_0,t_0)}{8(t-t_0)}}\|_{L^2(A_{r,2r}(x_0,t_0))}\leq C r^{\bar\kappa}|\ln r|^2, 
\end{align*}
where $\bar\kappa=\frac{3}{2}$ in case (i) and (iii), and $\bar \kappa=2-\frac{n+2}{p}$ in case (ii). 
\end{prop}
\begin{proof}
	We will show it for $(0,0)$ and for simplicity we write $\kappa:=\kappa_{(0,0)}$. 
	
	We first prove the case (i). The proof for case (iii) will be identical.   Assume that $\kappa<2-\frac{n+2}{p}$.  
	Let  $u_r$ be the blow-up sequence defined in Proposition \ref{prop:compactness}. 
By taking $r \rightarrow 0$ along converging subsequences, cf.  Proposition \ref{prop:compactness}, we find
that the blow-up limit  $u_0$ is a solution to the Signorini problem with constant coefficient, $(0,0)\in \Gamma_{u_0}^\ast$ and the vanishing order of $u_0$ at $(0,0)$ is $\kappa\in (1+\alpha, 2-\frac{n+2}{p})$.  Now by the gap of the vanishing order for the constant coefficient problem in \cite{DGPT} we conclude that $\kappa=\frac{3}{2}$. 

Next we prove the case (ii). Assume by contradiction that $\kappa<2-\frac{n+2}{p}$, then arguing as above using the frequency gap for the blow-up limits we conclude that $\kappa\geq \frac{3}{2}$. This is however a contradiction if $p\leq 2(n+2)$. 

With the estimate of the lowest vanishing order at hand, we conclude the growth estimate by Lemma \ref{lem:growth}(i). This completes the proof for the proposition.
\end{proof}

The almost optimal $L^2$-growth estimate combined with an $L^\infty-L^2$ estimate, cf. Lemma \ref{lem:linfty}, gives the almost optimal $L^\infty$-growth estimate of the solution. 
\begin{cor}\label{cor:al_opt_infty}
Let $u\in \mathcal{G}^{A,f}_p(S_1)$. Suppose that $p>2(n+2)$ (or $p>n+2$ when $f$ satisfies the support condition \eqref{eq:support}). Then there is a constant $C=C(n,p,\|u\|_{L^2(S_1)}, \|f\|_{L^p(S_1)},  \|a^{ij}\|_{W^{1,1}_p(S_1)})>0$ such that for all $(x_0,t_0)\in \Gamma_u^\ast\cap Q_{1/4}$ and for all $r\in (0,\frac{1}{4})$ we have
\begin{align*}
\sup_{B^+_r(x_0)\times (t_0-r^2, t_0+r^2)\cap S_1}|u|\leq Cr^{\frac{3}{2}}|\ln r|^2.
\end{align*}
When $p\in (n+2, 2(n+2)]$, we obtain for $\gamma=1-\frac{n+2}{p}$,
\begin{align*}
\sup_{B_r^+(x_0)\times (t_0-r^2, t_0+r^2)\cap S_1}|u|\leq Cr^{1+\gamma}|\ln r|^2.
\end{align*}
\end{cor}

\begin{rmk}\label{rmk:opt_h}
When $p\in (n+2, 2(n+2)]$, the above growth estimate at free boundaries together with the interior regularity estimate for the linear problem directly yields the almost optimal regularity of the solution $u\in \mathcal{G}^{A,f}_p(S_1^+)$ claimed in Theorem \ref{thm:opt}(ii). The proof is identical with that of \cite[Theorem 9.1]{DGPT}, and we do not repeat here.
\end{rmk}

\section{Optimal regularity}\label{sec:opt}
In this section we derive the optimal regularity of the solution. Without loss of generality throughout this section we assume that $(0,0)\in \Gamma_u^\ast$ and let $\kappa:=\kappa_{(0,0)}$. The arguments in this section work for other free boundary points as well, after making the change of variable in Remark \ref{rmk:sigma}. First we consider a one-parameter family of Weiss type energies for $u$ centered at $(0,0)$: for any $r\in (0,1)$,
\begin{align*}
W^{\kappa}_u(r):=\frac{1}{r^{2\kappa+2}}\int_{A^+_{r/2,r}}\left(|t||\nabla u|^2 -\frac{\kappa}{2}u^2\right)  G\ dxdt,
\end{align*}
where $G$ is the standard Gaussian, i.e.
\begin{align*}
G(x,t):=\frac{1}{(-4\pi t)^{\frac{n}{2}}}e^{\frac{|x|^2}{4t}},\text{ if } t\leq 0, \quad G(x,t)=0\text{ if } t>0.
\end{align*}
We note that $W^{\kappa}_u(r)$ is well-defined for all $r\in (0,1)$ and furthermore, $r\mapsto W^{\kappa}_u(r)$ is continuous. We also note the following scaling property: let
\begin{align*}
u_{\kappa,r}(x):=\frac{u(rx,r^2 t)}{r^\kappa}, \quad r>0,
\end{align*}
then
\begin{align*}
W^{\kappa}_{u}(r)=W^{\kappa}_{u_{\kappa,r}}(1).
\end{align*}

\subsection{Properties of the Weiss energy} In the following proposition we provide another expression of the Weiss energy:

\begin{prop}\label{prop:Weiss2}
Let $u\in \mathcal{G}^{A,f}_p(S_1^+)$ and let $g:=\p_ia^{ij}\p_j u + (a^{ij}-\delta^{ij})\p_{ij}u + f$. Then $W^{\kappa}_u(r)$ for $r\in (0,1)$ can also be written as 
\begin{align}\label{eq:Weiss_2}
W^{\kappa}_u(r)=\frac{1}{2 r^{2\kappa+2}}\int_{A^+_{r/2,r}} u(Zu-\kappa u)G\ dxdt -\frac{1}{r^{2\kappa+2}} \int_{A^+_{r/2,r}}t u g G\ dxdt,
\end{align}
where $Z:=x\cdot \nabla + 2t\p_t$ is the vector field associated with the parabolic dilation. If (i) $p>2(n+2)$ and $\kappa=\frac{3}{2}$, or (ii) $p>n+2$, $f$ satisfies the support condition \eqref{eq:support} and $\kappa\leq \bar\kappa$,  then there exists a constant $C_0=C_0(p,n,\|f\|_{L^p},\|a^{ij}\|_{W^{1,1}_p}, \|u\|_{L^2}, \bar\kappa)$ and $\sigma=\sigma(n,p)\in (0,1)$ such that
\begin{align*}
\left|W^{\kappa}_u(r)-\frac{1}{2 r^{2\kappa+2}}\int_{A^+_{r/2,r}} u(Zu-\kappa u)G\ dxdt\right|\leq C_0r^{\sigma}.
\end{align*}
\end{prop}
\begin{proof} 
First using the equation for $u_{\kappa,r}$ in $\R^n_+$
\begin{align}\label{eq:Delta_u}
\p_tu_{\kappa,r} -\Delta u_{\kappa,r} = g_{\kappa,r},\quad g_{\kappa,r}(x,t):= r^{2-\kappa}g(rx,r^2t)
\end{align}
 we write 
\begin{align*}
&\int_{A^+_{1/2,1}} u_{\kappa, r}\left( Zu_{\kappa,r}-\kappa u_{\kappa,r}\right) G =\int_{A^+_{1/2,1}}u_{\kappa, r} \left(x\cdot \nabla u_{\kappa,r}+2t(\Delta u_{\kappa,r} + g_{\kappa,r})\right)G -\int_{A^+_{1/2,1}} \kappa (u_{\kappa, r})^2 G.
\end{align*}
For the term involving $\Delta u_{\kappa,r}$, an integration by parts in $x$ together with the complementary condition at $\{x_n=0\}$ yields
\begin{align*}
\int_{A^+_{1/2,1}}2tu_{\kappa,r}\Delta u_{\kappa,r} G &=\int_{A^+_{1/2,1}} (-2t)|\nabla u_{\kappa,r}|^2 G - \int_{A^+_{1/2,1}}u_{\kappa,r}(x\cdot \nabla u_{\kappa, r}) G.
\end{align*}
Thus we obtain
\begin{align*}
 \int_{A^+_{1/2,1}} u_{\kappa, r}\left( Zu_{\kappa,r}-\kappa u_{\kappa,r}\right) G  &= \int_{A^+_{1/2,1}} \left(2|t| |\nabla u_{\kappa, r}|^2 -\kappa (u_{\kappa,r})^2 \right)G+ \int_{A^+_{1/2,1}}2tu_{\kappa, r} g_{\kappa,r} G\\
&=2W^{\kappa}_u(r)+\int_{A^+_{1/2,1}}2tu_{\kappa, r} g_{\kappa,r} G.
\end{align*}
Rewriting the above identity in terms of $u$ we get the claimed identity \eqref{eq:Weiss_2}.

Now we estimate the error term involving the the function $g$. Recalling that $g=\p_i((a^{ij}-\delta^{ij})\p_ju) +f$ and that $|a^{ij}-\delta^{ij}|\leq C(n,\|a^{ij}\|_{W^{1,1}_p}) (|x|+\sqrt{|t|})^\gamma$ with $\gamma=1-\frac{n+2}{p}$, we obtain from an integration by parts, Young's inequality and Cauchy-Schwarz inequality, that
\begin{equation}\label{eq:ug}
\begin{split}
&\quad \left|\int_{A^+_{r/2,r}}2t ug G \right|= 2\left|\int_{A^+_{r/2,r}}  (\delta^{ij}-a^{ij})(t\p_i u+\frac{1}{2}x_i u) \p_j u G + t uf G\right| \\
&\leq C r^\gamma\int\omega^\gamma \left(|t||\nabla u|^2  + \frac{|x|^2}{|t|}u^2 \right) G + \left(\int u^2 G\right)^{\frac{1}{2}}\left(\int |t|^2 |f|^2 G \right)^{\frac{1}{2}},
\end{split}
\end{equation}
where the integral is taken over ${A^+_{r/2,r}}$ and  $\omega=\omega(x,t)=1+\frac{|x|}{\sqrt{-t}}$. Using \cite[Claim A.1]{DGPT} and the $H^2$-estimate in Lemma \ref{lem:H2}, we get
\begin{align}\label{eq:e1}
\int_{A^+_{r/2,r}}\omega^\gamma \left(|t||\nabla u|^2  + \frac{|x|^2}{|t|}u^2 \right) G \lesssim \int_{S^+_{2r}} \omega^\gamma\left( u^2 + |t|^2 f^2\right)G.
\end{align}
To see \eqref{eq:e1}, we first apply \cite[Claim A.1]{DGPT} with $v=\omega^{\gamma/2}u$ to get
\begin{align*}
\int_{A^+_{r/2,r}} \omega^\gamma \frac{|x|^2}{|t|} u^2 G\lesssim \int_{S^+_{r}} \left(\omega^\gamma |t||\nabla u|^2 + \omega^{\gamma-2} |u|^2\right) G\leq \int_{S^+_{r}} \omega^\gamma\left( |t||\nabla u|^2 +|u|^2\right) G,
\end{align*}
where in the second inequality we have used that $\omega^{\gamma-2}\leq\omega^\gamma$ for $\gamma\in (0,1)$. Now we estimate the gradient term by Lemma \ref{lem:H2} to get \eqref{eq:e1}.
To estimate the right hand side of \eqref{eq:e1} we note the following two estimates: for all $s\in [0,2]$ and $\epsilon_0\in (0,\frac{1}{4}]$,
\begin{align}\label{eq:est_f}
\int_{S^+_{2r}}\omega^{s}|t|^2 f^2 G \lesssim  r^{4+2\gamma}\|f\|_{L^p(S_1^+)}^2,
\end{align}
\begin{align}\label{eq:weight_u}
\int_{S^+_{2r}}\omega^s u^2 G \leq C_{\bar\kappa, \epsilon_0} r^{2\kappa+2-\epsilon_0}.
\end{align}
Indeed, \eqref{eq:est_f} follows from a similar computation as in \eqref{eq:inhomo}. To see \eqref{eq:weight_u} we apply the growth estimate from Lemma \ref{lem:growth} in the region $\{(x,t): \frac{|x|}{\sqrt{-t}}\leq (\sqrt{-t})^{-\epsilon_0/2}\}$ and use the fast decay of the Gaussian in the complement.

Substituting \eqref{eq:e1}, \eqref{eq:est_f} and \eqref{eq:weight_u} into \eqref{eq:ug},  we get
\begin{align*}
\int_{A^+_{r/2,r}}|2t ug G| &\leq  C_{\epsilon_0} r^{5+\gamma-\epsilon_0}+ Cr^{4+3\gamma}\|f\|_{L^p(S_1^+)}+ C_{\epsilon_0}\|f\|_{L^p(S^+_1)} r^{\frac{9}{2}+\gamma-\epsilon_0}\leq C_0 r^{5+\sigma},
\end{align*}
where $\sigma:=\gamma-\epsilon_0-\frac{1}{2}$. Taking $\epsilon_0=\frac{1}{2}(\gamma-\frac{1}{2})$, we complete the proof of the case $p>2(n+2)$ (in this case $\gamma>1/2$ and $\kappa=3/2$).

Now we consider $p>n+2$ and $f$ satisfies the support condition \eqref{eq:support}. In this case \eqref{eq:est_f} can be improved. More precisely,  applying H\"older's inequality and using \eqref{eq:support} we get
\begin{align}\label{eq:est_suppf}
\int_{S^+_{2r}}\omega^{s}|t|^2 f^2 G \lesssim  r^{4+2\gamma}\left(1+\frac{1}{r}\right)^{s}e^{-\frac{1}{4r^2}}\|f\|_{L^p}^2,\quad \forall s\geq 0.
\end{align}
Thus, in this case, we find
\begin{align*}
\frac{1}{r^{2\kappa+2}}\int_{A^+_{r/2,r}}|2t ug G| &\leq C \frac{r^\gamma}{r^{2\kappa+2}} \int \omega^\gamma(u^2+ |t|^2 f^2)G + \frac{1}{r^{2\kappa+2}}\left(\int u^2 G\right)^{\frac{1}{2}}\left(\int |t|^2 f^2 G\right)^{\frac{1}{2}}\\
&\leq  C_{\bar\kappa, \epsilon_0} r^{\gamma-\epsilon_0}+ C\|f\|_{L^p}^2 \frac{e^{-\frac{1}{4r^2}}}{r^{2\kappa-2-2\gamma}}+ C\|f\|_{L^p}\frac{e^{-\frac{1}{8r^2}}}{r^{\kappa-1-\gamma+\epsilon_0/2}}.
\end{align*}
This leads to the desired estimate for some $\sigma=\sigma(n,p)$, which can be chosen the same as above, and $C_0=C_0(n,p,\|f\|_{L^p}, \|a^{ij}\|_{W^{1,1}_p}, \|u\|_{L^2},\bar\kappa)$. This completes the proof for the proposition.
\end{proof}

In the next proposition we show  that if $u\in \mathcal{G}^{A,f}_p(S_1^+)$ with $p>2(n+2)$, then the Weiss energy with $\kappa=3/2$ is almost monotone.
\begin{prop}[Almost monotonicity]\label{prop:Weiss}
If (i) $p>2(n+2)$ and $\kappa=\frac{3}{2}$, or (ii) $p>n+2$, $f$ satisfies the support condition \eqref{eq:support}) and $\kappa\leq \bar\kappa$, then there exists $\sigma=\sigma(n,p)\in (0,1)$ and $C_0=C_0(n,p,\|u\|_{L^2(S^+_1)}, \|f\|_{L^p(S_1^+)},\|a^{ij}\|_{L^2(S_1^+)}, \bar\kappa)$ such that
\begin{align*}
\frac{d}{dr} W^{\kappa}_u(r) \geq \frac{1}{2r^{2\kappa+3}}\int_{A^+_{r/2,r}}(Zu-\kappa u)^2 G \ dxdt -C_0 r^{-1+\sigma},\quad \forall r\in (0,1/2).
\end{align*}
Here the constants $\sigma$ and $C_0$ have the same dependence as in Proposition \ref{prop:Weiss2}, thus can be chosen to be the same. 
 \end{prop} 
\begin{proof}
The proof follows from a direct computation, see Theorem 13.1 in \cite{DGPT}. 
We rewrite the equation on $\R^n_+$ as 
\begin{align}\label{eq:Delta_u}
\p_tu -\Delta u = g,\quad g=\p_ia^{ij} \p_j u + (a^{ij}-\delta^{ij})\p_{ij}u +f.
\end{align}
Then from the proof for Theorem 13.1 in \cite{DGPT} (using the differential inequality for $(H^{\delta}_u)'$ and $(I^\delta_u)'$ in Lemma 6.1 which holds by a similar approximation argument as Proposition 6.2 in \cite{DGPT}) we have
\begin{align*}
r^{2\kappa+1}\frac{d}{dr}W^{\kappa}_u(r)\geq \frac{1}{r^2} \int_{A^+_{r/2,r}}\left(Zu-\kappa u \right)^2 G +2t g(Zu-\kappa u)G\ dxdt.
\end{align*}
We claim that when $p>2(n+2)$, for some $\sigma=\sigma(n,p)\in (0,1)$ we have
\begin{align}\label{eq:est_err}
\int_{A^+_{r/2,r}} |t|^2 |g|^2 G \ dxdt\leq C r^{5+\sigma}.
\end{align} 
Using the H\"older regularity of $a^{ij}$ and the $H^2$ estimate, cf. Lemma \ref{lem:H2}, we get
\begin{align*}
\int_{A^+_{r/2,r}}t^2 (a^{ij}-\delta^{ij})^2 (\p_{ij}u)^2 G &\leq C \int_{A^+_{r/2,r}}|t|^{2+\gamma} \omega^{2\gamma} (\p_{ij}u)^2 G\\
&\leq C r^{2\gamma}\int_{S^+_{2r}}\omega^{2\gamma}\left(u^2 +|t|^2 f^2\right)G,
\end{align*}
where as before $\omega=1+\frac{|x|}{\sqrt{|t|}}$. 
By \eqref{eq:est_f} and  \eqref{eq:weight_u}, we have that for any $\epsilon_0\in (0,1/4)$ and $C=C(n,p, \epsilon_0, \|u\|_{L^2}, \|f\|_{L^p}, \|a^{ij}\|_{W^{1,1}_p})$
\begin{align*}
&\int_{A^+_{r/2,r}}t^2 (a^{ij}-\delta^{ij})^2 (\p_{ij}u)^2 G \ dxdt \leq Cr^{5+2\gamma-\epsilon_0}+C\|f\|_{L^p}^2r^{4+4\gamma}.
\end{align*}
Next from the estimates for $X_2$ in the proof of Proposition \ref{prop:Carleman} (with $\psi=\frac{|x|^2}{8t} + \frac{n|\ln (-t)|}{4}$ and $\tau=\frac{n}{4}$), as well as the almost optimal regularity of $u$ (cf. Proposition \ref{prop:almost_op}) and H\"older's inequality, we get
\begin{align*}
\int_{A^+_{r/2,r}}t^2(\p_ia^{ij})^2 (\p_ju)^2 G &\leq   Cr^{\frac{4\gamma}{n+2}}\int_{S^+_{2r}} u^2 G  + Cr^{\frac{4\gamma}{n+2}+4}\int_{S^+_{2r}} f^2 G\\
&\leq C r^{5+\frac{4\gamma}{n+2}-\epsilon_0} + C \|f\|_{L^p}^2r^{4+\frac{4\gamma}{n+2}+2\gamma}.
\end{align*}
In the end, by H\"older's inequality, cf. \eqref{eq:est_f}, 
$\int_{A^+_{r/2,r}}t^2 f^2 G\ dxdt\lesssim  r^{4+2\gamma}\|f\|_{L^p}^2$. Combining the above inequalities together we thus obtain the claimed inequality \eqref{eq:est_err} with $\sigma=\min \{\frac{4\gamma}{n+2}-\epsilon_0, -1+2\gamma\}$. Note that $\sigma>0$ if $\epsilon_0=\frac{\gamma}{2(n+2)}$ and $\gamma>\frac{1}{2}$.
When $\kappa=\frac{3}{2}$, using \eqref{eq:est_err} and Young's inequality, we thus obtain
\begin{align*}
\frac{d}{dr}W^{\frac{3}{2}}_u(r)\geq \frac{1}{2r^{6}}\int_{A^+_{r/2,r}}\left(Zu-\kappa u\right)^2 G \ dxdt - C_0 r^{-1+\sigma}.
\end{align*}

In the case, when $p>n+2$ and $f$ satisfies the support condition \eqref{eq:support}, we instead use \eqref{eq:est_suppf} and Lemma \ref{lem:growth} (with a similar argument as in Proposition \ref{prop:Weiss2}) to obtain
\begin{align*}
&\int_{A^+_{r/2,r}}t^2 (a^{ij}-\delta^{ij})^2 (\p_{ij}u)^2 G \ dxdt \leq C_{\bar \kappa, \epsilon_0}r^{2\kappa+2+2\gamma-\epsilon_0}+C\|f\|_{L^p}^2r^{4+4\gamma}e^{-\frac{1}{r^2}},
\end{align*}
and 
\begin{align*}
\int_{A^+_{r/2,r}}t^2(\p_ia^{ij})^2 (\p_ju)^2 G 
&\leq C_{\bar \kappa, \epsilon_0} r^{2\kappa+2+\frac{\gamma}{n+2}-\epsilon_0} + C \|f\|_{L^p}^2r^{4+\frac{2\gamma}{n+2}+2\gamma}e^{-\frac{1}{r^2}}.
\end{align*}
Thus 
\begin{align*}
\frac{1}{r^{2\kappa+3}}\int_{A_{r/2,r}^+}|t|^2 |g|^2 G \leq C_{\bar \kappa,\epsilon_0} r^{-1+\frac{\gamma}{n+2}-\epsilon_0} + C\|f\|_{L^p}^2 \frac{r^{2+\frac{2\gamma}{n+2}+2\gamma}e^{-\frac{1}{r^2}}}{r^{2\kappa+1}}.
\end{align*}
If $\epsilon_0=\frac{\gamma}{2(n+2)}$, then the claimed inequality follows.
This completes the proof for the proposition.
\end{proof}

\begin{rmk}\label{rmk:dr}
When $\kappa=\frac{3}{2}$ and $p>2(n+2)$ (or $p>n+2$, $f$ satisfies the support condition \eqref{eq:support} and $\kappa\leq \bar\kappa$), by scaling we can also write Proposition \ref{prop:Weiss2} and Proposition \ref{prop:Weiss} as
\begin{align*}
\left|W^{\kappa}_{u}(r)-\frac{r}{4}\frac{d}{dr}\int_{A^+_{1/2,1}}(u_{\kappa, r} )^2 G \ dxdt \right|\leq C_0r^{\sigma},
\end{align*}
and
\begin{align*}
\frac{d}{dr}W^{\kappa}_{u}(r)&\geq  \frac{1}{2r}\int_{A^+_{1/2,1}}(Zu_{\kappa, r}-\kappa u_{\kappa, r})^2 G \ dxdt - C_0r^{-1+\sigma}\\
&=\frac{r}{2}\int_{A^+_{1/2,1}} \left(\frac{d}{dr}u_{\kappa,r}\right)^2 G \ dxdt - C_0 r^{-1+\sigma},
\end{align*}
where $\sigma\in (0,1)$ is sufficiently small depends on $p$ and $n$.
\end{rmk}

\subsection{Decay of the Weiss energy when $\kappa=\frac{3}{2}$}
We aim to derive the convergence rate of the Weiss energy when $\kappa=\frac{3}{2}$. Without loss of generality we assume that $(0,0)\in \Gamma^*_u$ and $\kappa=\kappa_{(0,0)}=\frac{3}{2}$. We denote the cone of blow-ups at $(0,0)$ by
\begin{align*}
\mathcal{P}_{3/2}:=\{c\Ree(x'\cdot e+ix_n)^{3/2}, \ c\geq 0,\ e\in \mathbb{S}^{n-2}\}.
\end{align*}
Note that each $h\in \mathcal{P}_{3/2}$ is a stationary solution to the Signorini problem with $a^{ij}=\delta^{ij}$ and $f=0$. Thus by Proposition \ref{prop:Weiss2} (or Remark \ref{rmk:dr}) ,
\begin{align*}
W_{h}^{3/2}(r)\equiv 0,\quad \forall h\in \mathcal{P}_{3/2}.
\end{align*}
For each $r>0$, we write the rescaled and renormalized solution by  
\begin{equation*}
u_r(x,t)=u_{3/2,r}(x,t):=\frac{u(rx,r^2 t)}{r^{3/2}},\quad (x,t)\in S^+_{1/r},
\end{equation*}
 and define its $L^2$ orthogonal complement by  
\begin{align}\label{eq:defi_v}
v_r: S^+_{1/r}\rightarrow \R, \quad v_r:=u_r-h_r,
\end{align}
where 
\begin{align*}
h_r(x',x_n):=c_r\Ree(x'\cdot e_r+ix_n)^{3/2},\quad (c_r,e_r)\in [0,\infty)\times \mathbb{S}^{n-2}
\end{align*}
 is a (weighted) $L^2$-projection of $u_{r}$ on $\mathcal{P}_{3/2}$, i.e.
\begin{align*}
\int_{A^+_{1/2,1}} |u_{r}-h_r|^2 G \ dxdt =\min_{h\in \mathcal{P}_{3/2}}\int_{A^+_{1/2,1}} |u_{r}-h_r|^2 G \ dxdt.
\end{align*} 
Note that here the minimum is attained, because the minimization is over $(c,e)\in [0,\infty)\times \mathbb{S}^{n-2}$, which is of finite dimensional, and $0\in \mathcal{P}_{3/2}$. From the minimality one immediately has
\begin{align}\label{eq:orth}
\int_{A^+_{1/2,1}}v_{r} h_r G \ dxdt =0
\end{align} 
which comes from the minimization in the multiplicative constant $c$ and 
\begin{align}\label{eq:orth2}
\int_{A^+_{1/2,1}}  c_r(x\cdot e') \Ree(x'\cdot e_r+ix_n)^{1/2}v_r G =0,\quad \forall e'\in \mathbb{S}^{n-1}, e'\perp e_r, e_n,
\end{align}
which comes from the minimization in rotations $e\in \mathbb{S}^{n-1}\cap \{x_n=0\}$.
It is not hard to see that the orthogonal complement $v_{r}$ solves the problem
\begin{equation}\label{eq:v}
\begin{split}
&\p_t v_{r}-\Delta v_{r}=g_r \text{ in } S^+_{1/r},\quad \text{ where} \\
g_r &:=(\p_ia^{ij}_r\p_ju_{r})+(a^{ij}_r-\delta^{ij})\p_{ij}u_{r}+ r^{1/2}f_r,\\
 a^{ij}_r(\cdot,\cdot)&:=a^{ij}(r\cdot,r^2\cdot),\quad f_r(\cdot,\cdot):=f(r\cdot,r^2 \cdot),
 \end{split}
\end{equation}
and on $S_{1/r}'$ it satisfies 
\begin{align}\label{eq:v_bdry}
v_r\p_n v_r= -(\p_n u_r) h_r \chi_{\{u=0\}\cap \{h_r>0\}} - (\p_n h_r) u_r \chi_{\{u>0\}\cap \{h_r=0\}}\geq 0.
\end{align}

The next proposition says that the Weiss energy does not increase after projecting out $\mathcal{P}_{3/2}$: 
\begin{lem}\label{lem:Weiss_v}
For each $r\in (0,1)$, let $v_{r}$ be defined as in \eqref{eq:defi_v}. Then 
\begin{align}\label{eq:Weiss_v}
W^{3/2}_{u_{r}}(1)=W^{3/2}_{v_{r}}(1)-2\int_{A'_{1/2,1}}|t|u_{r}\p_n h_r G\ dxdt. 
\end{align}
In particular, 
\begin{align*}
W^{3/2}_{v_{r}}(1)\leq W^{3/2}_{u_{r}}(1).
\end{align*}
Furthermore, we have
\begin{align}\label{eq:Wv}
&\left|W^{3/2}_{v_{r}}(1)  -\frac{1}{2}\int_{A^+_{1/2,1}}v_{r}(Zu_{r}-\frac{3}{2} u_{r}) G  -\int_{A'_{1/2,1}}|t|(\p_n h_{r} v_{r}+\p_nv_{r}h_{r}) G \right|\leq C_0 r^{\sigma},
\end{align}
where $C_0$ and $\sigma$ are the same constant as in Proposition \ref{prop:Weiss2}.
\end{lem}
\begin{proof}
One starts with 
$$u_{r}=v_{r}+h_r,\quad h_r\in \mathcal{P}_{3/2}.$$
From the definition of the Weiss energy, a direct computation yields
\begin{align*}
W^{3/2}_{u_{r}}(1) &=W^{3/2}_{v_{r}}(1)+W^{3/2}_{h_r}(1)+ 2\int_{A^+_{1/2,1}}|t|\nabla v_{r}\cdot\nabla h_r  G\ dxdt-\frac{3}{2}\int_{A^+_{1/2,1}}v_{r}h_r G \ dxdt,\\
&=W^{3/2}_{v_{r}}(1)+2\int_{A^+_{1/2,1}}|t|\nabla v_{r}\cdot\nabla h_r  G\ dxdt.
\end{align*}
Here we have used that $W^{3/2}_{h_r}(1)=0$ and the orthogonality condition \eqref{eq:orth}. Using integration by parts and that $\Delta h_r=0$ in $\R^n_+$, $x\cdot \nabla h_r=\frac{3}{2} h_r$ we get
\begin{align}\label{eq:int1}
\int_{A^+_{1/2,1}}|t|\nabla v_{r}\cdot\nabla h_r  G\ dxdt = -\int_{A'_{1/2,1}}|t|v_{r}\p_n h_r G \ dx'dt.
\end{align}
Recalling the definition of $v_r$ and using the fact that $h_{r}\p_nh_r=0$ on $\R^{n-1}\times \{0\}$, we get the claimed identity \eqref{eq:Weiss_v}.
In particular, $W^{3/2}_{v_{r}}(1)\leq W^{3/2}_{u_{r}}(1)$, as $u_{r}\geq 0$ and $\p_nh_r\leq 0$ on $\R^{n-1}\times \{0\}$.

To show \eqref{eq:Wv} we use the identity
\begin{align*}
W^{3/2}_{v_r}(1)= W^{3/2}_{u_r}(1)-2 \int_{A^+_{1/2,1}}|t|\nabla v_r\cdot \nabla h_r G
\end{align*}
and Proposition \ref{prop:Weiss2} to get
\begin{align}\label{eq:Wv1}
W^{3/2}_{v_r}(1)=\frac{1}{2}\int_{A^+_{1/2,1}}u_r (Zu_r-\frac{3}{2} u_r) G +\int_{A^+_{1/2,1}}|t| u_r g_r G -2 \int_{A^+_{1/2,1}}|t|\nabla v_r\cdot \nabla h_r G.
\end{align}
From \eqref{eq:orth} and that $Zh_r -\frac{3}{2} h_r =0$ we have
\begin{align}\label{eq:Wv2}
\int_{A^+_{1/2,1}}u_r (Zu_r-\frac{3}{2} u_r) G = \int_{A^+_{1/2,1}} v_r(Zu_r-\frac{3}{2}u_r)G + h_r Zv_r G.
\end{align}
We apply an integration by parts to get
\begin{align*}
\int_{A^+_{1/2,1}}|t|\nabla v_r\cdot \nabla h_r G =-\int_{A'_{1/2,1}}|t|h_r \p_n v_r G -\int_{A^+_{1/2,1}}|t|h_r\nabla\cdot (\nabla v_r G).
\end{align*}
Since $\p_t v_r-\Delta v_r=g_r$ in $\R^n_+$, $\nabla\cdot(\nabla v_r G) = (\p_t v_r + \frac{x}{2t}\cdot \nabla v_r)G - g_r G$. This implies 
\begin{align*}
\int_{A^+_{1/2,1}}|t|h_r\nabla\cdot (\nabla v_r G) =-\int_{A^+_{1/2}}\frac{1}{2} h_r Zv_r  G + |t|h_r g_r G.
\end{align*}
Thus from \eqref{eq:int1} and the above two equalities we obtain
\begin{align}\label{eq:Wv3}
2\int_{A^+_{1/2,1}}|t|\nabla v_r\cdot \nabla h_r G =-\int_{A'_{1/2,1}}|t|(v_r\p_nh_r+h_r\p_n v_r)G+\int_{A^+_{1/2}} (\frac{1}{2}h_r Zv_r +|t|h_rg_r) G 
\end{align} 
Combining \eqref{eq:Wv1}, \eqref{eq:Wv2} and \eqref{eq:Wv3} we have
\begin{align*}
W^{3/2}_{v_r}(1)=\frac{1}{2}\int_{A^+_{1/2,1}} v_r(Zu_r-\frac{3}{2}u_r)G +\int_{A'_{1/2,1}}|t|(v_r\p_nh_r+h_r\p_n v_r)G + \int_{A^+_{1/2,1}}|t| v_r g_r G.
\end{align*}
Since $|\int_{A^+_{1/2,1}}|t|v_rgG|\leq C_0r^\sigma$, which follows from $\|v_r\|_{L^2(A^+_{1/2,1})}\leq \|u_r\|_{L^2(A^+_{1/2,1})}$ and the estimates for $|\int_{A^+_{1/2,1}}tugG|$ in Proposition \ref{prop:Weiss2}, we get the claimed inequality \eqref{eq:Wv}.
\end{proof}

The next proposition is a discrete decay estimate of the Weiss energy along the scalings. The proof is based on a compactness argument. 

\begin{prop}\label{prop:epi}
Let $u\in \mathcal{G}^{A,f}_p(S_1^+)$ with (i) $p>2(n+2)$ or (ii) $p>n+2$ and $f$ satisfies \eqref{eq:support}. Assume 
\begin{align*}
\int_{S^+_1} u^2 G \ dxdt=1.
\end{align*}
Then there exists $c_0=c_0(n,p,\|a^{ij}\|_{W^{1,1}_p}, \|f\|_{L^p})\in (0,1)$  (independent of $k$ or $u$) such that
\begin{align*}
W^{3/2}_u(2^{-k-1})\leq (1-c_0)W^{3/2}_u(2^{-k})+2C_02^{-k\sigma/2}, \quad \forall k\in \N.
\end{align*}
Here $C_0$ and $\sigma$ are the constants in Remark \ref{rmk:dr}.
\end{prop} 

\begin{proof}
We argue by contradiction.  Assume there were a sequence of solutions $u_j\in \mathcal{G}^{A_j, f_j}_p(S_1^+)$ with $\int_{S^+_1} |u_j|^2 G =1$, $\|A_j\|_{W^{1,1}_p}\leq \|A\|_{W^{1,1}_p}$ and $\|f_j\|_{L^p}\leq \|f\|_{L^p}$,                                                                                                                                                                                                                                                                                                                                                                                                                                                                                                                                                                                                                                                                                                                                                                             a sequence of integers $k_j\rightarrow \infty$ and $c_j\in (0,\frac{1}{4})\rightarrow 0$ such that
\begin{align}\label{eq:con}
W^{3/2}_{u_j}(2^{-k_j-1})>(1-c_j)W^{3/2}_{u_j}(2^{-k_j})+2C_02^{-k_j\sigma/2},\quad \forall j\in \N,
\end{align}
which is equivalent to 
\begin{align*}
\int_{I_j} \frac{d}{dr} W^{3/2}_{u_j}(r)\ dr + 2C_02^{-k_j\sigma/2} < c_j W^{3/2}_{u_j}(2^{-k_j}),\quad I_j:=[2^{-k_j-1},2^{-k_j}].
\end{align*}
The above inequality  together with $\frac{d}{dr}W^{3/2}_{u_j}(r)\geq -C_0 r^{-1+\sigma}$, which is the almost monotonicity of the Weiss energy, cf. Proposition \ref{prop:Weiss}, yields for $k_j\geq k_0(\sigma)$,
\begin{align}\label{eq:contra}
\int_{I_j}\left|\frac{d}{dr} W^{3/2}_{u_j}(r)\right|\ dr+ C_02^{-k_j\sigma/2}\leq c_j W_{u_j}^{3/2}(2^{-k_j}).
\end{align}
Indeed, we note that the almost monotonicity of the Weiss energy gives $|\frac{d}{dr}W^{3/2}_{u_j}(r)|\leq \frac{d}{dr}W^{3/2}_{u_j}(r)+2C_0r^{-1+\sigma}$. Thus if $k_j\geq k_0(\sigma)$ for some $k_0(\sigma)$ large, $\int_{I_j}|\frac{d}{dr}W^{3/2}_{u_j}(r)| dr \leq \int_{I_j} \frac{d}{dr}W^{3/2}_{u_j}(r) \ dr  + \frac{2C_0}{\sigma} 2^{-k_j\sigma}\leq \int_{I_j} \frac{d}{dr}W^{3/2}_{u_j}(r) \ dr  + C_02^{-k_j\sigma/2}$. The inequality \eqref{eq:contra} in particular gives $W^{3/2}_{u_j}(r)>0$ for all $r\in I_j$, if $k_j\geq k_0(\sigma)$.
Furthermore, by the mean value theorem, for some $r_\ast\in I_j$, 
$|W_{u_j}^{3/2}(r) - \fint_{I_j} W^{3/2}_{u_j}(r') \ dr'|=|W_{u_j}^{3/2}(r) -  W^{3/2}_{u_j}(r_\ast) |\leq \int_{I_j}|\frac{d}{dr'}W^{3/2}_{u_j}(r')| \ dr'$. This together with \eqref{eq:contra} yields for $k_j\geq k_0(\sigma)$
\begin{align}\label{eq:contra21}
\left|W_{u_j}^{3/2}(r) - \fint_{I_j} W^{3/2}_{u_j}(r') \ dr'\right|\leq 2c_j \fint_{I_j} W^{3/2}_{u_j}(r') \ dr', \quad \forall r\in I_j.
\end{align}
Using the lower bound for $\frac{d}{dr}W^{3/2}_{u_j}$ in Proposition \ref{prop:Weiss} and \eqref{eq:contra21} we also obtain
\begin{align}\label{eq:contra2}
\fint_{I_j} \int_{A^+_{1/2,1}}(Zu_{j,r}-\frac{3}{2}u_{j,r})^2 G \ dxdtdr+C_0 2^{-k_j\sigma/2} &\lesssim c_j\fint_{I_j}W^{3/2}_{u_j}(r).
\end{align}
Let $h_{j,r}=c_{j,r}\Ree(x'\cdot e_{j,r}+ix_n)^{3/2}$ be the $L^2$-projection of $u_{j,r}$ to $\mathcal{P}_{3/2}$ and let
\begin{align}\label{eq:decomp}
v_{j,r}:=u_{j,r}-h_{j,r}.
\end{align}


\emph{Step 1:} In the first step we aim to show that for $j$ sufficiently large
\begin{align}
\max\{C_02^{-k_j\sigma/2},\fint_{I_j} W^{3/2}_{v_{j,r}}(1), \fint_{I_j}W^{3/2}_{u_{j,r}}(1)\} &\lesssim \fint_{I_j} H_{v_{j, r}}(1),\label{eq:bd_Weiss_v}
\end{align}
where for a function $w$, 
\begin{align*} 
H_{w_r}(1):=\int_{A^+_{1/2,1}} w^2_{r}  G \ dxdt, \quad w_r(x,t):=\frac{w(rx,r^2t)}{r^{3/2}}.
\end{align*}
For that we first recall from \eqref{eq:Wv} that for all $r\in I_j$, 
\begin{align*}
&\left|W^{3/2}_{v_{j,r}}(1)  -\frac{1}{2}\int_{A^+_{1/2,1}}v_{j,r}(Zu_{j,r}-\frac{3}{2} u_{j,r}) G  -\frac{1}{2}\int_{A'_{1/2,1}}|t|(\p_n h_{j,r} v_{j,r}+\p_nv_{j,r}h_{j,r}) G \right|\leq C_0 r^{\sigma}.
\end{align*}
Integrating in $r$ over $I_j$, applying Cauchy-Schwarz to the first integral and using \eqref{eq:contra2} we get
\begin{equation}\label{eq:Wv4}
\begin{split}
\fint_{I_j} W^{3/2}_{v_{j,r}}(1)  &\leq Cc_j\fint_{I_j} W^{3/2}_{u_{j,r}}(1)  + \fint_{I_j}\frac{1}{4}H_{v_{j,r}}(1)  +\frac{1}{2}\fint_{I_j}\int_{A'_{1/2,1}}|t|(\p_n h_{j,r} v_{j,r}+\p_nv_{j,r}h_{j,r}) G\\
&\stackrel{\eqref{eq:Weiss_v}}{=} Cc_j\fint_{I_j} W^{3/2}_{v_{j,r}}(1) + \fint_{I_j}\frac{1}{4}H_{v_{j,r}}(1)+ \left(\frac{1}{2}-2Cc_j\right)\fint_{I_j}\int_{A^+_{1/2,1}}|t|\p_n h_{j,r}v_{j,r}G \\
&\qquad +\frac{1}{2}\fint_{I_j}\int_{A^+_{1/2,1}}|t|\p_nv_{j,r}h_{j,r} G.
\end{split}
\end{equation}
From this and that  $\p_n h_{j,r} v_{j,r}\leq 0$, $\p_nv_{j,r}h_{j,r}\leq 0$ on $S_1'$ we immediately get for $j$ sufficiently large,
\begin{align*}
-\frac{3}{4}\fint_{I_j} H_{v_{j,r}}(1) \ dr\leq  \fint_{I_j}W^{3/2}_{v_{j,r}}(1) \ dr \leq \frac{1}{2}\fint_{I_j}H_{v_{j,r}}(1)\ dr,
\end{align*}
where the first inequality is simply due to the definition of the Weiss energy. Furthermore, from \eqref{eq:Wv4} and the above lower bound for the Weiss energy we get for $j$ sufficiently large,
\begin{equation*}
\begin{split}
-\fint_{I_j}\int_{A'_{1/2,1}}|t|(\p_n h_{j,r} v_{j,r}+\p_nv_{j,r}h_{j,r}) G \ dx'dtdr\leq 4 \fint_{I_j} H_{v_{j,r}}(1)\ dr .
\end{split}
\end{equation*}
We note that the above two inequalities together with 
 \eqref{eq:Weiss_v} yields that 
\begin{align}\label{eq:W_u_v}
\fint_{I_j}W^{3/2}_{u_{j,r}}(1)\ dr \lesssim \fint_{I_j} H_{v_{j,r}}(1) \ dr.
\end{align}
This together with \eqref{eq:contra2} yields
$\fint_{I_j} H_{v_{j,r}}(1) \ dr  \gtrsim C_0 2^{-k_j\sigma/2}$.
Thus we obtain the claimed estimate \eqref{eq:bd_Weiss_v}.


\emph{Step 2:} In this step we  show that \eqref{eq:bd_Weiss_v} holds without taking the average in $r$, i.e. 
\begin{align}\label{eq:pt_Weiss_v}
\max\{C_02^{-k_j\sigma/2}, W^{3/2}_{v_{j,r}}(1), W^{3/2}_{u_{j,r}}(1)\} &\lesssim H_{v_{j,r}}(1), \quad\forall r\in I_j.
\end{align}
For that we first establish the following claim:\\
\emph{Claim:} For all $r\in I_j$ and for all $j$ we have
\begin{align*}
\left|H_{v_{j,r}}(1)-\fint_{I_j} H_{v_{j,r}}(1)\right|\leq \frac{1}{2} \fint_{I_j}H_{v_{j,r}}(1) + C\left( c_j \fint_{I_j}W^{3/2}_{u_{j,r}}(1) -C_02^{-k_j\sigma/2}\right),
\end{align*}
where $C>1$ is some absolute constant. In particular,
\begin{align}\label{eq:H_v_bar}
H_{v_{j,r}}(1)\geq \frac{1}{2}\fint_{I_j}H_{v_{j,r}}(1) -Cc_j\fint_{I_j} W^{3/2}_{u_{j,r}}(1) + C_02^{-k_j\sigma/2},\quad \forall r\in I_j.
\end{align} 
To prove the claim, we first note that a direct computation and Young's inequality lead to 
\begin{equation}
\label{eq:H_v}
\begin{split}
\left|\frac{d}{dr} H_{v_{j,r}}(1)\right|&=2\left|\int_{A^+_{1/2,1}} v_{j,r}\frac{d}{dr} v_{j,r}  G\ dxdt\right|\\
&\leq \frac{1}{4r} H_{v_{j,r}}(1) + 16\int_{A^+_{1/2,1}} r\left(\frac{d}{dr}v_{j,r}\right)^2 G\ dxdt.
\end{split}
\end{equation}
Observe that $
\frac{d}{dr} v_{j,r}=\frac{d}{dr}u_{j,r}=\frac{1}{r}(Zu_{j,r}-\frac{3}{2}u_{j,r})$. Thus  \eqref{eq:contra2} can be rewritten as
\begin{align*}
\int_{I_j}\int_{A^+_{1/2,1}} r\left(\frac{d}{dr}v_{j,r}\right)^2 G\ dxdtdr\lesssim  c_j\fint_{I_j} W^{3/2}_{u_{j,r}}(1) \ dr - C_02^{-k_j\sigma/2}.
\end{align*}
Thus after integrating \eqref{eq:H_v} in $r$ and using that $2|I_j|\geq r\geq |I_j|$ for $r\in I_j$ we get the claimed estimate.\\
We now use \eqref{eq:H_v_bar} and \eqref{eq:bd_Weiss_v} to get 
\begin{align}\label{eq:H_v_bar1}
H_{v_{j,r}}(1)-\frac{1}{2}\fint_{I_j}H_{v_{j,r}}(1)\gtrsim  -c_j\fint_{I_j}H_{v_{j,r}}(1),\quad \forall r\in I_j.
\end{align} 
For $j$ sufficiently large, the above inequality along with  \eqref{eq:contra21} and \eqref{eq:bd_Weiss_v} yields the desired inequality \eqref{eq:pt_Weiss_v}.
 This complete the proof for \emph{Step 2}. 

\medskip

\emph{Step 3: Compactness.} In this step we show that there is a sequence $\tilde{u}_j$, which are suitable rescaling and renormalization of $u_j$, such that $\tilde{u}_j$ converge to a nonzero limit in $\mathcal{P}_{3/2}$.

To show this, first from \eqref{eq:pt_Weiss_v} and \eqref{eq:contra2} we conclude the existence of radii $r_j\in I_j$ such that
\begin{align}
\max\{C_02^{-k_j\sigma/2}, W^{3/2}_{u_{j,r_j}}(1)\}&\lesssim H_{u_{j,r_j}}(1),\label{eq:Weiss_u_j}\\
\int_{A^+_{1/2,1}}(Zu_{j,r_j}-\frac{3}{2}u_{j,r_j})^2 G \ dxdt &\lesssim c_jH_{u_{j,r_j}}(1).\label{eq:contra_u_j}
\end{align}
 We now carry out a blow-up argument: let
\begin{align*}
\tilde{u}_j(x,t)&:=\frac{u_{j,r_j}(x,t)}{H_{u_{j,r_j}}(1)^{1/2}}.
\end{align*}
Then from the scaling we immediately have
\begin{align*}
H_{\tilde{u}_j}(1)=1.
\end{align*}
Furthermore, \eqref{eq:Weiss_u_j} and \eqref{eq:contra_u_j} together with the definition of the Weiss energy yield
\begin{align*}
\int_{A^+_{1/2,1}} t|\nabla \tilde{u}_{j}|^2 G\ dxdt &\lesssim 1,\\
\int_{A^+_{1/2,1}}(x\cdot \nabla \tilde{u}_j+2t\p_t\tilde{u}_j)^2 G\ dxdt &\lesssim 1.
\end{align*}
To see that this yields the desired compactness we make the change of variable
\begin{align*}
(x,t)\mapsto (y,t), \quad y:=\frac{x}{\sqrt{-t}}.
\end{align*}
Then in terms of $\hat u_j(y,t):=\tilde{u}_j(\sqrt{-t}y, t)$ and $d\mu:=e^{-|y|^2/4} dy$, the above estimates yield
\begin{align*}
\int_{A^+_{1/2,1}} t^2 |\p_t\hat u_j|^2+|\nabla \hat u_j|^2 \ d\mu dt&\lesssim 1.
\end{align*}
Due to the compact embedding $W^{1,2}({A^+_{1/2,1}}; d\mu dt)\hookrightarrow L^2({A^+_{1/2,1}}; d\mu dt)$, there is $\hat u_0\in W^{1,2}({A^+_{1/2,1}}; d\mu dt)$ such that along a subsequence (not relabelled) one has
\begin{align*}
\int_{A^+_{1/2,1}}(\hat u_j-\hat u_0)^2 \ d\mu dt \rightarrow 0,
\end{align*}
 and $(\nabla\hat u_j,\p_t\hat u_j)$ converges weakly in $L^2({A^+_{1/2,1}}; d\mu dt)$ to $(\nabla\hat u_0,\p_t\hat u_0)$. Writing the convergence back to $\tilde{u}_j$ we get that
\begin{align*}
\int_{A^+_{1/2,1}} (\tilde{u}_j-\tilde{u}_0)^2 G\ dxdt \rightarrow 0
\end{align*}
and $(\nabla \tilde{u}_j,\p_t\tilde{u}_j)$ converge to $(\nabla \tilde u_0,\p_t\tilde{u}_0)$ weakly in $L^2({A^+_{1/2,1}}, G dxdt)$. In particular, $\int_{A^+_{1/2,1}}\tilde{u}_0^2 G \ dxdt=1$, thus $\tilde{u}_0$ is nontrivial. To derive the limiting equation for $\tilde{u}_0$ we note that $\tilde{u}_j$ solves the Signorini problem
\begin{align*}
\p_t\tilde{u}_j-\Delta \tilde{u}_j&= \frac{\tilde{g}_j}{H_{u_{j,r_j}}(1)^{1/2}} \text{ in }{A^+_{1/2,1}}\\
\tilde{u}_j\geq 0, \ \p_n \tilde{u}_j\leq 0, \ \tilde{u}_j\p_n\tilde{u}_j&=0 \text{ on } {A'_{1/2,1}},
\end{align*}
where 
\begin{align*}
\tilde{g}_j(x,t):= r_j^{1/2}\left(\p_i a^{ik}_j\p_k u_j+(a^{ik}_j-\delta^{ik})\p_{ik}u_j+ f_j\right)(r_jx,r_j^2t),\quad r_j\in I_j.
\end{align*} 
From the proof for Proposition \ref{prop:Weiss} (cf. \eqref{eq:est_err}) and \eqref{eq:Weiss_u_j} we get
\begin{align*}
\frac{\int_{A^+_{1/2,1}} \tilde{g}_j^2 G \ dxdt }{H_{u_{j,r_j}}(1)}\leq C \frac{r_j^{-1+\sigma}r_j}{2^{-k_j\sigma/2}}\leq C r_j^{\sigma/2}.
\end{align*}
As $\int_{A^+_{1/2,1}}\tilde{u}_j^2 G \ dxdt$ is uniformly bounded,  from the up to the boundary uniform H\"older regularity estimate for $\tilde{u}_j$ one has that for each $K\Subset \overline{\R^n_+}\times (-1,-1/2)$ and some $\alpha\in (0,1)$ small,
\begin{align*}
\tilde{u}_j\rightarrow \tilde{u}_0 \text{ in } H^{1+\alpha,\frac{1+\alpha}{2}}(K).
\end{align*} 
This allows us to pass to the limit in the equation for $\tilde{u}_j$ and  conclude that the limit function $\tilde{u}_0$ is a solution to the Signorini problem 
\begin{align*}
\p_t \tilde{u}_0-\Delta \tilde{u}_0=0 &\text{ in } {A^+_{1/2,1}},\\
\tilde{u}\geq 0,\ \p_n\tilde{u}\leq 0, \ \tilde{u} \p_n\tilde{u} =0 &\text{ on } {A'_{1/2,1}}.
\end{align*}
Moreover, by \eqref{eq:contra_u_j} and the weak convergence, $\tilde{u}_0$ is parabolic $3/2$-homogeneous, i.e. 
\begin{align*}
Z\tilde{u}_0=\frac{3}{2}\tilde{u}_0\ \text{ a.e. in } {A^+_{1/2,1}}.
\end{align*}
Therefore from the Liouville theorem, cf. eg. \cite[Proposition 1]{Shi20} and the unique continuation property for parabolic equations, there is some $c_0>0$ and $e_0\in \mathbb{S}^{n-2}$ such that 
\begin{align*}
\tilde{u}_0=c_0\Ree(x'\cdot e_0+ ix_n)^{3/2}\in \mathcal{P}_{3/2}.
\end{align*}

\medskip

\emph{Step 4: } In the last step we show that with the scale $r_j$ chosen in \emph{Step 3}, the sequence $\tilde{v}_j$, which are the rescaled and renormalized remainder terms, up to a subsequence converge to a nonzero function in $\mathcal{P}_{3/2}$. This immediately leads to a contradiction, because each $\tilde{v}_j$ lies in the orthogonal projection of $\mathcal{P}_{3/2}$ by our construction.

More precisely, from \eqref{eq:pt_Weiss_v} we have
\begin{align}
\max\{C_02^{-k_j\sigma/2}, W^{3/2}_{v_{j,r_j}}(1)\}&\lesssim H_{v_{j,r_j}}(1),\label{eq:Weiss_v_j}
\end{align}
where $r_j$ are the same as those in \emph{Step 3}. 
Furthermore, from \eqref{eq:contra2}, \eqref{eq:pt_Weiss_v} and that $Zh-\frac{3}{2} h=0$ for all $h\in \mathcal{P}_{3/2}$, we also have
\begin{align}
 \int_{A^+_{1/2,1}}(Zv_{j,r_j}-\frac{3}{2}v_{j,r_j})^2 G \ dxdt &\lesssim c_j  H_{v_{j,r_j}}(1). \label{eq:contra_v_j}
 \end{align}
 Consider
 \begin{align*}
 \tilde{v}_j(x,t)&:=\frac{v_{j,r_j}(x,t)}{H_{v_{j,r_j}}(1)^{1/2}}.
 \end{align*}
Similar arguments as in \emph{Step 3} yields a further subsequence (not relabelled) and a nontrivial limit function $\tilde{v}_0\in W^{1,2}({A^+_{1/2,1}}; G dxdt)$ such that 
\begin{align*}
\tilde{v}_j\rightarrow \tilde{v}_0 &\text{ in } L^2({A^+_{1/2,1}}; G dxdt)\\
\p_tv_j\rightharpoonup \p_tv_0, \quad \nabla v_j\rightharpoonup \nabla v_0 &\text{ weakly in } L^2({A^+_{1/2,1}};  G dx dt).
\end{align*}
Note from the equation of $v_j$, cf. \eqref{eq:v}, and our normalization, we have that $\tilde{v}_j$ satisfies the equation
\begin{align*}
\p_t \tilde{v}_j-\Delta \tilde{v}_j= \frac{\tilde{g}_j}{H_{v_j}(r_j)^{1/2}} \text{ in } {A^+_{1/2,1}},
\end{align*}
where $\tilde{g}_j$ is the same as in \emph{Step 3}, and on ${A'_{1/2,1}}$ it satisfies 
\begin{align*}
\p_n\tilde{v}_j=0 \text{ on } \{u_{j,r_j}>0\}\cap \{h_{j,r_j}>0\},\quad \tilde{v}_j=0 \text{ on } \{u_{j,r_j}=0\}\cap \{h_{j, r_j}=0\}.
\end{align*}
Again using the estimate for $\tilde{g}_j$ in \emph{Step 3} and that $H_{v_j}(r_j)\gtrsim 2^{-k_j\sigma/2}$, in the limit $j\rightarrow \infty$ we have
\begin{align*}
\p_t \tilde{v}_0-\Delta \tilde{v}_0 =0 \text{ in } {A^+_{1/2,1}}.
\end{align*}
By \eqref{eq:contra_v_j}, $\tilde{v}_0$ is $3/2$-parabolic homogeneous. 
Furthermore,  $\tilde{v}_0$ satisfies the Dirichlet-Neumann condition on ${A'_{1/2,1}}$: 
\begin{align*}
\tilde v_0(x',0,t)& =0  \text{ on } \Lambda_0:=\{(x',x_n): x_n=0, x'\cdot e_0\leq 0\},\\
\p_n \tilde v_0(x',0,t)&=0 \text{ on } \R^{n-1}\setminus \Lambda_0,
\end{align*}
where $e_0$ is the same direction as in \emph{Step 3}. Indeed, this follows from the similar argument as in \cite[Proposition 2]{Shi20}. Then we conclude that $v_0\in \mathcal{P}_{3/2}$ using \eqref{eq:orth2} (cf. \cite[Proposition 2]{Shi20} for more detailed arguments). This however gives a contradiction, as in each step we have already subtract the projection on $\mathcal{P}_{3/2}$.
\end{proof}

Since so far it is not known that the Weiss energy is nonnegative, we need the following proposition, which indicates that the Weiss energy goes to $-\infty$ exponentially fast (in terms of $\ln r$) once it becomes negative. 

\begin{prop}\label{prop:epi1}
	 Let $u\in \mathcal{G}^{A,f}_p(S_1^+)$ with (i) $p>2(n+2)$, or (ii) $p>n+2$ if $f$ satisfies \eqref{eq:support}. Assume that 
	\begin{align*}
		\int_{S_1^+} u^2 G \ dxdt=1.
	\end{align*}
	 Then there exists $c_0=c_0(n,p,\|f\|_{L^p},\|a^{ij}\|_{W^{1,1}_p})\in (0,1)$  (independent of $k$ or $u$)  such that
	\begin{align*}
		W^{3/2}_u(2^{-k-1})\leq (1+c_0)W^{3/2}_u(2^{-k})+2C_02^{-k\sigma/2}, \quad \forall k\in \N,
	\end{align*}
	where the constants $C_0$ and $\sigma$ are the same as in Remark \ref{rmk:dr}.
\end{prop} 

\begin{proof}
	The proof follows from the same contradiction argument as in Proposition \ref{prop:epi}. Below we only point out the differences in the proof. Assume there is a sequence of solutions $u_j$ with $\int_{S_1^+}u_j^2 G \ dxdt=1$, a sequence of integers $k_j\rightarrow \infty$ and $c_j\in (0,\frac{1}{4})\rightarrow 0$ such that
	\begin{align*}
		W^{3/2}_{u_j}(2^{-k_j-1})>(1+c_j)W^{3/2}_{u_j}(2^{-k_j})+2C_02^{-k_j\sigma/2},\quad \forall j\in \N.
	\end{align*}
This implies that 
	\begin{align*}
		\int_{I_j} \frac{d}{dr} W^{3/2}_{u_j}(r)\ dr +C_02^{-k_j\sigma/2} < -c_j W^{3/2}_{u_j}(2^{-k_j}),\quad I_j:=[2^{-k_j-1},2^{-k_j}].
	\end{align*}
	Together with the almost monotonic property of $r\mapsto W^{3/2}_{u_j}(r)$, cf. Proposition \ref{prop:Weiss}, we get
	\begin{align}\label{eq:contra11}
	(1+c_j)W^{3/2}_{u_j}(r)\lesssim \fint_{I_j}W^{3/2}_u(r')\ dr'\lesssim (1-c_j)W^{3/2}_{u_j}(r),\quad \forall r\in I_j.
	\end{align}
	Furthermore, by Remark \ref{rmk:dr} this implies that for $j$ sufficiently large
	\begin{align}\label{eq:contra1}
		\fint_{I_j} \int_{A^+_{1/2,1}}(Zu_{j,r}-\frac{3}{2}u_{j,r})^2 G \ dxdtdr+ C_02^{-k_j\sigma/2} &\lesssim -c_j\fint_{I_j}W^{3/2}_{u_j}(r)\ dr.
	\end{align}
	Let $h_{j,r}=c_{j,r}\Ree(x'\cdot e_{j,r}+ix_n)^{3/2}$ be the $L^2$-projection of $u_{j,r}$ to $\mathcal{P}_{3/2}$ and let
	\begin{align*}
		v_{j,r}:=u_{j,r}-h_{j,r}.
	\end{align*}
	
	\emph{Step 1:} In the first step we show that for $j$ sufficiently large
	\begin{align}
	2^{-k_j\sigma/2}\leq - W^{3/2}_{u_{j}}(r)\leq -W^{3/2}_{v_{j}}(r) \leq \frac{3}{4} H_{v_{j,r}}(1),\quad \forall r\in I_j,\label{eq:bd_Weiss_v1}
	\end{align}
	particularly $W^{3/2}_{u_j}(r)\leq 0$ for $r\in I_j$, and
	\begin{align}\label{eq:hom_v}
		\fint_{I_j} \int_{A^+_{1/2,1}}(Zv_{j,r}-\frac{3}{2}v_{j,r})^2 G \ dxdtdr  \lesssim c_j\fint_{I_j}H_{v_{j,r}}(1)\ dr . 
	\end{align} 
	For that we note that the first inequality of \eqref{eq:bd_Weiss_v1} follows from  \eqref{eq:contra11} and \eqref{eq:contra1}, and the second  inequality is by Lemma \ref{lem:Weiss_v}. The third inequality of \eqref{eq:bd_Weiss_v1} simply follows from the definition of the Weiss energy.
Using that $Zh=\frac{3}{2}h$ for all $h \in \mathcal{P}_{3/2}$, we obtain \eqref{eq:hom_v} from \eqref{eq:contra1} and \eqref{eq:bd_Weiss_v1}.
	
	\emph{Step 2:} By \eqref{eq:hom_v}, there is a sequence of radii $r_j\in I_j$ such that 
	\begin{align*}
\int_{A^+_{1/2,1}}(Zv_{j,r_j}-\frac{3}{2}v_{j,r_j})^2 G dxdt=\int_{A^+_{1/2,1}}(Zu_{j,r_j}-\frac{3}{2}u_{j,r_j})^2 G dxdt \lesssim c_jH_{v_{j,r_j}}(1).
	\end{align*}
	Furthermore, from $W^{3/2}_{v_j}(r)\leq W^{3/2}_{u_j}(r)\leq 0$ and the definition of the Weiss energy we have
	\begin{align*}
	\int_{A^+_{1/2,1}} |t||\nabla u_{j,r_j}|^2 \lesssim H_{u_{j,r_j}}(1),\quad
	\int_{A^+_{1/2,1}}|t||\nabla v_{j,r_j}|^2 \lesssim H_{v_{j,r_j}}(1).
	\end{align*}
	Therefore, if we consider
	\begin{align*}
	\tilde{u}_j=\frac{u_{j,r_j}}{H_{u_{j,r_j}}(1)^{1/2}},\quad \tilde{v}_j=\frac{v_{j,r_j}}{H_{v_{j,r_j}}(1)^{1/2}}
	\end{align*}
	and argue similarly as in Proposition \ref{prop:epi}, we can find  subsequences (not relabelled) $\tilde{u}_j$, $\tilde{v}_j$ such that
	\begin{align*}
	\tilde{u}_j\rightarrow \tilde{u}_0=c_0\Ree(x'\cdot e_0+ix_n)^{3/2}
	\end{align*}
	in $L^2({A^+_{1/2,1}}, Gdxdt)$, where $c_0>0$ (is such that $H_{\tilde{u}_j}(1)=1$) and $e_0\in \mathbb{S}^{n-1}\cap \{x_n=0\}$, and
	\begin{align*}
	\tilde{v}_j\rightarrow \tilde{v}_0=\tilde{u}_0\in \mathcal{P}_{3/2}
	\end{align*}
	in $L^2({A^+_{1/2,1}}, Gdxdt)$. 
	 This is however a contradiction, as for each $j$ and each $r$ we have subtracted the projection of $u_j$ to $\mathcal{P}_{3/2}$.
\end{proof}

As a consequence of Proposition \ref{prop:epi} and Proposition \ref{prop:epi1}, we obtain the decay estimate of the Weiss energy:

\begin{cor}[Decay estimate of the Weiss energy]\label{Cor:Weiss_pos}
	 Let $u\in \mathcal{G}^{A,f}_p(S_1^+)$ with  (i) $p>2(n+2)$, or (ii) $p>n+2$ if $f$ satisfies \eqref{eq:support}. Assume that 
	\begin{align*}
		\int_{S_1^+} u^2 G \ dxdt=1.
	\end{align*}
	Then there are constants $C>0$ sufficiently large depending on $n,p,\|f\|_{L^p}$, $\|a^{ij}\|_{W^{1,1}_p}$ and $\epsilon_0=\epsilon_0(n,p,\|f\|_{L^p}, \|a^{ij}\|_{W^{1,1}_p})\in (0,1)$,  such that for all $r\in (0,1/4)$, 
	\begin{equation*}
	-Cr^{\sigma/2}\leq W^{3/2}_{u}(r)\leq Cr^{\epsilon_0}.
	\end{equation*}
	Here $\sigma$ is the same constant as in Remark \ref{rmk:dr}.
\end{cor}
\begin{proof}
\emph{(i) Upper bound:} From Proposition \ref{prop:epi}, for any $k\in \N$, $k\geq 1$, we have
\begin{align*}
W^{3/2}_u(2^{-k})&\leq (1-c_0)^{k} W^{3/2}_u(1)+ 2C_0\sum_{i=0}^{k-1} (1-c_0)^i 2^{-(k-1-i)\sigma/2}\\
&\leq (1-c_0)^{k} W^{3/2}_u(1)+ 2C_0(1-c_0)^{k}.                                                                 
\end{align*}
Let $\alpha_0:=\frac{-\ln (1-c_0)}{\ln 2}$. Then from the above inequality we get
\begin{align*}
W^{3/2}_u(2^{-k})\leq C 2^{-\alpha_0 k},\quad \forall k\in \N.
\end{align*}
where $C=C(n,p,\|f\|_{L^p}, \|a^{ij}\|_{W^{1,1}_p}, \|u\|_{L^2})$. For any $r\in [2^{-k},2^{-k+1}]$, using the almost monotonicity of the Weiss energy, cf. Proposition \ref{prop:Weiss}, we get
\begin{align}\label{eq:interpolation}
W^{3/2}_u(r)\leq W^{3/2}_u(2^{-k+1})+ C_02^{-\sigma k}.
\end{align}
Combining the above two inequalities, we obtain the desired upper bound by setting $\epsilon_0:=\min\{\alpha_0, \sigma\}\in (0,1)$.

\emph{(ii) Lower bound:} 
	 Assume on contrary that for some $k_0\in \N$, $W^{3/2}_u(2^{- k_0})< -C 2^{-\frac{ k_0\sigma}{2}}$ for some large $C$ which will be determined later. We define 
	\begin{equation*}
	m(r):=\int_{A^+_{1/2,1}}u_{r}^2G\ dxdt,\quad u_r(x,t)=\frac{u(rx,r^2 t)}{r^{3/2}}.
	\end{equation*}
	 Then using Remark \ref{rmk:dr} we have that 
\begin{equation*}
m'(r) =\frac{d}{dr}\int_{A^+_{1/2,1}}u_{r}^2G\ dxdt \le \frac{4}{r}W^{3/2}_u(r)+ C_0 r^{\sigma-1}. 
\end{equation*}
Now we integrate from $2^{- k}$ to $2^{- k_0}$ to find
\begin{align}\label{eq:change_m}
m(2^{- k_0})-m(2^{- k}) 
&\leq 4\sum^{j=k}_{j=k_0+1}\int_{2^{- j}}^{ 2^{- (j-1)}}\frac{W^{3/2}_u(r)}{r}\ dr + C_0 \int_{2^{-  k}}^{ 2^{- k_0}}r^{\sigma-1}\ dr.
\end{align}
Using Proposition \ref{prop:epi1} and arguing similarly as in (i), we get that for any $j\geq k_0+1$ and for a $\tilde C=\tilde C(n,p,\|f\|_{L^p},\|a^{ij}\|_{W^{1,1}_p})$,
\begin{align*}
W^{3/2}_u(2^{- j}) 
&\leq (1+c_0)^{j-k_0} W^{3/2}_u(2^{- k_0})+ \tilde C (1+c_0)^{j-k_0}2^{- \frac{k_0\sigma}{2}}.
\end{align*}
Choosing $C\geq 2\tilde C$ in the contradiction assumption, we thus get
\begin{align*}
W^{3/2}_u(2^{- j})\leq -\tilde C(1+c_0)^{j-k_0}2^{-\frac{k_0\sigma}{2}}.
\end{align*} 
Then it follows from \eqref{eq:interpolation} that
\begin{align*}
W^{3/2}_u(r)\leq -\tilde C(1+c_0)^{j-1-k_0}2^{- \frac{k_0\sigma}{2}}+\tilde{C}2^{- \sigma j},\quad \forall r\in [2^{- j}, 2^{- (j-1)}].
\end{align*}
Inserting the above inequality into \eqref{eq:change_m} we have that if $k$ is sufficiently large, then 
\begin{align*}
m(2^{- k_0})-m(2^{- k})
&\le -\tilde C \sum_{j=k_0+1}^{j=k}\left((1+c_0)^{(j-k_0-1)} 2^{-\frac{k_0\sigma}{2}}+2^{- \sigma j}\right)+ \tilde C 2^{-k_0\sigma}\\
&\le - C(1+c_0)^{k},
\end{align*}
where $C$ depends also on $k_0.$ After rearranging the terms we obtain
\begin{align*}
 m(2^{-k})\geq m(2^{-k_0})+ C(1
+c_0)^{k}.
\end{align*}
We now recall the definition of $m(r)$ and use almost optimal regularity for $m(r)$, cf. Proposition \ref{prop:almost_op}, to find
\begin{align*}
m(2^{-k})\leq C_\epsilon 2^{k\epsilon},
\end{align*}
for any $\epsilon\in (0,1)$. This gives a contradiction if $\epsilon\leq \frac12 \log_2(1+c_0)$  and if $k$ is sufficiently large. 
Thus $W^{3/2}_u(2^{- \delta k})\geq -C2^{- \frac{k\sigma}{2}}$ for all $k\in \N$. This together with Proposition \ref{prop:Weiss} completes the proof for all $r$. 
\end{proof}
%
%

The decay estimate of the Weiss energy in Corollary \ref{Cor:Weiss_pos} immediately gives the optimal growth estimate of the solution at $(0,0)\in \Gamma_u$:

\begin{prop}\label{prop:opt_growth}
 Let $u\in \mathcal{G}^{A,f}_p(S_1^+)$ with (i) $p>2(n+2)$, or (ii) $p>n+2$ and $f$ satisfies \eqref{eq:support}. Assume that $(0,0)\in \Gamma_u^\ast$. Then for all $r$ small enough we have 
 \begin{equation*}
 \frac{1}{r^2}\int_{A^+_{r/2,r}}u^2Gdxdt \leq C r^3,
 \end{equation*}
 where $C=C(n,p,\|f\|_{L^p},\|a^{ij}\|_{W^{1,1}_p}, \|u\|_{L^2})$. 
\end{prop}
\begin{proof}
Let $\ell(r):=m(r)^{1/2}$, where $m(r)$ is defined in Corollary \ref{Cor:Weiss_pos}. 
	Then it suffices to show that $\ell(r)\leq C$. A direct computation yields
	\begin{equation*}
	\ell'(r)=\ell(r)^{-1}\int_{A^+_{1/2,1}}u_{r}\frac{d}{dr}u_{r}G\ dxdt.
	\end{equation*}
	We now use Cauchy-Schwarz inequality and Remark \ref{rmk:dr} to obtain
	\begin{align*}
	|\ell'(r)|&\le \left(\int_{A^+_{1/2,1}}\left(\frac{d}{dr}u_{r}\right)^2Gdxdt\right)^{1/2}\leq \frac{1}{\sqrt r}\left(\frac{d}{dr}W_{u}^{3/2}(r) +Cr^{-1+\sigma}\right)^{1/2}.
	\end{align*}
Let $r_1<r_2\leq 1/4$. We now integrate $\ell'(r)$ from $r_1$ to $r_2$ and use Cauchy-Schwarz to obtain
\begin{align*}
|\ell(r_1)-\ell(r_2)| &\leq \int_{r_1}^{r_2}\frac{1}{r^{1/2}}\left(\frac{d}{dr}W_{u}^{3/2}(r) +Cr^{-1+\sigma}\right)^{1/2} \ dr\\
&\leq \left(\int_{r_1}^{r_2}\frac{1}{r}dr\right)^{1/2}\left(\int_{r_1}^{r_2}(\frac{d}{dr}W_{u}^{3/2}(r) +Cr^{-1+\sigma})dr\right)^{1/2}\\
&\leq \left(\operatorname{ln}{r_1/r_2}\right)^{1/2}\left(W_{u}^{3/2}(r_2) +Cr_2^{\sigma}-W_{u}^{3/2}(r_1) -C r_1^{\sigma}\right)^{1/2}.
\end{align*}
Here and below the constant $C=C(n,p,\|u\|_{L^2}, \|f\|_{L^p})$ may vary line by line but is always with the same parameter dependence. By Corollary \ref{Cor:Weiss_pos}, $W_{u}^{3/2}(r_1)\ge -Cr_1^{\sigma}$, therefore we obtain 
\begin{align*}
|\ell(r_1)-\ell(r_2)| &\leq \left(\operatorname{ln}{r_1/r_2}\right)^{1/2}\left(W_{u}^{3/2}(r_2) + C r_2^{\sigma}\right)^{1/2}.
\end{align*}
Using a dyadic argument and the upper bound of the Weiss energy in Corollary \ref{Cor:Weiss_pos}, we find $\ell(r)\leq C$. Indeed, let $k\in \N$ be such that $r_1\in [2^{-(k+1)}r_2,2^{-k}r_2)$. Then
\begin{align*}
|\ell(r_1)-\ell(r_2)|&\leq |\ell(r_1)-\ell(2^{-k}r_2)|+\sum_{j=0}^{k-1}|\ell(2^{-j}r_2)-\ell(2^{-j-1}r_2)|\\
&\leq 2\sum_{j=0}^k\left(W^{3/2}_u(2^{-j}r_2)\right)^{1/2}+C\sum_{j=0}^k(2^{-j}r_2)^{\frac{\sigma}{2}}\leq C(r_2)^{\min \{\frac{\epsilon_0}{2},\frac{\sigma}{2}\}}.
\end{align*}  
This completes the proof.
\end{proof}

With the optimal growth estimate at the free boundaries at hand, a similar argument as in \cite[Section 9]{DGPT} yields the desired optimal H\"older regularity  result in Theorem \ref{thm:opt}:
\begin{proof}[Proof for Theorem \ref{thm:opt}] The proof for Theorem \ref{thm:opt}(ii) has already been done in Remark \ref{rmk:opt_h}, we thus only prove (i). 
Since the proof is similar as the constant coefficient case, cf. \cite[Section 9]{DGPT}, we only provide an outline here. 
Given any $(x_0,t_0)\in \Gamma_u^\ast\cap Q_{1/4}$, we can make a change of variable as in Remark \ref{rmk:sigma} to recenter the problem at $(x_0,t_0)$ and then apply Proposition \ref{prop:opt_growth} and Lemma \ref{lem:linfty} to get
\begin{align*}
\sup_{(B_{r}(x_0,t_0))^+\times [t_0-cr^2, t_0+cr^2]\cap Q_1}|u|\leq C r^{3/2},\quad \forall r\in (0,1/4).
\end{align*}
Here $c\in (0,1)$ is an absolute constant. Combining the above inequality with the interior H\"older regularity estimate for parabolic equations, we obtain the desired up to the boundary H\"older regularity results (see \cite[Theorem 9.1]{DGPT}). 

When $p\in (n+2, \infty]$ and $f$ satisfies \eqref{eq:support}, we have $u\in H^{1+\gamma, \frac{1+\gamma}{2}}(\overline{Q_{1/4}^+})$. This is also optimal, due to the optimal regularity of the linear problem away from the free boundary.
\end{proof}

\section{Regular free boundary and its regularity}\label{sec:fb_reg}
In this section we study the set of points of free boundary with minimal vanishing order, the \emph{regular set}:
\begin{equation*}
\mathcal{R}_u:=\left\{(x_0,t_0)\in \Gamma_u: \kappa_{(x_0,t_0)}=\frac{3}{2}\right\}\subset\Gamma_u
\end{equation*}
We note by the vanishing order gap (Proposition \ref{prop:almost_op}) and the upper semi-continuity (Lemma \ref{lem:upper_semi}), $\mathcal{R}_u$ is relatively open in $\Gamma_u^\ast$.  In the following lemma we show that $\mathcal{R}_u=\mathcal{R}_u^\ast:=\{(x_0,t_0)\in \Gamma_u^\ast: \kappa_{(x_0,t_0)}=\frac{3}{2}\}$.
\begin{lem}\label{lem:ext_bdr}
 Let $u\in \mathcal{G}^{A,f}_p(S_1^+)$. Suppose that (i) $p>2(n+2),$ or (ii) $p>n+2$ and $f$ satisfies the support condition \eqref{eq:support}. Let $\kappa_{(x_0,t_0)}$ be the vanishing order of $u$ at $(x_0,t_0).$ If $(x_0,t_0)\in \Gamma^*_u$ with   $\kappa_{(x_0,t_0)}=\frac32$, then $(x_0,t_0)\in \Gamma_u.$
\end{lem}
\begin{proof}
 Given $(x_0,t_0)\in \Gamma^*_u$ with   $\kappa_{(x_0,t_0)}=\frac32.$ Let $\epsilon>0$ be arbitrary. Consider $u_r^{(x_0,t_0)}$ as defined in \eqref{eq:u_r}. Then,  from Proposition \ref{prop:compactness}, we have 
 \begin{align}\label{eq:extl2}
 u_{r_j}^{(x_0,t_0)} \rightarrow u^{(x_0,t_0)} \;\text{in}\; L^2(B_1^+ \times (-1,-\delta^2))
 \end{align} for each $\delta>0$ and $u^{(x_0,t_0)}$ solves the constant coefficient Signorini problem in $Q_1^+$ with vanishing order $\kappa_{(x_0,t_0)}$ at $(0,0)$. Moreover, by Proposition \ref{prop:almost_op} $(0,0) \in \Gamma_{u^{(x_0,t_0)}}^*.$ From \cite[Proposition 10.8]{DGPT}, we get $(0,0) \in \Gamma_{u^{(x_0,t_0)}}.$ 
Consequently, using the definition of $\Gamma_{u^{(x_0,t_0)}},$ we find that
$$Q'_{\epsilon} \cap \{(x',t)\in Q_1'\; : \; u^{(x_0,t_0)}(x',t)>0\} \neq \emptyset$$
for all $Q'_\epsilon$ with $\epsilon>0$ small. Fix an arbitrary $Q'_\epsilon$, 
the continuity of $u^{(x_0,t_0)}$ yields the existence of $Q'_{\tilde{\epsilon}}(y_0,s_0) \subset Q'_{\epsilon}$, $s_0<0$ with $$Q'_{\tilde{\epsilon}}(y_0,s_0) \subset  \{(x',t)\in Q_1'\; : \; u^{(x_0,t_0)}(x',t)>0\}.$$
Using the uniform convergence of $u_{r_j}^{(x_0,t_0)}$ to $u^{(x_0,t_0)}$ along a subsequence $r_j$, we get for $j$ sufficiently large,
 $$u_{r_j}^{(x_0,t_0)}(x,t)>0 \text{ in } Q'_{\bar\epsilon}(y_0,s_0).$$  This then immediately implies that for $j\gg 1$,
 \begin{align*}
 Q'_\epsilon\cap \{(x',t)\in Q'_1: u^{(x_0,t_0)}_{r_j}(x',t)>0\}\neq \emptyset.
 \end{align*}
 Writing this in terms of $u$, we obtain a sequence of $\epsilon_j\rightarrow 0$ such that 
  $$Q'_{\epsilon_j}(x_0,t_0) \cap \{(x',t)\in Q_1'\; : \; u(x',t)>0\} \neq \emptyset.$$
This along with the fact that $u(x_0,t_0) =0$ (by the definition of $\Gamma_u^*$) implies that $(x_0,t_0)\in \p_{Q'_1}\{(x',t)\in Q'_1: u(x',t)>0\}$. This completes the proof of the lemma. 
\end{proof}

In the rest of the section we will show that the decay rate of the Weiss energy  yields the regularity of $\mathcal{R}_u$. 
For that we start by showing the convergence rate of the $3/2$-homogeneous scaling at $(x_0,t_0)$ 
\begin{align}\label{eq:32scaling}
u_{(x_0,t_0),r}(x,t):=\frac{ u(\sqrt{A(x_0,t_0)} rx+x_0, tr^2+t_0)}{r^{3/2}}
\end{align}
 to its blowup limit.

\begin{prop}\label{prop:rate_conv}
Let $u\in \mathcal{G}^{A,f}_p(S_1^+)$. Suppose that (i) $p>2(n+2),$ or (ii) $p>n+2$ and $f$ satisfies the support condition \eqref{eq:support}. Given $(x_0,t_0)\in \mathcal{R}_u\cap Q_{1/2}$,  there is a unique $u_{(x_0,t_0)}\in \mathcal{P}_{3/2}\setminus\{0\}$, a constant $\sigma_0=\sigma_0(n,p)\in (0,1)$  and $C=C(n,p,\|f\|_{L^p(S_1^+)}, \|u\|_{L^2(S_1^+)})$ such that 
\begin{align*}
\|(u_{(x_0,t_0),r}-u_{(x_0,t_0)})\sqrt{G}\|_{L^2(S_1^+)}\leq Cr^{\sigma_0}.
\end{align*}
Here $u_{(x_0,t_0),r}$ is the $3/2$-homogeneous scaling defined in \eqref{eq:32scaling}.
\end{prop}
\begin{proof}
For simplicity assume that $(0,0)\in \mathcal{R}_u$. It suffices to show the convergence rate at $(0,0)$. For simplicity, we write $u_r:=u_{(0,0),r}$ and $u_0:=u_{(0,0)}$. From the proof for Proposition \ref{prop:opt_growth} we have that
\begin{align*}
\|(u_{r_1}-u_{r_2})\sqrt{G}\|_{L^2(A^+_{1/2,1})}\leq C r_2^{\sigma_0},\quad \forall\ 0<r_1<r_2,
\end{align*}
for some $\sigma_0\in (0,1)$ depending on $n,p$. By rescaling we can get the decay in $A^+_{1/4,1}$ as well. This implies that $\{u_{2^{-k}}\}_{k\in \N}$ is a Cauchy sequence in $L^2(A^+_{1/4,1}; G dxdt)$. Thus there is a unique limit $u_0$ such that 
\begin{align*}
\|(u_r-u_0)\sqrt{G}\|_{L^2(A^+_{1/4,1})}\leq Cr^{\sigma_0},\quad \forall r\in (0,1/4)
\end{align*}
By a simple scaling argument, one can conclude 
\begin{align}\label{eq:ur_cgs}
\|(u_r-u_0)\sqrt{G}\|_{L^2(S_1^+)}\leq Cr^{\sigma_0},\quad \forall r\in (0,1/4).
\end{align} 
The limit function $u_0$ is nonzero, because otherwise we would have $\|u_r\sqrt{G}\|_{L^2(S_1^+)}\leq Cr^{\sigma_0}$, which contradicts the assumption that $(0,0)\in \mathcal{R}_u$. 

Next we show that $u_0\in \mathcal{P}_{3/2}$. First by Lemma \ref{lem:energy_close} we have
\begin{align*}
\|\nabla(u_{r_1}-u_{r_2})\sqrt{|t|G}\|_{L^2(S_{1/4}^+)}\lesssim & \|(u_{r_1}-u_{r_2})\sqrt{G}\|_{L^2(S_{1/2}^+)} + \||t|(f_{r_1}-f_{r_2})\sqrt{G}\|_{L^2(S_{1/2}^+)}\ \\
&+\|A_{r_1}-A_{r_2}\|_{L^\infty}\|(u_{r_2}+|t|f_{r_2})\sqrt{G}\|_{L^2(S_{1}^+)}.
\end{align*} 
Using \eqref{eq:est_f} for $p>2(n+2)$ and \eqref{eq:est_suppf} for $p>n+2$ and $f$ satisfies \eqref{eq:support}, we find $\||t|f_r\sqrt{G}\|_{L^2(S_{1}^+)} \le C r^{-5/2}\||t|f \sqrt{G}\|_{L^2(S_{r}^+)} \le Cr^{\sigma_1}.$ Combining this with  \eqref{eq:ur_cgs} and the H\"older continuity of the matrix $A,$ we obtain
 \begin{align*}
\|\nabla(u_{r_1}-u_{r_2})\sqrt{|t|G}\|_{L^2(S_{1/4}^+)}\le Cr_2^{\sigma_1},
\end{align*}
where $\sigma_1$  depends only on $n$ and $p$. We now proceed similarly to the approach used for \eqref{eq:ur_cgs} to conclude
\begin{align*}
\|\nabla(u_{r}-u_{0})\sqrt{|t|G}\|_{L^2(S_{1/4}^+)}\le Cr^{\sigma_1}.
\end{align*}
Thus passing to the limit in $W^{3/2}_{u_r}(1/4)$ as $r\rightarrow 0$ and using Corollary \ref{Cor:Weiss_pos} we have that $W^{3/2}_{u_0}(r)\equiv 0$ for all $r\in (0,1/8)$. We observe that $u_0$ is a solution to the Signorini problem for the heat equation in any compact sets in $A^+_{\delta, 1/4}$, $\delta>0$, with zero inhomogeneity. Indeed, we have shown that $u_r$ is uniformly bounded in $W^{1,1}_2(K)$ for any compact set $K\subset {A^+_{\delta,1/4}}$ for any $\delta>0$. Then we can pass to  the limit in the variational inequality of $u_r$, cf. the proof for Proposition \ref{prop:compactness}. The inhomogeneity is zero for the equation of $u_0$, because of our assumption that $p>2(n+2)$ (or $p>n+2$ and $f$ satisfies the support condition \eqref{eq:support}). Thus from Proposition \ref{prop:Weiss} (with $C_0=0$, as $a^{ij}=\delta^{ij}$ and $f=0$ for $u_0$) we can conclude that $u_0$ is $3/2$-parabolic homogeneous. By the Liouville theorem, cf. \cite[Proposition 8.5]{DGPT} or \cite[Proposition 1]{Shi20}, we have $u_0\in \mathcal{P}_{3/2}$. 
\end{proof}

The rate of convergence in Proposition \ref{prop:rate_conv} yields immediately that $u_{(x_0,t_0)}$ varies H\"older continuously with respect to $(x_0,t_0)$, which we will show in the next corollary. To fix the notation, given $(x_0,t_0)\in \mathcal{R}_u\cap Q_{1/4}$, let 
\begin{align*}
u_{(x_0,t_0)}(x):=c_{(x_0,t_0)}\Ree(x'\cdot e_{(x_0,t_0)}+ i|x_n|)^{3/2}, 
\end{align*}
where $c_{(x_0,t_0)}>0$ and $e_{(x_0,t_0)}\in \mathbb{S}^{n-1}\cap \{x_n=0\}$, be  the unique blowup limit of $u$ at $(x_0,t_0)$. 
\begin{cor}\label{cor:holder_var}
Let $u_{(x_0,t_0)}$ defined as above be the unique blowup limit at $(x_0,t_0)\in \mathcal{R}_u\cap Q_{1/4}$. Then 
\begin{align*}
\mathcal{R}_u\ni (x_0,t_0)&\mapsto c_{(x_0,t_0)}\in (0,\infty),\\
\mathcal{R}_u\ni (x_0,t_0)&\mapsto e_{(x_0,t_0)}\in \mathbb{S}^{n-1}\cap \{x_n=0\}
\end{align*}
are parabolic H\"older continuous. More precisely, there are constants $\theta=\theta(n,p)\in (0,1)$ and $C=C(n,p,\|f\|_{L^p(S_1^+)}, \|u\|_{L^2(S_1^+)}, \|a^{ij}\|_{W^{1,1}_p})$ such that for any $(x_0,t_0)$ and  $(y_0,s_0)\in \mathcal{R}_u\cap Q_{1/4}$ we have
\begin{align*}
|c_{(x_0,t_0)}-c_{(y_0,s_0)}|+|e_{(x_0,t_0)}-e_{(y_0,s_0)}|\leq C (|x_0-y_0|+|t_0-s_0|^{1/2})^{\theta}.
\end{align*}
\end{cor}
\begin{proof}
First we show the H\"older continuity for $c_{(x_0,t_0)}$.
For that we note that
\begin{align*}
|c_{(x_0,t_0)}-c_{(y_0,s_0)}|=c_n \left| \|u_{(x_0,t_0)}\sqrt{G}\|_{L^2}-\|u_{(y_0,s_0)}\sqrt{G}\|_{L^2}\right|.
\end{align*}
Here and below the $L^2$ norm is taken over $A^+_{1/2,1}$. 
By Proposition \ref{prop:rate_conv},
\begin{align*}
\|(u_{(x_0,t_0),r}-u_{(x_0,t_0)})\sqrt{G}\|_{L^2}\leq C r^{\sigma_0},
\end{align*}
and the same inequality holds if we replace $(x_0,t_0)$ by $(y_0,s_0)$. Thus by triangle inequality,
\begin{align*}
|c_{(x_0,t_0)}-c_{(y_0,s_0)}|&\lesssim \|(u_{(x_0,t_0),r}-u_{(x_0,t_0)})\sqrt{G}\|_{L^2}+ \|(u_{(y_0,s_0),r}-u_{(y_0,s_0)})\sqrt{G}\|_{L^2}\\
&+\big\|(u_{(x_0,t_0),r}-u_{(y_0,s_0),r})\sqrt{G}\big\|_{L^2}\\
&\lesssim Cr^{\sigma_0}+\big\|(u_{(x_0,t_0),r}-u_{(y_0,s_0),r})\sqrt{G}\big\|_{L^2}.
\end{align*}
Note that
\begin{align*}
&\big\|(u_{(x_0,t_0),r}-u_{(y_0,s_0),r})\sqrt{G}\big\|_{L^2}\\
&=r^{-3/2}\|(u(A(x_0,t_0)^{1/2} rx+x_0, tr^2+t_0)-u(A(y_0,s_0)^{1/2} rx+y_0, tr^2+s_0))\sqrt{G}\|_{L^2}.
\end{align*}
From the interior Lipschitz regularity of $u$ and the $H^{\gamma,\gamma/2}$ regularity of the coefficients, we get
\begin{align*}
&|u(A(x_0,t_0)^{1/2} rx+x_0, tr^2+t_0)-u(A(y_0,s_0)^{1/2} rx+y_0, tr^2+s_0)|\\
&\leq C\left(|A(x_0,t_0)^{1/2}-A(y_0,s_0)^{1/2}||rx|+ |x_0-y_0|+|s_0-t_0|^{1/2}\right)\\
&\leq C\left((|x_0-y_0|^{\gamma}+|t_0-s_0|^{\gamma/2})|rx|+|x_0-y_0|+|s_0-t_0|^{1/2}\right).
\end{align*}
Thus
\begin{align}\label{eq:interior_holder}
\big\|(u_{(x_0,t_0),r}-u_{(y_0,s_0),r})\sqrt{G}\big\|_{L^2}\leq Cr^{-3/2}(|x_0-y_0|^{\gamma}+|t_0-s_0|^{\gamma/2}).
\end{align}
This yields
\begin{align*}
|c_{(x_0,t_0)}-c_{(y_0,s_0)}|\leq Cr^{\sigma_0}+Cr^{-3/2}(|x_0-y_0|^{\gamma}+|t_0-s_0|^{\gamma/2}).
\end{align*}
Balancing the two terms on the right hand side we have
\begin{align*}
|c_{(x_0,t_0)}-c_{(y_0,s_0)}|\leq C(|x_0-y_0|+|t_0-s_0|^{1/2})^\theta, \quad \theta:=\frac{\gamma\sigma_0}{\sigma_0+\frac{3}{2}}.
\end{align*}

Notice that
\begin{align*}
|e_{(x_0,t_0)}-e_{(y_0,s_0)}|&\leq c_n \big\|(\Ree(x'\cdot e_{(x_0,t_0)}+i|x_n|)^{3/2}-\Ree(x'\cdot e_{(y_0,s_0)}+i|x_n|)^{3/2}) \sqrt{G}\big\|_{L^2}\\
&=c_n\left\|\left(\frac{u_{(x_0,t_0)}}{c_{(x_0,t_0)}}-\frac{u_{(y_0,s_0)}}{c_{(y_0,s_0)}}\right)\sqrt{G}\right\|_{L^2}.
\end{align*}
Thus the H\"older continuity of $e_{(x_0,t_0)}$ follows from Proposition \ref{prop:rate_conv} and the H\"older continuity of $c_{(x_0,t_0)}$.
\end{proof}

At the end of the section we will show that the regularity free boundary $\mathcal{R}_u$ is locally parabolic H\"older continuous:

\begin{thm}\label{thm:opt_fb2}
Let $u\in \mathcal{G}^{A,f}_p(S_1^+)$. Suppose $p>2(n+2)$. Let $\mathcal{R}_u$ be the regular set of the free boundary. Then $\mathcal{R}_u$ is locally a parabolic Lipschitz continuous graph, i.e. given any $(x_0,t_0)\in \mathcal{R}_u$, there is a small neighborhood $Q'_\delta(x_0,t_0)$ and a function $g:Q_{\delta}''(x_0,t_0)\rightarrow \R$, which is in the parabolic Lipschitz class $H^{1, \frac{1}{2}}$, such that up to a rotation of the coordinates
\begin{align*}
\mathcal{R}_u\cap Q'_{\delta}(x_0,t_0)=\{(x',0, t)\in Q'_{\delta}(x_0,t_0): x_{n-1}=g(x'',t)\}\cap \mathcal{R}_u.
\end{align*}
Moreover, $\nabla''g \in H^{\theta,\theta/2}(Q''_{\delta}(x_0,t_0))$ for some $\theta=\theta(n,p) \in (0,1).$ If instead $f$ satisfies the support condition \eqref{eq:support}. Then, for $p\in (n+2,2(n+2)]$ we have $g \in H^{1,\frac{1+\gamma}{3}}$, with $\gamma=1-\frac{n+2}{p}$, and $\nabla''g \in H^{\theta,\theta/2}(Q''_{\delta}(x_0,t_0))$ for some $\theta=\theta(n,p) \in (0,1).$ 
\end{thm}

\begin{proof}
Let $(x_0,t_0)\in \mathcal{R}_u$. Without loss of generality we may assume that $(x_0,t_0)=(0,0)$ and $e_{(x_0,t_0)}=e_{n-1}$, where $e_{n-1}$ denotes the unit vector in $\R^{n-1}$ with all coordinates zero, except for the $(n-1)$-th coordinate. Since $\mathcal{R}_u$ is relatively open, there is $\delta>0$ such that $\Gamma_u^\ast\cap Q'_{\delta}\subset \mathcal{R}_u$. 

First we show that for any $\epsilon>0$, there is a radius $r_\epsilon>0$ such that 
\begin{align}\label{eq:unif}
\|u_{(\bar x, \bar t), r}-u_{(\bar x, \bar t)}\|_{H^{1,0}(\overline{Q_{1/2}^+})}\leq \epsilon
\end{align}
for all $r<r_\epsilon$ and for all $(\bar x, \bar t)\in \Gamma_{u}^\ast\cap Q'_{\delta/2}$. Here $u_{(\bar x, \bar t), r}$ is the $3/2$-homogeneous scaling defined in \eqref{eq:32scaling}.
The proof follows from a compactness argument using Proposition \ref{prop:rate_conv}, Corollary \ref{cor:holder_var} and an interior H\"older estimate. More precisely, assume \eqref{eq:unif} is not true, then one can find a sequence  of regular free boundary points $(\bar x_k, \bar t_k)\in \Gamma_u^\ast \cap Q'_{\delta/2}$, a sequence of radii $r_k\rightarrow 0$ such that 
\begin{align*}
\|u_{(\bar x_k, \bar t_k),r_k}- u_{(\bar x_k, \bar t_k)}\|_{H^{1,0}(\overline{Q_{1/2}^+})}\geq \epsilon_0>0.
\end{align*}
 Clearly, we may assume that $(\bar x_k, \bar t_k) \rightarrow (\bar x_0, \bar t_0)\in \Gamma_u^\ast \cap \overline{Q'_{\delta/2}}.$ Thus, using Corollary \ref{cor:holder_var}, we obtain 
\begin{align}\label{eq:contuni}
\|u_{(\bar x_k, \bar t_k),r_k}- u_{(\bar x_0, \bar t_0)}\|_{H^{1,0}({Q_{1/2}^+})}\geq \epsilon_0/2.
\end{align} 
By Proposition \ref{prop:rate_conv} and Lemma \ref{lem:linfty}, $u_{(\bar x_k, \bar t_k),r_k}$ are uniformly bounded in $L^\infty({Q_{3/4}^+})$ and solve the Signorini problem with inhomogeneities uniformly bounded in $L^p$. By the interior H\"older estimate we have that the sequence $u_{(\bar x_k, \bar t_k), r_k}$ are uniformly bounded in the parabolic H\"older space $H^{1+\alpha,\alpha}(\overline{Q^+_{1/2}})$ for some $\alpha>0$. Thus up to a subsequence (not relabeled) $\|u_{(\bar x_k, \bar t_k), r_k}-\tilde{u}_0\|_{H^{1,0}(\overline{Q_{1/2}^+})}\rightarrow 0$ for some $\bar u_0$. On the other hand, by Proposition \ref{prop:rate_conv} and Corollary \ref{cor:holder_var}, 
\begin{align*}
\|(u_{(\bar x_k \bar t_k), r_k}-u_{(\bar x_0, \bar t_0)})\sqrt{G}\|_{L^2(S_1^+)}\rightarrow 0.
\end{align*}
This implies that $\bar u_0 = u_{(\bar x_0, \bar t_0)}$ in $\overline{B_{1/2}^+}\times [-\delta, -1]$ for any $\delta>0$. This is a contradiction to \eqref{eq:contuni}.


We now show that by taking $\delta$ possibly smaller depending on $(x_0,t_0)$, $\Gamma^\ast_u\cap Q'_{\delta}$ is a graph, i.e.  $\Gamma^\ast_u\cap Q'_{\delta}=\{(x',0, t)\in Q'_{\delta}: x_{n-1}=g(x'',t)\}$ for some function $g$.  For that we will show for any $\eta>0$ small, there is $r_\eta>0$ such that for each $(\bar x, \bar t)\in \Gamma_u^\ast\cap Q'_{\delta}$ one has 
\begin{equation}\label{eq:cone}
\begin{split}
\bar x+\left(\mathcal{C}'_{\eta}(e_{n-1})\cap B'_{r_\eta}\right)&\subset \{u(\cdot,\bar t)>0\},\\
\bar x-\left(\mathcal{C}'_{\eta}(e_{n-1})\cap B'_{r_\eta}\right)&\subset \{u(\cdot,\bar t)=0\},
\end{split}
\end{equation}
 where, for given $e \in \mathbb{S}^{n-1} \cap \{x_n=0\}$, $C'_\eta(e)$ is the space cone
$$\mathcal{C}'_{\eta}(e):=\{x'\in \R^{n-1}: x'\cdot e \geq \eta |x'|\}.$$
Indeed, 
from the explicit expression for $u_{(\bar x, \bar t)}$ and the H\"older continuity of $(\bar x, \bar t)\mapsto c_{(\bar x, \bar t)}$, we know that if $\delta=\delta(0,0)>0$ is chosen small, then for any $(\bar x,\bar t)\in \mathcal{R}_u\cap Q'_{\delta}$ and any $\eta\in (0,1/8)$ we have \begin{equation}\label{eq:blowup-nondeg}
u_{(\bar x, \bar t)}\gtrsim c_{(\bar x, \bar t)}\gtrsim \frac{1}{2}c_{(0,0)} \text{ in } K_\eta:=\mathcal{C}'_{\eta}(e_{(\bar x, \bar t)})\cap \p B'_{1/2}.
\end{equation}
 From the uniform convergence of $u_{(\bar x, \bar t), r}$ to $u_{(\bar x, \bar t)}$, cf. \eqref{eq:unif}, there is $r_\eta>0$ such that $u_{(\bar x, \bar t), r}>0$ in $K_\eta$ if $r<r_\eta$. Scaling back this yields that $\bar x+ \left(\mathcal{C}'_{\eta}(e_{(\bar x, \bar t)})\cap B'_{r_\eta}\right)\subset \{u(\cdot,\bar t)>0\}$ for a possibly bigger $\eta$ which differs by an absolute multiplicative constant. Finally, we use Corollary \ref{cor:holder_var} to get the first inclusion in \eqref{eq:cone}.

Using the strict negativity of the Neumann derivative $\p_n u_{(\bar x, \bar t)}$ in the interior of the contact set $\{(x',t):u_{(\bar x, \bar t)}(x',t)=0\}$ and the uniform convergence of the Neumann derivatives, we can argue as above and obtain the interior cone condition for $\{u(\cdot, \bar t)=0\}$. This completes the proof for \eqref{eq:cone}. 

The cone condition \eqref{eq:cone} implies that for each $t\in (-\frac{\delta^2}{4}, 0]$ the set $\mathcal{R}_u\cap (B_{\delta/2}\times \{t\})$ is a Lipschitz graph, i.e. up to a spacial rotation of the coordinates  $\mathcal{R}_u\cap Q'_{\delta/2}=\{(x',t)\in Q'_{\delta/2}: x_{n-1}=g(x'', t)\}$ for some function $g$ which is Lipschitz in $x''$. Moreover, we have that for all $(\bar x, \bar t) \in Q'_{\delta/2},$
\begin{equation}\label{eq:graph}
\begin{split}
\{x_{n-1}>g(x'', \bar t)\} \cap B_{\delta/2}'(\bar x)&= \{u(\cdot,\bar t)>0\}\cap B_{\delta/2}'(\bar x),\\
\{x_{n-1}\le g(x'', \bar t)\} \cap B_{\delta/2}'(\bar x)&= \{u(\cdot,\bar t)=0\}\cap B_{\delta/2}'(\bar x).
\end{split}
\end{equation}
By letting $\eta\rightarrow 0$ (thus $r_\eta\rightarrow 0$) in  \eqref{eq:cone} we have that 
$g(\cdot, t)$ is differentiable at each point $x''$ with normal vector $e_{(x'',g(x'',t),t)}$  or equivalently the unit normal $\nu(x'';t)=(\nabla''g, -1)/\sqrt{1+|\nabla''g|^2}$ exists for each $x''$. As the normals varies H\"older continuously, cf. Corollary \ref{cor:holder_var}, we conclude that $\nabla'' g\in H^{\theta,\frac{\theta}{2}}$ for some $\theta\in (0,1)$.

The regularity of $g$ in $t$ follows from the nondegeneracy estimate
\begin{align}\label{eq:nondeg_t}
u(x',0,t)\geq c \left(x_{n-1}-g(x'',t)\right)^{\frac{3}{2}}, 
\end{align}
for all $(x',0,t)\in Q'_{\delta/2}$ and $x_{n-1}>g(x'',t)$, and the positive constant $c$ is independent of $t$. This can be shown using \eqref{eq:blowup-nondeg} and Corollary \ref{cor:holder_var}, along with a scaling argument as in \cite[ p.23]{BDGP20}). Indeed using Corollary \ref{cor:holder_var} $\mathcal{C}'_{2\eta}(e_{n-1})\cap \p B'_{1/2} \subset \mathcal{C}'_{\eta}(e_{(\bar x, \bar t)})\cap \p B'_{1/2}.$ Now we use \eqref{eq:unif} and \eqref{eq:blowup-nondeg} to get 
\begin{equation}\label{eq:blowup-nondeg2}
u_{(\bar x, \bar t),r}\left(\frac12e_{n-1},0,0\right)\gtrsim \frac{1}{2}c_{(0,0)},
\end{equation} 
for all $(\bar x, \bar t) \in \mathcal{R}_u \cap Q'_{\delta}$ and $r$ small depending on $c_{(0,0)}$. For possibly smaller $\delta$, let $(x',0,t)\in Q'_{\delta/2}$  with $x_{n-1}>g(x'',t).$ We now take $r=2(x_{n-1}-g(x'',t))$ , $(\bar x, \bar t)=(x'',g(x'',t),t)$ in \eqref{eq:blowup-nondeg2} and scale back to $u$  to conclude \eqref{eq:nondeg_t}.
With the nondegeneracy estimate \eqref{eq:nondeg_t} at hand,  by the similar argument as in  step 2 in the proof for \cite[Theorem 4.22]{BDGP20}, we conclude that $g \in H^{1,\frac12}$ for $p>2(n+2)$.

In the end we show the H\"older $\frac{1}{2}$- regularity of $g$ in $t$. Applying  \eqref{eq:nondeg_t} to $u(x',0,s)$ and the regularity in $t$ of $u$ in Theorem \ref{thm:opt}, we have 
$$u(x',0,t) \geq c |(x_{n-1}- g(x'',s)|^{\frac{3}{2}} - C |t-s|^{\frac{3}{4}}$$
for all $(x',0,s)\in Q'_{\delta/2}$ such that $x_{n-1}>g(x'',s)$ and for all $t\in [-\delta^2/4,0]$. Taking $r:=x_{n-1}- g(x'',s)>0$, we get from the above inequality 
\begin{align}\label{eq:pos}
u(x'', g(x'',s)+ r,0, t) \geq c r^{\frac{3}{2}}- C |t-s|^{\frac{3}{4}} >0
\end{align}
provided that  $C |t -s|^{\frac{3}{4}} \leq c r^{\frac{3}{2}}$.
Notice that  \eqref{eq:pos}  and   \eqref{eq:graph}  imply that 
$g(x'',t)< x_{n-1}= g(x'',s) + r.$
By choosing $r$ such that
$c r^{\frac{3}{2}} = 2 C|t- s|^{\frac{3}{4}},$
we find that
\begin{align*}
g(x'',t)< g(x'',s)+ C |t -s|^{\frac{1}{2}}
\end{align*}
for a possibly different $C$. Interchanging  $t$ and $s$, we thus conclude
$$|g(x'',t) - g(x'',s)| < C |t -s|^{\frac{1}{2}}.$$
Now for $p>n+2$ and $f$ satisfies \eqref{eq:support}, we use Remark \ref{rmk:opt_h} and proceed similarly to conclude that $g \in H^{1,\frac{1+\gamma}{3}}$ with $\gamma=1-\frac{n+2}{p}.$
This completes the proof of the theorem.
\end{proof}

\appendix 

\section{Appendix}
\subsection{Reduction to off-diagonal}
\begin{prop}\label{prop:off_diag}
Let $u\in W^{1,1}_2(Q_1^+)$, $\nabla u\in L^\infty(Q_1^+)$ be a solution of \eqref{eq:main} with $f\in L^p(Q_1^+)$, the coefficients $a^{ij}\in W^{1,1}_p(Q_1^+)$, $p>n+2$, and $a^{ij}$ being uniformly elliptic. Then there exists an open neighborhood $U$ of $Q_{1/2}'$ in $Q_1^+$ and a $W^{2,2}_p$ diffeomorphism $T:U \rightarrow T(U)$ with  $T(x,t)=(T_1(x,t),...,T_{n-1}(x,t),x_n,t)=:(y,\tau)$, such that if $v=u \circ T^{-1}$, then $v$ solves
\begin{equation*}
\begin{split}
\p_{\tau} v - \p_i(b^{ij}(y,\tau)\p_j v)= \tilde{f} &\text{ in } T(U)\cap\{y_n>0\},\\
v\geq 0, \quad \p_{n} v \geq 0,\quad v\p_{n} v=0 &\text{ on } T(U)\cap\{y_n=0\},
\end{split}
\end{equation*}
where $\tilde{f} \in L^p(T(U))$ and $(b^{ij})\in W^{1,1}_p(T(U))$ with $b^{nj}(y',0,\tau)=0$ on $T(U)\cap \{y_n=0\}$, $j\in \{1,\cdots, n-1\}$.
\end{prop}

\begin{proof}
The proof follows along the same line as in \cite[Page 1183]{Ural89}. Notice that due to the uniform ellipticity $\frac{a^{nj}}{a^{nn}} \in W^{1,1}_p(Q_1^+)$, $j\in \{1,\cdots, n-1\}$. Without loss of generality, in view of Sobolev extension, we can assume that $\frac{a^{nj}}{a^{nn}} \in W^{1,1}_p(\R^{n}_+ \times \R)$. By the trace theorem, $\frac{a^{nj}}{a^{nn}} \in W^{1-1/p,1-1/p}_p(\R^{n-1}\times \R).$ Here, $W^{1-1/p,1-1/p}_p(\R^{n-1}\times \R)$ denotes the fractional Sobolev space $W^{1-1/p,p}(\R^n)$.

 Let $\phi:\R \rightarrow \R$ be a smooth cut-off function such that $\phi=1$ in $(-1,1)$ and $\phi=0$ in $\R \setminus (-2,2)$. There exist $T_j:\R^n_+\times \R\rightarrow \R$, $j\in \{1,\cdots, n-1\}$, such that
\begin{align}\label{eq:def_T}
T_j(x',0,t)=x_j\phi(x_j) \;\;\;\text{and} \;\;\; \p_nT_j(x',0,t)=-\frac{a^{nj}}{a^{nn}}(x',0,t).
\end{align}
 Additionally, by the trace theorem $T_j \in W^{2,2}_p(\R^{n}_+\times \R)$ with 
\begin{align*}
\|T_j\|_{W^{2,2}_p} &\lesssim \|(a^{nn})^{-1}a^{nj}\|_{W^{1-1/p,1-1/p}_p(\R^{n-1}\times \R)}+1\\
& \lesssim \|(a^{nn})^{-1}a^{nj}\|_{W^{1,1}_p(\R^{n}_+\times \R)}+1\leq C,
\end{align*}
where $C$ depends on $\|a^{ij}\|_{W^{1,1}_p(Q_1^+)}$ and the ellipticity constants. 
We now define $T\in W^{2,2}_p(\R^{n}_+\times \R ;\R^{n+1})$ as follows:
$$T(x',x_n,t)=(T_1(x',x_n,t),...,T_{n-1}(x',x_n,t),x_n,t).$$
Notice that $T \in H^{1+\alpha,1+\alpha}(\R^{n}_+\times \R;\R^{n+1})$ with $\alpha:=1-\frac{n+1}{p}$ by the Sobolev embedding.  Also, it is easy to see that $\det(JT)(x',0,t)=1$ for $(x',t) \in Q_{1}'.$ Therefore by inverse function theorem there exists a open neighborhood $U$ in $Q_1^+$ of $Q_{1/2}'$ such that $T$ is a $W^{2,2}_p$ diffeomorphism. From now on we denote $T(x,t)=:(y,\tau)$. We now define $v(y,\tau):=u(T^{-1}(y,\tau))$ and $ b^{kl}(y,\tau)=\sum a^{ij}(\p_{x_i}T_k) (\p_{x_j}T_l)(T^{-1}(y,\tau)).$ Then it is straightforward to check that
\begin{align*}
\p_k(b^{k l}\p_l v) -\p_{\tau}v=\tilde{f} \;\;\;\text{in}\;\;\;T(U)\cap\{y_n>0\},
\end{align*}
where $\tilde{f}=\left(f+\sum a^{ij}(\p_{x_ix_j}T_k) (\p_{x_j}u \p_{y_k}x_j)-\sum (\p_{x_j}u \p_{y_k}x_j)(\p_{t}y_k)\right)(T^{-1}(y,\tau)).$\\
It is straightforward that $T$ preserves the uniform ellipticity and the regularity of the coefficients. Furthermore, 
using $\nabla_x u \in L^\infty$ and the regularity of $T$, we have that $\tilde{f}\in L^p(T(U))$ if $f\in L^p(Q_1^+)$.  It remains to check that $(b^{k\ell})$ satisfies the off-diagonal condition. Indeed, by \eqref{eq:def_T} we have that for $i<n$, $\p_{x_i}T_k =\delta_{ik}$ on $Q'_1$. Consequently, we have for $l<n$
\begin{align*}
b^{nl}(y',0,\tau)&=\sum a^{ij}(\p_{x_i}T_n) (\p_{x_j}T_l)(T^{-1}(y,\tau))\\
&=\sum_{j<n} a^{ij}\delta_{ni} \delta_{lj}(T^{-1}(y,\tau))-\sum a^{in}\delta_{ni} \frac{a^{nl}}{a^{nn}}(T^{-1}(y,\tau))\\
&=a^{nl}-a^{nl}=0,
\end{align*}
where we have used \eqref{eq:def_T}. This completes the proof.
\end{proof}

\subsection{Parabolic $L^\infty-L^2$ estimate}
We first prove the following sub-mean value property for subsolutions:
\begin{lem}\label{lem:sub_mean_value}
Let $u:\R^n\times [s,s+T]\rightarrow \R_+$, $s\in \R$, $T>0$, be a nonnegative function with Tychonoff type growth condition 
\begin{align}\label{eq:T}
\int_{s}^{s+T}\int_{\R^n} \left(|\nabla u(x,t)|^2+u(x,t)^2\right) e^{-\alpha_0|x|^2} \ dxdt<\infty
\end{align}
and satisfies 
\begin{align*}
\di(A\nabla u)-\p_t u \geq -|f|
\end{align*}
for some $f \in L^p(\R^n\times [s,s+T])$, $p>n+2$, where $A=A(x,t)$ is a $n\times n$-symmetric positive definite matrix with 
\begin{align*}
 \lambda^{-1} |\xi|^2\leq \langle A\xi, \xi\rangle\leq \lambda |\xi|^2,\quad \forall \xi\in \R^{n+1}
\end{align*}
for some $\lambda>0$ and $A$ is constant outside a compact set. Then there exists $t_0\in (0,T]$ depending on $\alpha_0, \lambda$ ($t_0=T$ if $\alpha_0=\alpha_0(\lambda)$ is sufficiently small), such that for all $(x,t)\in \R^n\times (s,s+t_0)$,
\begin{align}\label{eq:sub_mean}
u(x,t) \le \int_{\R^n} u(y,s) \G(x,y;s,t)\ dy +\int_{s}^t \int_{\R^n}|f(y,\tau)|\G(x,y;\tau,t)\ dy d\tau,
\end{align}
where $\G$ is the fundamental solution to the operator $\di (A\nabla)-\p_t$. 
\end{lem}
\begin{proof}
We start by taking a cut-off function in space $\zeta_R \in C_c^{\infty}(B_{R+1})$ such that $0 \le \zeta_R \le 1$ and $\zeta_R\equiv 1$ in $B_R.$ Then, $w=u\zeta_R$ solves 
\begin{align*}
\operatorname{div}(A \nabla w)- w_t \ge -|f| \zeta_R + \di(A \nabla \zeta_R) u  + 2\langle A \nabla u, \nabla \zeta_R\rangle.
\end{align*}
We now use the comparison principle in $B_{R+1}\times [s, s+T]$ to get, for all $(x,t)\in B_{R+1}\times [s,s+T]$,
\begin{align*}
w(x,t) &\le \int_{\R^n} u(y,s)\zeta_R \G(x,y;s,t)\ dy +\int_{s}^t \int_{\R^n}|f(y,\tau)|\G(x,y;\tau,t)\ dy d\tau\\
&+\int_{s}^t \int_{\R^n}|u\di(A \nabla \zeta_R) + 2\langle A \nabla u, \nabla \zeta_R\rangle|\G(x,y;\tau,t)\ dy d\tau.
\end{align*}
Note that due to the support assumption of $\nabla A$ and $\nabla\zeta_R$, to show the sub-mean value property, it suffices to show that there exists $t_0$ such that for any $t\in (s,s+t_0)$,
\begin{align}\label{eq:Rinfty}
\int_{s}^t \int_{B_{R+1}\setminus B_R}\left(|u\di( \nabla \zeta_R)|+ |\langle \nabla u, \nabla \zeta_R\rangle|\right)\G(x,y;\tau,t)\ dy d\tau \to 0,\quad R\rightarrow \infty.
\end{align}
From the estimate for the fundamental solution, cf. \cite[Chapter 1]{F64}, there is a constant $K>0$ depending only on $n$ such that given any $(x,t)$ with $\tau<t$,
\begin{align}\label{eq:sub_mp1}
\G(x,y;\tau,t) \le \frac{K}{(4\pi (t-\tau))^{n/2}}e^{-\frac{|x-y|^2}{4\lambda(t-\tau)}}.
\end{align} 
For $(x,t)$ fixed, if $R\gg 2|x|$, then $|x-y|\geq \frac{1}{2}|y|$ for each $y\in B_{R+1}\setminus B_R$. Thus with $t_0=\min\{\frac{1}{8\lambda\alpha_0},T\}$ we have
$
e^{-\frac{|x-y|^2}{4\lambda(t-\tau)}}\leq e^{-\alpha_0|y|^2}$ for any  $(x,t)\in \R^n\times (s,s+t_0)$ and $\tau\in (s,t)$. 
This combined with the growth assumption of $u$, cf. \eqref{eq:T}, and Cauchy-Schwarz  yields that 
\begin{align*}
\int_s^t\int_{B_{R+1}\setminus B_R}|u\di(\nabla \zeta_R)|\G(x,y;\tau, t)\ dyd\tau\rightarrow 0 \text{ as } R\rightarrow \infty.
\end{align*}
The proof for the gradient term in \eqref{eq:Rinfty} follows by the same argument.
\end{proof}

\begin{rmk}
The growth condition \eqref{eq:T} is not only on $u$ but also on $\nabla u$. However, this is not restrictive, since the bound on $|\nabla u|$ can be obtained via a global (or local) energy estimate for subsolutions. More precisely, if $u$ satisfies
\begin{align*}
\int_s^{s+T}\int_{\R^n}|u(x,t)|^2 e^{-2\alpha_0|x|^2}\ dxdt+\int_{\R^n} |u(x,s)|^2 e^{-2\alpha_0|x|^2}\ dx<\infty,
\end{align*}
then by plugging in $\eta=ue^{-2\alpha_0|x|^2}$ in the weak formulation $\int_s^{s+T}\int_{\R^n} A\nabla u\cdot\nabla \eta + \p_t u \eta\leq \int_s^{s+T} \int_{\R^n} |f| \eta$, we can obtain
 that $\int_s^{s+T}\int_{\R^n}|\nabla u|^2 e^{-\alpha_0|x|^2}\leq \int_s^{s+T}\int_{\R^n}(|u(x,t)|^2 +|f|^2)e^{-2\alpha_0|x|^2}+\int_{\R^n} |u(x,s)|^2 e^{-2\alpha_0|x|^2}$.
\end{rmk}

\begin{rmk}\label{rmk:Laplacian}
When the operator is the heat operator $\Delta-\p_t$ in $S_1$ (i.e. $A=id$, $s=-1$ and $T=1$), and $f=0$, from Lemma \ref{lem:sub_mean_value} and the semi-group property of the heat kernel one can conclude that the sub-mean value property \eqref{eq:sub_mean} holds for all $(x,t)\in \R^n\times (-1,0)$. In particular, under the assumption
\begin{align*}
\int_{S_1\setminus S_\delta} (|\nabla u(x,t)|^2+|u(x,t)|^2 ) G(x,t)\ dxdt <\infty,\ \forall 0<\delta \ll 1,
\end{align*}
where $G(x,t)=(-4\pi t)^{-n/2}e^{|x|^2/t}$ for $t<0$ is the standard Gaussian, the sub-mean value property \eqref{eq:sub_mean} holds true.

We also note that if $u$ and $\nabla u$ have at most polynomial growth, then one can take $t_0=T$ in Lemma \ref{lem:sub_mean_value}.
\end{rmk}

\begin{lem}\label{lem:linfty}
Let $u:\R^n\times [-1,1]\rightarrow \R_+$ be a nonnegative function with Tychonoff type growth \eqref{eq:T} and satisfies 
\begin{align*}
\di(A \nabla u)-\dt u \ge -|f|
\end{align*}
for some $f \in L^p(\R^n\times [-1,1])$, $p>n+2$, where $A=A(x,t)$ satisfies the same assumptions as Lemma \ref{lem:sub_mean_value}. 
Then if $\alpha_0$ is sufficiently small depending on $\lambda$,  for some absolute constant $c_0\in (0,1)$ and $\gamma=1-\frac{n+2}{p}$ we have
\begin{align*}
\sup_{B_r\times [-c_0r^2,c_0r^2]}u \lesssim H_r(u)^{1/2} + \|f\|_{L^p(\R^n\times [-1,1])}r^{1+\gamma}, \quad \forall r\in (0,1).
\end{align*}
Here $H_r(u)=\frac{1}{r^2}\int_{A_{r/2,r}} u^2 G$ is defined in Proposition \ref{prop:compactness}.
\end{lem}

\begin{rmk}\label{rmk:linfty}
Lemma \ref{lem:linfty} is different from the classical $L^\infty-L^2$ estimate, in the sense that the $L^\infty$ norm is taken over a full parabolic cylinder which includes future time. 
\end{rmk}
\begin{rmk}
It is enough to assume that $u$ and $f$ satisfy the hypothesis of Lemma \ref{lem:linfty} in $\R^n\times [-1,c_0r^2]$ instead of $\R^n\times [-1,1].$
\end{rmk}

\begin{proof}
The proof is similar as that for Theorem 7.3(iii) in \cite{DGPT}. We denote the fundamental solution of the operator $\di(A\nabla )-\dt$ by $\G$. Since $u$ is a nonnegative subsolution with growth assumption at infinity, for given $s\in [-1,0]$ and $(x,t)\in \R^n\times (s,1)$, by Lemma \ref{lem:sub_mean_value} we have
\begin{align}\label{eq:sub_mp}
u(x,t) \le \int_{\R^n} u(y,s)\G(x,y;s,t) dy + \int_{s}^{t}\int_{\R^n} |f(y,\tau)|\G(x,y;\tau,t) dyd\tau.
\end{align}
We recall the bound \eqref{eq:sub_mp1} for the fundamental solution $\G$. 
Let $t\in [-\frac{1}{100}r^2, \frac{1}{100}r^2]$ and $s \in (-r^2,-\frac{99}{100}r^2]$. Then $\frac{1}{2}|s|<t-s<\frac{5}{4\lambda}|s|.$ Thus using \eqref{eq:sub_mp1} we have
\begin{align*}
\G(x,y;s,t)\le \frac{K 2^{\frac{n}{2}}}{(4\pi |s|)^{n/2}}e^{-\frac{|x-y|^2}{5|s|}}\le \frac{K 2^{n/2}}{(4\pi |s|)^{n/2}}e^{-\frac{|y|^2}{6|s|}}e^{\frac{|x|^2}{|s|}}
\end{align*}
where the last inequality is due to $|x-y|^2  \ge \frac{5}{6}|y|^2-5|x|^2$.
Since $\frac{|x|^2}{s}\lesssim 1$ for $x\in B_r$ and $s\in (-r^2,-\frac{99}{100}r^2]$, one has
\begin{align*}
\G(x,y;s,t) \lesssim \sqrt{G(y,s)}\frac{1}{(4\pi |s|)^{n/4}}e^{-\frac{|y|^2}{24|s|}},
\end{align*}
where $G(y,s)$ is the standard Gaussian.
Hence using Cauchy-Schwarz we obtain from \eqref{eq:sub_mp} that for $(x,t)\in B_r\times [-\frac{1}{100}r^2, \frac{1}{100}r^2]$, 
 \begin{align*}
u(x,t) &\lesssim \left(\int_{\R^n} u^2(y,s)G(y,s) dy\right)^{1/2} \left(\int_{\R^n}\frac{1}{(4\pi |s|)^{n/2}} e^{-\frac{|y|^2}{12|s|}} dy\right)^{1/2}\\
& \ \ \  +\int_{s}^{t}\int_{\R^n} |f(y,\tau)|\frac{1}{|t-\tau|^{\frac{n}{2}}}e^{-\frac{|x-y|^2}{4\lambda(t-\tau)}} dyd\tau\\
&\lesssim \left(\int_{\R^n} u^2(y,s)G(y,s) dy\right)^{\frac{1}{2}} + \int_{s-t}^0\int_{\R^n}|f(x-y,\tau+t)|\frac{1}{|\tau|^{\frac{n}{2}}}e^{-\frac{|y|^2}{4\lambda |\tau|}}\ dyd\tau.
\end{align*}
To estimate the integral involving the inhomogeneity, we use   H\"older's inequality to get for $\frac{1}{p}+\frac{1}{q}=1$ and $\gamma=1-\frac{n+2}{p}$,
\begin{align*}
&\int_{s-t}^0\int_{\R^n}|f(x-y,\tau+t)|\frac{1}{|\tau|^{\frac{n}{2}}}e^{-\frac{|y|^2}{4\lambda |\tau|}}\ dyd\tau\\
&\lesssim \left(\int_{-r^2}^0\int_{\R^n}|f(x-y,\tau+t)|^p)\ dy d\tau\right)^{\frac{1}{p}}\left(\int_{-r^2}^0\int_{\R^n}\frac{1}{|\tau|^{\frac{nq}{2}}}e^{-\frac{q|y|^2}{4\lambda|\tau|}}\ dyd\tau\right)^{\frac{1}{q}}\\
&\lesssim r^{\gamma+1}\|f\|_{L^p(\R^n\times [-1,1])}
\end{align*}
Therefore, for any $x\in B_r$, $t\in [-\frac{1}{100}r^2, \frac{1}{100}r^2]$ and $s\in (-r^2, -\frac{99}{100}r^2]$ we have
\begin{align*}
u(x,t)&\lesssim \left(\int_{\R^n}u^2(y,s)G(y,s)\ dy\right)^{\frac{1}{2}} +  r^{\gamma+1}\|f\|_{L^p(\R^n\times [-1,1])}.
\end{align*}
Integrating over $s\in (-r^2,-\frac{99}{100}r^2]$, we obtain
\begin{align*}
u(x,t)
&\lesssim \left(\frac{1}{r^2}\int_{-r^2}^{-\frac{1}{4}r^2}\int_{\R^n}u^2(y,s)G(y,s)dyds\right)^{\frac{1}{2}}+ r^{\gamma+1}\|f\|_{L^p(\R^n\times [-1,1])}\\
&\lesssim H_r(u)^{1/2}+ r^{\gamma+1}\|f\|_{L^p(\R^n\times [-1,1])}.
\end{align*}
Thus we obtain the desired inequality with $c_0=\frac{1}{100}$.
\end{proof}

 \medskip

\subsection{$H^2$ and stability estimates}
\begin{lem}\label{lem:H2}
Let $u$ be a solution to \eqref{eq:main} with $\supp(u) \subset \overline{B_1^+} \times (-1,0].$ Suppose that $a^{ij}$ satisfies the assumptions (i)--(iii).  Then, there exists $\delta  >0$ sufficiently small depending on $n, p$, such that if  \begin{equation*}
\sup_{S_1^+} |a^{ij}-\delta^{ij}| \le \delta \quad \text{and} \quad \|\p_t a^{ij}\|_{L^{\frac{n+2}{2}}(S_1^+)}+\|\nabla a^{ij}\|_{L^{n+2}(S_1^+)} \le \delta,
\end{equation*}
 then for all $0<\rho<1/4$, $0\leq s\leq 2$ and for $\omega(x,t):=1+\frac{|x|}{\sqrt{|t|}}$, we have 
\begin{align}\label{eq:H2}
	& \int_{S^+_{\rho}}\left(|t||\nabla u|^2+ |t|^2  |\p_t u|^2 +  |t|^2| D^2u|^2\right)\omega^{2s}G \lesssim \int_{S^+_{2\rho}} (u^2+|t|^2f^2 )\omega^{2s}G.
\end{align}
\end{lem}

 \begin{proof}
 Our proof is inspired by that of \cite[Lemma 5.1]{DGPT}. 
  We approximate $u$ by $u^{\epsilon}$, which solves
\begin{equation*}
\begin{split}
\p_t  u^\epsilon -\p_i (a^{ij}_{\epsilon}\p_j  u^\epsilon)&=f^\epsilon \text{ in } B_R^+\times (-1,0],\\
a^{nn}_{\epsilon}\p_n  u^\epsilon &= \beta_\epsilon( u^\epsilon) \text{ on } B'_R\times (-1,0],\\
 u^\epsilon &= 0 \text{ on } (\p B_R)^+\times (-1,0],\\
 u^\epsilon(\cdot, -1)&=0 \text{ on } B_R^+\times \{-1\},
 \end{split}
\end{equation*}
where $R\geq 3$, $\beta_\epsilon$ is as defined in \eqref{eq:approx}; $a^{ij}_{\epsilon}$ and $f^\epsilon$ are appropriate mollifications of $a^{ij}$ and $f$.
	Writing the equation in the weak form, for any $\eta\in W^{1,0}_2( B_R^+\times (-1,0])$ vanishing a.e. on $(\p B_R)^+\times (t_1,t_2]$ and $(t_1,t_2] \subset (-1,0]$, we have 
	\begin{align}\label{eq:var_ineq}
		\int_{B_R^+\times (t_1,t_2]} (\p_tu^\epsilon \eta + a^{ij}_{\epsilon}\p_ju^\epsilon \p_i\eta)
		+ \int_{_{B'_R\times (t_1,t_2]}}\beta_\epsilon(u^\epsilon) \eta  = \int_{B_R^+\times (t_1,t_2]} f^\epsilon \eta.
	\end{align}
	Take a smooth function 
  $$h: [0,\infty) \rightarrow \R_+,\quad h(r)=(1+r)^s \iota(r),$$
   where $s\geq 0$ and $\iota: [0,\infty)\rightarrow [0,\infty)$ is a smooth cut-off function such that 
   \begin{align*}
   \iota'\leq 0,\quad \iota=1 \text{ in } [0,R-1], \quad \iota=0 \text{ in } [R,\infty).
   \end{align*} 
   We now fix  $\hat{\zeta_0}(y)=h(|y|) \in C_c^\infty(\R^n\setminus\{0\})\cap Lip(\R^n)$. Now, for the cut-off function $\hat{\zeta_0}$, define the family of homogeneous functions in $S_1$ as follows
	$$\zeta_k(x,t)=|t|^{k/2}\hat{\zeta}_0(x/\sqrt{|t|}),\quad k\in \N.$$
	\medskip
	
	\emph{Step 1: Gradient estimate.} We show that  for any $0<\rho <1/2$, there exists $C_{n,s}$ such that
	\begin{align*}
\int_{S_{\rho}^+}|t| \omega^{2s} |\nabla u|^2 G   \le C_{n, s} \int_{S_{2\rho}^+}   \omega^{2s}(u^2 + t^2 f^2) G. 
\end{align*}
\emph{Proof of Step 1:} We now take $r \in [\rho,2\rho]$ and $\epsilon'>0$ be any small number (Note that $\epsilon'$ is $\delta$ in \cite{DGPT} and the purpose of $\epsilon'$ is to avoid the singularity of the Gaussian kernel at $t=0$). We plug $\eta=u^\epsilon \zeta_1^2G$ in \eqref{eq:var_ineq} to get 
	\begin{align*}
		\int_{A_{\epsilon',r}^+} \p_tu^\epsilon u^\epsilon \zeta_1^2G + \int_{A_{\epsilon',r}^+} a^{ij}_{\epsilon}\p_iu^\epsilon \p_j(u^\epsilon \zeta_1^2G)
		+\int_{A'_{\epsilon',r}}\beta_\epsilon(u^\epsilon)u^\epsilon \zeta_1^2G = \int_{A_{\epsilon',r}^+} f^\epsilon u^\epsilon \zeta_1^2G.
	\end{align*}
On using $s \beta_\epsilon (s) \ge 0$, we obtain
\begin{align}\label{eq:grad_est_start}
		\int_{A_{\epsilon',r}^+} \p_tu^\epsilon u^\epsilon \zeta_1^2G + \int_{A_{\epsilon',r}^+} a^{ij}_{\epsilon}\p_iu^\epsilon \p_j(u^\epsilon \zeta_1^2G)
		 \le \int_{A_{\epsilon',r}^+} f^\epsilon u^\epsilon \zeta_1^2G.
	\end{align}
	We can re-write the above equation as follows
	\begin{align*}
		&\int_{A_{\epsilon',r}^+} \p_tu^\epsilon u^\epsilon \zeta_1^2G +\int_{A_{\epsilon',r}^+} a^{ij}_{\epsilon}\p_iu^\epsilon \p_ju^\epsilon \zeta_1^2G + \int_{A_{\epsilon',r}^+} a^{ij}_{\epsilon}\p_iu^\epsilon \p_j \zeta_1^2 u^\epsilon G \\
&+\int_{A_{\epsilon',r}^+} a^{ij}_{\epsilon}\p_iu^\epsilon \p_jG \zeta_1^2 u^\epsilon 
		 \le \int_{A_{\epsilon',r}^+} f^\epsilon u^\epsilon \zeta_1^2G.
	\end{align*}
Use the ellipticity of $a^{ij}_{\epsilon}$ in the second term of left hand side and $a^{ij}_{\epsilon}=a^{ij}_{\epsilon}-\delta^{ij}+\delta^{ij}$ in the fourth term of left hand side to find (recall that $Z=x\cdot \nabla +2t\p_t$)
 \begin{align}\label{eq:h2s1}
		&\int_{A_{\epsilon',r}^+} \frac{1}{4t}Z((u^\epsilon)^2) \zeta_1^2G +\frac{3}{4}\int_{A_{\epsilon',r}^+} |\nabla u^\epsilon|^2 \zeta_1^2G + \int_{A_{\epsilon',r}^+} a^{ij}_{\epsilon}\p_iu^\epsilon \p_j \zeta_1^2 u^\epsilon G \\
& +\int_{A_{\epsilon',r}^+} (a^{ij}_{\epsilon}-\delta^{ij})\p_iu^\epsilon \p_jG \zeta_1^2 u^\epsilon 
		 \le \int_{A_{\epsilon',r}^+} f^\epsilon u^\epsilon \zeta_1^2G.\notag
	\end{align}
 We handle the first term in left hand side of \eqref{eq:h2s1} by arguing as in \cite{DGPT} to obtain
\begin{align}\label{eq:h2s11}
\int_{A_{\epsilon',r}^+} \frac{1}{4t}Z((u^\epsilon)^2) \zeta_1^2G \ge -r^2 \int_{\R^n_+}u^\epsilon(\cdot, -r^2) \zeta_0^2G(\cdot, -r^2).
\end{align}
We now estimate the fourth term in \eqref{eq:h2s1} as follows. We first use $\sup |a^{ij}_{\epsilon}-\delta^{ij}| \lesssim \delta$, $|\nabla G|\leq \frac{|x|}{2|t|}G$ and then Young's inequality to find
\begin{align*}
\left|\int_{A_{\epsilon',r}^+} (a^{ij}_{\epsilon}-\delta^{ij})\p_iu^\epsilon \p_jG \zeta_1^2 u^\epsilon\right| &\le \delta \int_{A_{\epsilon',r}^+} \frac{|x|}{2|t|} |\nabla u^\epsilon|G \zeta_1^2 |u^\epsilon|\\
& \le \frac{\delta}{2} \int_{A_{\epsilon',r}^+} |\nabla u^\epsilon|^2 G \zeta_1^2 + \delta \int_{A_{\epsilon',r}^+} \frac{|x|^2}{8|t|^2} \zeta_1^2 |u^\epsilon|^2 G.
\end{align*}
Now an application of Claim A.1 in \cite{DGPT} for $v=\zeta_0 u^{\epsilon}$ and recalling that $\zeta_1= \sqrt{|t|} \zeta_0$ yields
\begin{align*}
\int_{A_{\epsilon',r}^+} \frac{|x|^2}{|t|^2} \zeta_1^2 |u^\epsilon|^2 G\lesssim \int_{A_{\epsilon',r}^+} (\zeta_0^2+|\nabla\zeta_1|^2)|u^\epsilon|^2G+ \zeta_1^2 |\nabla u^\epsilon|^2 G
\end{align*}
Combining the above two inequalities we thus obtain
\begin{align}\label{eq:h2s14}
\left|\int_{A_{\epsilon',r}^+} (a^{ij}_{\epsilon}-\delta^{ij})\p_iu^\epsilon \p_jG \zeta_1^2 u^\epsilon\right| \lesssim \delta \int_{A_{\epsilon',r}^+} \zeta_1^2|\nabla u^\epsilon|^2 G  +  (\zeta_0^2 + |\nabla \zeta_1|^2)|u^\epsilon|^2 G.
\end{align}
For the third term in \eqref{eq:h2s1} we use again Young's inequality and get
\begin{align}\label{eq:h2s13}
\int_{A_{\epsilon',r}^+}a^{ij}_\epsilon \p_iu^\epsilon\p_j\zeta_1^2 u^\epsilon G\leq \frac{1}{4}\int_{A_{\epsilon',r}^+} |\nabla u^\epsilon|^2 \zeta_1^2 G + 8\int_{A_{\epsilon',r}^+} |u^\epsilon|^2 |\nabla\zeta_1|^2 G.
\end{align}
We now use \eqref{eq:h2s11}, \eqref{eq:h2s14} and \eqref{eq:h2s13} in \eqref{eq:h2s1} to conclude 
 \begin{align*}
\left(\frac{1}{2}-C_n\delta\right)\int_{A_{\epsilon',r}^+} |\nabla u^\epsilon|^2 \zeta_1^2G 
		 &\le r^2 \int_{\R^n_+}u^\epsilon(\cdot, -r^2) \zeta_0^2G(\cdot, -r^2) \\
		 &+ (8+C_n\delta) \int_{A_{\epsilon',r}^+}  (\zeta_0^2 + |\nabla \zeta_1|^2)|u^\epsilon|^2G +\int_{A_{\epsilon',r}^+} f^\epsilon u^\epsilon \zeta_1^2 G.\notag
\end{align*}
We now use $\zeta_1^2=\zeta_0\zeta_2$ and Young's inequality in the last term on the right-hand side of the above equation to get 
\begin{align*}
\left(\frac{1}{2}-C_n\delta\right)\int_{A_{\epsilon',r}^+} |\nabla u^\epsilon|^2 \zeta_1^2G 
		 &\le r^2 \int_{\R^n_+}u^\epsilon(\cdot, -r^2) \zeta_0^2G(\cdot, -r^2) \\
		 &+ 2(8+C_n\delta) \int_{A_{\epsilon',r}^+}  [(\zeta_0^2 + |\nabla \zeta_1|^2)|u^\epsilon|^2 +|f^\epsilon|^2  \zeta_2^2] G.\notag
\end{align*}
Integrating over $r\in [\rho, 2\rho]$ and choosing $\delta \leq \delta_0=\delta_0(n)$ small we get
\begin{align}\label{eq:grad_est_eps}
\int_{A_{\epsilon',\rho}^+} |\nabla u^\epsilon|^2 \zeta_1^2 G\lesssim \int_{A_{\epsilon',2\rho}^+} [(\zeta_0^2 + |\nabla \zeta_1|^2)|u^\epsilon|^2 + |f^\epsilon|^2\zeta_2^2] G.
\end{align}
Now we first let  $\epsilon\rightarrow 0+$ and then let $R\rightarrow \infty$ and $\epsilon'\rightarrow 0+$. Noticing that 
\begin{align*}
\lim_{R\rightarrow \infty}\lim_{\epsilon\rightarrow 0}\int_{A_{\epsilon',2\rho}^+}\zeta_0^2|u^\epsilon|^2 G= \int_{A_{\epsilon',2\rho}^+}\omega^{2s}u^2 G,
\end{align*}
and
\begin{align*}
\lim_{R\rightarrow \infty}\lim_{\epsilon\rightarrow 0}\int_{A_{\epsilon',2\rho}^+}|\nabla\zeta_1|^2|u^\epsilon|^2 G=s^2  \int_{A_{\epsilon',2\rho}^+} \omega^{2s-2}u^2 G\leq s^2 \int_{A_{\epsilon',2\rho}^+} \omega^{2s}u^2 G,
\end{align*}
 we obtain the claimed estimate.
This completes the proof for \emph{Step 1}.
\\

\emph{Step 2: $L^2$-time derivative and $L^{\infty}$-gradient estimates.} We now plug $\eta =\partial_t u^{\epsilon}\zeta_2^2G$ in \eqref{eq:var_ineq}.
Then, for $\varsigma \in [-\rho^2,-(\epsilon')^2] $ and for any $\tilde\rho\in [\rho, 2\rho]$,  we obtain 
\begin{equation}\label{eq:l_infty1}
\begin{split}
\int_{-\tilde{\rho}^2}^{\varsigma}	\int_{\R^n_+} |\p_tu^\epsilon|^2\zeta_2^2G+ \int_{-\tilde{\rho}^2}^{\varsigma}\int_{\R^n_+} a^{ij}_{\epsilon}\p_ju^\epsilon \p_i(\p_t u^\epsilon \zeta_2^2 G)\\
 + \int_{-\tilde{\rho}^2}^{\varsigma}\int_{\R^{n-1}} \beta_\epsilon(u^\epsilon)\p_tu^\epsilon \zeta_2^2 G = \int_{-\tilde{\rho}^2}^{\varsigma}\int_{\R^n_+} f^\epsilon \p_tu^\epsilon \zeta_2^2 G.
 \end{split}
\end{equation}
For the third term on left hand side of \eqref{eq:l_infty1} we apply an integration by parts in $t-$variable to get
	\begin{align*}
		&\int_{-\tilde{\rho}^2}^{\varsigma}\int_{\R^{n-1}} \beta_\epsilon(u^\epsilon)\p_t u^\epsilon \zeta_2^2 G= \int_{-\tilde{\rho}^2}^{\varsigma}\int_{\R^{n-1}} \p_tB_\epsilon(u^\epsilon)\zeta_2^2 G  \\
		& =-\int_{-\tilde{\rho}^2}^{\varsigma}\int_{\R^{n-1}}B_\epsilon(u^\epsilon) \p_t(\zeta_2^2 G) +\int_{\R^{n-1}}B_\epsilon(u^\epsilon) \zeta_2^2G(\cdot, \varsigma)-B_\epsilon(u^\epsilon)\zeta_2^2 G(\cdot, -\tilde{\rho}^2),
	\end{align*}
	where $B_\epsilon:\R\rightarrow \R$ is the anti-derivative of $\beta_\epsilon$ with $B_\epsilon(0)=0$. Since
	$|B_\epsilon(u^\epsilon)|\leq C\epsilon \text{ on } Q'$  for $C=C(n,p,\|f\|_{L^p}, \|a^{ij}\|_{W^{1,1}_p})$ (cf. Lemma 4 and Lemma 5 in \cite{AU}). Also, we have
	\begin{align*}
	\p_t(\zeta_2^2 G)=2t\zeta_0^2G+|t|^{\frac{1}{2}}\zeta_0\langle\nabla \zeta_0,x\rangle G -\frac{|x|^2}{4} \zeta_0^2G -\frac{n}{2}t\zeta_0^2 G.
\end{align*}	 
Using $\zeta_0$ and $\sqrt{|t|}|\nabla \zeta_0|$ are bounded, and that $u^\epsilon$ is supported in $B_R$, we have that for all $t\in (-1,0)$ 
	\begin{equation}\label{eq:in_bdr}
	\begin{split}
		|\int_{-\tilde{\rho}^2}^{\varsigma}\int_{\R^{n-1}}B_\epsilon(u^\epsilon)\p_t(\zeta_2^2 G)|&\leq  C_{R,\epsilon'}\epsilon \int_{-\tilde{\rho}^2}^{\varsigma}\int_{B'_R}\frac{e^{\frac{|x'|^2}{4t}}}{|t|^{\frac{n}{2}}}dx'dt \leq  C_{R,\epsilon'}\epsilon,\\
		\int_{\R^{n-1}}B_\epsilon(u^\epsilon) \zeta_2^2G(\cdot, \tilde \varsigma) &\le C_{R,\epsilon'}\epsilon,\quad \forall \tilde \varsigma \in [-\tilde\rho^2, -(\epsilon')^2]
		\end{split}
		\end{equation}
	with $C$ particularly independent of $\epsilon$.  For the term on the right hand side of \eqref{eq:l_infty1}, we simply use Young's inequality to get
	\begin{align}\label{eq:in_rhs}
		|\int_{-\tilde{\rho}^2}^{\varsigma}\int f^\epsilon \p_tu^\epsilon \zeta_2^2G|\leq \int_{-\tilde{\rho}^2}^{\varsigma} \int |f^\epsilon|^2\zeta_2^2G + \frac{1}{4}                                                                                                                                                                                                                                                                                                                              \int_{-\tilde{\rho}^2}^{\varsigma} \int |\p_tu^\epsilon|^2 \zeta_2^2G.
	\end{align}
Applying \eqref{eq:in_bdr} and \eqref{eq:in_rhs} to \eqref{eq:l_infty1}  we find
\begin{equation}\label{eq:l_infty2}
\begin{split}
 \int_{-\tilde{\rho}^2}^{\varsigma} \int |\p_tu^\epsilon|^2\zeta_2^2G+\int_{-\tilde{\rho}^2}^{\varsigma}\int a^{ij}_{\epsilon}\p_ju^\epsilon \p_i(\p_t u^\epsilon \zeta_2^2G)\lesssim  C_{R,\epsilon'}\epsilon  + \int_{-\tilde{\rho}^2}^{\varsigma}\int |f^\epsilon|^2\zeta_2^2G.
 \end{split}
\end{equation}
\emph{Claim:} We have for some $C=C_n>0$,
\begin{align*}
X&:=\int_{-\tilde{\rho}^2}^{\varsigma}\int a^{ij}_{\epsilon}\p_ju^\epsilon \p_i(\p_t u^\epsilon \zeta_2^2G)\\
&\ge - \frac{17}{20}\int_{-\tilde{\rho}^2}^{\varsigma}\int  (\p_t u^{\epsilon})^2 \zeta_2^2  G+\left(\frac{1}{64}-C\delta\right)\int_{-\tilde{\rho}^2}^{\varsigma}\int |\nabla u^{\epsilon}|^2 \frac{|x|^2}{4t^2}   \zeta_2^2G\\
& \quad +\frac{1}{2}\int\langle A_{\epsilon} \nabla u^{\epsilon}, \nabla u^{\epsilon}\rangle \zeta_2^2 G(\cdot,\varsigma)-\frac{1}{2}\int\langle A_{\epsilon} \nabla u^{\epsilon}, \nabla u^{\epsilon}\rangle \zeta_2^2 G(\cdot,-\tilde\rho^2)\\
&\quad - C\delta \sup_{[-\tilde{\rho^2},-(\epsilon')^2]}\|\nabla u^{\epsilon}\zeta_2 \sqrt{G(\cdot, t)}\|^2_{L^2}-C \delta \int_{-\tilde{\rho}^2}^{\varsigma}\int |D^2 u^{\epsilon}|^2\zeta_2^2 G \\
&\quad -C \int_{-\tilde{\rho}^2}^{\varsigma} \int |\nabla u^{\epsilon}|^2  (|\nabla\zeta_2|^2 +\zeta_1^2) G.
\end{align*}
\emph{Proof of the claim:} Using Young's inequality and the ellipticity  $|a^{ij}_{\epsilon}\p_ju^{\epsilon}|^2 \le (1+\frac{1}{8})^2|\nabla u^{\epsilon}|^2$ we have
\begin{equation}\label{eq:h2s22}
\begin{split}
X&=\int_{-\tilde{\rho}^2}^{\varsigma}\int a^{ij}_{\epsilon}\p_ju^\epsilon \p_i\p_t u^\epsilon \zeta_2^2G\\
&\quad + 2\int_{-\tilde{\rho}^2}^{\varsigma}\int a^{ij}_{\epsilon}\p_ju^\epsilon  \p_t u^\epsilon \zeta_2 \p_i \zeta_2 G + \int_{-\tilde{\rho}^2}^{\varsigma}\int a^{ij}_{\epsilon}\p_ju^\epsilon \frac{x_i}{2t}   \p_t u^\epsilon \zeta_2^2 G\\
&\ge \int_{-\tilde{\rho}^2}^{\varsigma}\int a^{ij}_{\epsilon}\p_ju^\epsilon \p_i\p_t u^{\epsilon} \zeta_2^2G- 20\int_{-\tilde{\rho}^2}^{\varsigma}\int |\nabla u^\epsilon|^2 |\nabla \zeta_2|^2 G\\
& \quad - \frac{17}{20}\int_{-\tilde{\rho}^2}^{\varsigma}\int  (\p_t u^\epsilon)^2 \zeta_2^2  G - \frac{27}{64}\int_{-\tilde{\rho}^2}^{\varsigma}\int |\nabla u^\epsilon|^2 \frac{|x|^2}{4t^2}   \zeta_2^2 G.
\end{split}
\end{equation}
We now estimate the first term on the right hand side of \eqref{eq:h2s22}. To do this, first we observe that 
	\begin{align}\label{eq:2nd}
		\partial_t (\langle A_{\epsilon} \nabla u^{\epsilon}, \nabla u^{\epsilon}\rangle \zeta_2^2G)=2\langle A_{\epsilon} \nabla u^{\epsilon}, \nabla u^{\epsilon}_t\rangle\zeta_2^2G 
		+\langle \partial_t A_{\epsilon} \nabla u^{\epsilon}, \nabla u^{\epsilon}\rangle \zeta_2^2G \\
		 +2\langle A_{\epsilon} \nabla u^{\epsilon}, \nabla u^{\epsilon}\rangle \zeta_2 \partial_t\zeta_2 G+\langle A_{\epsilon} \nabla u^{\epsilon}, \nabla u^{\epsilon}\rangle\zeta_2^2 \p_tG\notag,
	\end{align}
	where $\partial_t A_{\epsilon}$ is the matrix with entries $(\partial_t A_\epsilon)^{ij}=\partial_t a^{ij}_{\epsilon}.$ We use fundamental theorem of calculus in $t$-variable in \eqref{eq:2nd}  to obtain
\begin{equation}\label{eq:l22}
\begin{split}
		& 2\int_{-\tilde{\rho}^2}^{\varsigma}\int a^{ij}_{\epsilon}\p_ju^\epsilon \p_i\p_t u^{\epsilon} \zeta_2^2 G
		  = \int\langle A_\epsilon \nabla u^{\epsilon}, \nabla u^{\epsilon}\rangle\zeta_2^2 G(\cdot,\varsigma)-\int\langle A_\epsilon \nabla u^{\epsilon}, \nabla u^{\epsilon}\rangle\zeta_2^2 G(\cdot,-\tilde\rho^2)\\
		 &- \int_{-\tilde{\rho}^2}^{\varsigma} \int \partial_t a^{ij}_{\epsilon}\p_ju^\epsilon \p_i u^\epsilon \zeta_2^2 G- 2\int_{-\tilde{\rho}^2}^{\varsigma} \int a^{ij}_{\epsilon}\p_ju^\epsilon \p_i u^\epsilon \zeta_2 \partial_t\zeta_2 G- \int_{-\tilde{\rho}^2}^{\varsigma} \int a^{ij}_{\epsilon}\p_ju^\epsilon \p_i u^\epsilon \zeta_2^2   \p_t G\\
		 &=:I_1+I_2+I_3+I_4+I_5.
	\end{split}
	\end{equation}
	As $\iota'\leq 0$ and $s\leq 2$ by our assumption, we have
	\begin{align*}
	\partial_t\zeta_2&=-\zeta_0+\frac{1}{2}s|y|(1+|y|)^{s-1}\iota(|y|)+\frac{|y|}{2}(1+|y|)^s \iota'(|y|) \\
	&= \iota(|y|) (1+|y|)^{s-1}\left(-1-(1-\frac{s}{2})|y|\right) + \frac{|y|}{2}(1+|y|)^s \iota'(|y|)\le 0.
	\end{align*}
	 Thus $I_4\geq 0$. To estimate $I_3$ we first apply H\"older's inequality and then Sobolev embedding to get
	\begin{align*}
		|I_3 |&\leq  \|\partial_t a^{ij}_{\epsilon}\|_{L^{\frac{n+2}{2}}}\|\nabla u^\epsilon\zeta_2 \sqrt{G}\|_{L^{\frac{2(n+2)}{n}}}^2\\
		%
		%
		&\lesssim  \|\partial_t a^{ij}_{\epsilon}\|_{L^{\frac{n+2}{2}}}\left(\sup_{[-\tilde{\rho^2},-(\epsilon')^2] }\|\nabla u^{\epsilon}\zeta_2 \sqrt{G}(\cdot, t)\|_{L^2}^2\right) ^{\frac{2}{n+2}}\|\nabla(\nabla u^\epsilon\zeta_2 \sqrt{G})\|_{L^2}^{\frac{2n}{n+2}}.
	\end{align*}
	We now recall $\|\partial_t a^{ij}_{\epsilon}\|_{L^{\frac{n+2}{2}}}\lesssim
\|\partial_t a^{ij}\|_{L^{\frac{n+2}{2}}}  =\delta$ and use Young's inequality to obtain
	\begin{equation}\label{eq:i_bdd}
	\begin{split}
	|I_3 |&\lesssim \delta \sup_{[-\tilde{\rho^2},-(\epsilon')^2]}\|\nabla u^{\epsilon}\zeta_2 \sqrt{G}(\cdot,t)\|^2_{L^2}+\delta\|\nabla(\nabla u^\epsilon\zeta_2 \sqrt{G})\|_{L^2}^{2}\\
	&\lesssim \delta \sup_{[-\tilde{\rho^2},-(\epsilon')^2]}\|\nabla u^{\epsilon}\zeta_2 \sqrt{G}\|^2_{L^2} + \delta\|D^2 u^\epsilon\zeta_2 \sqrt{G}\|_{L^2}^{2}\\
	&\quad +\delta\|\nabla u^\epsilon\cdot\nabla \zeta_2\sqrt{G}\|_{L^2}^{2}+\delta\|\nabla u^\epsilon\frac{|x|}{|t|}\zeta_2 \sqrt{G}\|_{L^2}^{2}. 
	\end{split} 
\end{equation} 
We bound $I_5$ from below as
\begin{align}\label{eq:s2l2l}
I_5&=\int_{-\tilde{\rho}^2}^{\varsigma} \int a^{ij}_{\epsilon}\p_ju^{\epsilon} \p_i u^{\epsilon} \zeta_2^2 \frac{|x|^2}{4t^2} G -\frac{n}{2} \int_{-\tilde{\rho}^2}^{\varsigma} \int a^{ij}_{\epsilon}\p_ju^{\epsilon} \p_i u^{\epsilon} \zeta_1^2   G\\
&\ge \frac{7}{8} \int_{-\tilde{\rho}^2}^{\varsigma} \int |\nabla u^{\epsilon}|^2  \zeta_2^2  \frac{|x|^2}{4t^2} G -n\int_{-\tilde{\rho}^2}^{\varsigma} \int |\nabla u^{\epsilon}|^2  \zeta_1^2   G,\notag
\end{align}
where in the last inequality we have used the ellipticity. Therefore, using  \eqref{eq:i_bdd} and \eqref{eq:s2l2l} in \eqref{eq:l22}, we can conclude the following bound for the first term in the right hand side of \eqref{eq:h2s22}: for some $C=C_n>0$,
\begin{equation}\label{eq:l22f}
	\begin{split}
		&2\int_{-\tilde{\rho}^2}^{\varsigma}\int a^{ij}_{\epsilon}\p_ju^{\epsilon} \p_i\p_t u^{\epsilon} \zeta_2^2 G
		 \\
		 &\ge \int\langle A_{\epsilon} \nabla u^{\epsilon}, \nabla u^{\epsilon}\rangle \zeta_2^2 G(\cdot,\varsigma)-\int\langle A_{\epsilon} \nabla u^{\epsilon}, \nabla u^{\epsilon}\rangle \zeta_2^2 G(\cdot,-\tilde\rho^2) \\
		&\quad -C\delta \sup_{[-\tilde{\rho^2},-(\epsilon')^2]}\|\nabla u^{\epsilon}\zeta_2 \sqrt{G}(\cdot, t)\|^2_{L^2} -C \delta \int_{-\tilde{\rho}^2}^{\varsigma}\int |D^2 u^{\epsilon}|^2\zeta_2^2 G \\
		&\quad +\left(\frac{7}{8} -C \delta\right) \int_{-\tilde{\rho}^2}^{\varsigma} \int |\nabla u^{\epsilon}|^2  \zeta_2^2 \frac{|x|^2}{4t^2} G -C \int_{-\tilde{\rho}^2}^{\varsigma} \int |\nabla u^{\epsilon}|^2  (|\nabla\zeta_2|^2 +\zeta_1^2) G.
	\end{split}
\end{equation}
Plugging \eqref{eq:l22f} in \eqref{eq:h2s22}, we thus get the claimed lower bound for $X$, which completes the proof for the claim.

With the claim at hand, \eqref{eq:l_infty2} becomes 
\begin{align*}
&\frac{3}{20}\int_{-\tilde{\rho}^2}^{\varsigma} \int |\p_tu^{\epsilon}|^2\zeta_2^2G+\left(\frac{1}{64}-C\delta\right)\int_{-\tilde{\rho}^2}^{\varsigma}\int |\nabla u^{\epsilon}|^2 \frac{|x|^2}{4t^2}   \zeta_2^2 G\\
& +\frac{1}{2}\int\langle A_{\epsilon} \nabla u^{\epsilon}, \nabla u^{\epsilon}\rangle \zeta_2^2 G(\cdot,\varsigma)- C\delta \sup_{[-\tilde{\rho^2},-(\epsilon')^2]}\|\nabla u^{\epsilon}\zeta_2 \sqrt{G(\cdot, t)}\|^2_{L^2}\\
 &\lesssim C_{R,\epsilon'}\epsilon+ \int_{-\tilde{\rho}^2}^{\varsigma}\int |f^\epsilon|^2\zeta_2^2G  + \delta \int_{-\tilde{\rho}^2}^{\varsigma}\int |D^2 u^{\epsilon}|^2\zeta_2^2 G. \notag \\
& \quad +\int_{-\tilde{\rho}^2}^{\varsigma} \int |\nabla u^{\epsilon}|^2  (|\nabla\zeta_2|^2 +\zeta_1^2) G+\int|\nabla u^{\epsilon}|^2 \zeta_2^2 G(\cdot,-\tilde\rho^2).\notag
\end{align*}
Notice that by ellipticity $\langle A_{\epsilon} \nabla u^{\epsilon}, \nabla u^{\epsilon} \rangle \ge \frac{7}{8}|\nabla u^{\epsilon}|^2.$ Then, we upper bound the right hand side by replacing $\varsigma$ with $-(\epsilon')^2$ to get 
\begin{align*}
&\frac{3}{20}\int_{-\tilde{\rho}^2}^{\varsigma} \int |\p_tu^{\epsilon}|^2\zeta_2^2\xi^2G+\left(\frac{1}{64}-C\delta\right)\int_{-\tilde{\rho}^2}^{\varsigma}\int |\nabla u^{\epsilon}|^2 \frac{|x|^2}{4t^2}   \zeta_2^2 G\\
& +\frac{7}{16}\int|\nabla u^{\epsilon}|^2 \zeta_2^2 G(\cdot,\varsigma)- C\delta \sup_{[-\tilde{\rho}^2,-(\epsilon')^2]}\|\nabla u^{\epsilon}\zeta_2 \sqrt{G(\cdot, t)}\|^2_{L^2}\\
 &\lesssim C_{R,\epsilon'}\epsilon+ \int_{-\tilde{\rho}^2}^{-(\epsilon')^2}\int |f^\epsilon|^2\zeta_2^2G  + \delta \int_{-\tilde{\rho}^2}^{-(\epsilon')^2}\int |D^2 u^{\epsilon}|^2\zeta_2^2 G. \notag \\
& \quad +\int_{-\tilde{\rho}^2}^{-(\epsilon')^2} \int |\nabla u^{\epsilon}|^2  (|\nabla\zeta_2|^2 +\zeta_1^2) G+\int|\nabla u^{\epsilon}|^2 \zeta_2^2 G(\cdot,-\tilde\rho^2).\notag
\end{align*}
Since the above inequality holds for arbitrary $\varsigma\in [-\tilde\rho^2, -(\epsilon')^2]$, we can now take $\sup$ in $\varsigma$ on the left hand side to find
\begin{equation*}
\begin{split}
	&\frac{3}{20}\int_{-\tilde{\rho}^2}^{-(\epsilon')^2} \int |\p_tu^{\epsilon}|^2\zeta_2^2G +\left(\frac{1}{64}-C\delta\right)\int_{-\tilde{\rho}^2}^{-(\epsilon')^2}\int |\nabla u^{\epsilon}|^2 \frac{|x|^2}{4t^2}   \zeta_2^2 G\\
	& + \left(\frac{7}{16}- C \delta\right) \sup_{[-\tilde{\rho^2},-(\epsilon')^2]}\|\nabla u^{\epsilon}\zeta_2 \sqrt{G}\|^2_{L^2}\\
	&\lesssim C_{R,\epsilon'}\epsilon+\int|\nabla u^{\epsilon}|^2 \zeta_2^2 G(\cdot,-\tilde\rho^2)+ \delta \int_{-\tilde{\rho}^2}^{-(\epsilon')^2}\int |D^2 u^{\epsilon}|^2\zeta_2^2 G\\
	&\quad + \int_{-\tilde{\rho}^2}^{-(\epsilon')^2}\int |f^\epsilon|^2\zeta_2^2G+\int_{-\tilde{\rho}^2}^{-(\epsilon')^2} \int |\nabla u^{\epsilon}|^2  (|\nabla\zeta_2|^2 +\zeta_1^2) G.
	\end{split}
\end{equation*}
We now choose $0<\delta<\delta_0(n)$ sufficiently small to conclude
\begin{equation}\label{eq:s2f}
\begin{split}
	&\int_{-\tilde{\rho}^2}^{-(\epsilon')^2} \int |\p_tu^{\epsilon}|^2\zeta_2^2G +\int_{-\tilde{\rho}^2}^{-(\epsilon')^2}\int |\nabla u^{\epsilon}|^2 \frac{|x|^2}{t^2}   \zeta_2^2 G + \sup_{[-\tilde{\rho^2},-(\epsilon')^2]}\|\nabla u^{\epsilon}\zeta_2 \sqrt{G}\|^2_{L^2}\\
	&\lesssim C_{R,\epsilon'}\epsilon+ \int_{-\tilde{\rho}^2}^{-(\epsilon')^2}\int |f^\epsilon|^2\zeta_2^2G+\int|\nabla u^{\epsilon}|^2 \zeta_2^2 G(\cdot,-\tilde\rho^2) \\
	 &+ \delta \int_{-\tilde{\rho}^2}^{-(\epsilon')^2}\int |D^2 u^{\epsilon}|^2\zeta_2^2 G+\int_{-\tilde{\rho}^2}^{-(\epsilon')^2} \int |\nabla u^{\epsilon}|^2  (|\nabla\zeta_2|^2 +\zeta_1^2) G.
	 \end{split}
\end{equation}

\medskip

	\emph{Step 3: $H^2$ estimates.} 
	We take $\p_{\mu}\eta$, where $\mu\in \{1,\cdots, n-1\}$ is a tangential direction in \eqref{eq:var_ineq} (with $r\in [\rho, 2\rho]$ and $\epsilon'>0$) to get
	\begin{align*}
		\int_{A_{\epsilon',r}^+} \p_tu^{\epsilon} \p_{\mu}\eta + \int_{A_{\epsilon',r}^+} \langle A_{\epsilon} \nabla u^{\epsilon}, \nabla \p_{\mu}\eta \rangle
		+ \int_{A'_{\epsilon',r}}\beta_\epsilon(u^\epsilon) \p_{\mu}\eta  = \int_{A_{\epsilon',r}^+} f^\epsilon \p_{\mu}\eta.
	\end{align*}
	Now, we integrate by parts in $\mu-$variable
	and take $\eta = \p_\mu u^\epsilon\zeta_2^2G$ to obtain  
	\begin{equation}\label{eq:H2mu}
	\begin{split}
		X_1+X_2+X_3+X_4&:=\int_{A_{\epsilon',r}^+} \p_{t\mu}u^\epsilon \p_\mu u^\epsilon\zeta_2^2G + \int_{A_{\epsilon',r}^+} \langle \p_{\mu} A_{\epsilon} \nabla u^\epsilon, \nabla (\p_\mu u^\epsilon\zeta_2^2G) \rangle \\
		&\quad + \int_{A_{\epsilon',r}^+} \langle  A_{\epsilon} \nabla \p_{\mu} u^\epsilon, \nabla (\p_\mu u^\epsilon\zeta_2^2G) \rangle +\int_{A_{\epsilon',r}^+} f^\epsilon \p_{\mu}(\p_\mu u^\epsilon\zeta_2^2G)  \\
		&=-\int_{A'_{\epsilon',r}}\beta'_\epsilon(u^\epsilon) (\p_{\mu} u^\epsilon)^2 \zeta_2^2G\leq 0,
		\end{split}
	\end{equation}
	where the last inequality follows from  $\beta'_\epsilon\geq 0$. First we  estimate the term $X_2$ on the left hand side of \eqref{eq:H2mu}:
	\begin{align}\label{eq:2ndH2}
	X_2&=\int_{A_{\epsilon',r}^+} \langle \p_{\mu} A_{\epsilon} \nabla u^\epsilon, \nabla (\p_\mu u^\epsilon) \rangle\zeta_2^2G\\
	&\quad  +2\int_{A_{\epsilon',r}^+} \langle \p_{\mu} A_{\epsilon} \nabla u^\epsilon, \zeta_2\nabla \zeta_2 \rangle\p_\mu u^\epsilon G+\int_{A_{\epsilon',r}^+} \langle \p_{\mu} A_{\epsilon} \nabla u^\epsilon, \nabla G \rangle\p_\mu u^\epsilon\zeta_2^2 \notag \\
	&\le \sqrt{\delta} \int_{A_{\epsilon',r}^+}  |\nabla (\p_\mu u^\epsilon)|^2 \zeta_2^2G+\frac{3}{\sqrt{\delta}}\int_{A_{\epsilon',r}^+}  |\p_\mu a^{ij}_\epsilon \p_j u^\epsilon|^2 \zeta_2^2G \notag\\
	&\quad   +\int_{A_{\epsilon',r}^+}|\nabla \zeta_2|^2 |\p_\mu u^\epsilon|^2 G+\frac{1}{2}\int_{A_{\epsilon',r}^+} |\p_\mu u^\epsilon|^2\frac{|x|^2}{16t^2}\zeta_2^2G. \notag
	\end{align}
	For the second term on the above right hand side, using H\"older's inequality, Sobolev embedding and $\|\nabla a^{ij}_{\epsilon}\|_{L^{n+2}} \lesssim \delta$ and arguing as in step 2 we get
	\begin{align*}
		\int_{A_{\epsilon',r}^+} |\p_\mu a^{ij}_{\epsilon}\p_j u^\epsilon|^2\zeta_2^2G &\lesssim {\delta} \left(\sup_{[-r^2,-(\epsilon')^2]} \int_{\R^n_+}|\nabla u^\epsilon|^2\zeta_2^2G  + \int_{A_{\epsilon',r}^+} |\nabla (\nabla u^\epsilon\zeta_2\sqrt{G})|^2\right)\\
	&\lesssim \delta\left( \sup_{[-r^2,-(\epsilon')^2]} \int_{\R^n_+}|\nabla u^\epsilon|^2\zeta_2^2G(\cdot, t)  + \int_{A_{\epsilon',r}^+} |D^2 u^\epsilon|^2\zeta_2^2G\right.\\
	&\qquad \left.+ \int_{A_{\epsilon',r}^+} | \nabla u^\epsilon|^2 |\nabla\zeta_2|^2G +  \int_{A_{\epsilon',r}^+} | \nabla u^\epsilon|^2 \frac{|x|^2}{t^2}\zeta_2^2G\right).\notag
	\end{align*}
	 Therefore $X_2$ can be bounded from below as 
	 \begin{align*}
	 X_2 &\gtrsim  - \sqrt{\delta} \sup_{[-r^2,-(\epsilon')^2]} \int_{\R^n_+}|\nabla u^\epsilon|^2\zeta_2^2G  - \sqrt{\delta} \int_{A_{\epsilon',r}^+} |D^2 u^\epsilon|^2\zeta_2^2G\\
	&\quad -  \int_{A_{\epsilon',r}^+} | \nabla u^\epsilon|^2 |\nabla\zeta_2|^2G -\int_{A_{\epsilon',r}^+} | \nabla u^\epsilon|^2 \frac{|x|^2}{t^2}\zeta_2^2G.
	 \end{align*}
	 We now estimate $X_3$ in \eqref{eq:H2mu} as follows:
	 \begin{align*}
	  X_3&= \int \langle  (A_{\epsilon}-I) \nabla \p_{\mu} u^\epsilon, \nabla (\p_\mu u^\epsilon\zeta_2^2G) \rangle + \int \langle  \nabla \p_{\mu} u^\epsilon, \nabla (\p_\mu u^\epsilon\zeta_2^2G) \rangle.
	  \end{align*}
	  Using  $|A_\epsilon-I|\lesssim \delta$ and Young's inequality, we can estimate the first term above as
	  \begin{align*}
	  &\int \langle  (A_{\epsilon}-I) \nabla \p_{\mu} u^\epsilon, \nabla (\p_\mu u^\epsilon\zeta_2^2G) \rangle
	 \\
	 &=  \int \langle  (A_{\epsilon}-I) \nabla \p_{\mu} u^\epsilon, \nabla \p_\mu u^\epsilon \rangle\zeta_2^2G + 2 \int \langle  (A_{\epsilon}-I) \nabla \p_{\mu} u^\epsilon,  \p_\mu u^\epsilon\zeta_2\nabla \zeta_2G \rangle\\
	 &\quad  + \int \langle  (A_{\epsilon}-I) \nabla \p_{\mu} u^\epsilon,  \p_\mu u^\epsilon\zeta_2^2 \frac{x}{2t} \rangle G \notag\\
	  &\le 3\delta \int |\nabla \p_{\mu} u^\epsilon|^2\zeta_2^2G + \delta \int   |\p_\mu u^\epsilon|^2 |\nabla \zeta_2|^2 G +\frac{\delta}{2}\int  |\p_\mu u^\epsilon|^2 \zeta_2^2 \frac{|x|^2}{4t^2}  G.
	 \end{align*}
Hence, combining the above estimate, \eqref{eq:H2mu} becomes 
	\begin{equation}\label{eq:H2mu21}
	\begin{split}
		&\int \p_{t\mu}u^\epsilon \p_\mu u^\epsilon\zeta_2^2G + \int \langle   \nabla \p_{\mu} u^\epsilon, \nabla (\p_\mu u^\epsilon\zeta_2^2G) \rangle
		+\int f^\epsilon \p_{\mu}(\p_\mu u^\epsilon\zeta_2^2G) \\
		 &\lesssim \sqrt{\delta} \sup_{[-r^2,-(\epsilon')^2]} \int_{\R^n_+}|\nabla u^\epsilon|^2\zeta_2^2G +  \sqrt{\delta} \int |D^2 u^\epsilon|^2\zeta_2^2G  \\
		 &\quad + \int | \nabla u^\epsilon|^2 |\nabla\zeta_2|^2G +  \int | \nabla u^\epsilon|^2 \frac{|x|^2}{t^2}\zeta_2^2G.
		 \end{split}
	\end{equation} 
	We recall that $Z=x \cdot \nabla +2t \partial_t$ and re-write the left-hand side of \eqref{eq:H2mu21} as follows
	\begin{align*}
	&\int \p_{t\mu}u^\epsilon \p_\mu u^\epsilon\zeta_2^2G + \int \langle   \nabla \p_{\mu} u^\epsilon, \nabla (\p_\mu u^\epsilon\zeta_2^2G) \rangle
		+\int f^\epsilon \p_{\mu}(\p_\mu u^\epsilon\zeta_2^2G)\\
		&= \int \frac{1}{4t}Z((\p_\mu u^\epsilon)^2)\zeta_2^2 G +\int |\nabla \p_\mu u^\epsilon|^2 \zeta_2^2 G + 2 \int \langle   \nabla \p_{\mu} u^\epsilon, \nabla \zeta_2) \rangle\p_\mu u^\epsilon \zeta_2 G\\
		& \ \ \ +\int f^\epsilon (\p_{\mu \mu} u^\epsilon) u^\epsilon\zeta_2^2G +2\int f^\epsilon u^\epsilon(\p_{\mu} \zeta_2)\zeta_2G + \int f^\epsilon u^\epsilon  \zeta_2^2 \p_{\mu} G.
	\end{align*}
	For the first term (i.e., $\int \frac{1}{4t}Z((\p_\mu u^\epsilon)^2)\zeta_2^2 G$), we use the arguments from \cite[page 93]{DGPT}, and for the other terms, we apply Young's inequality with $\sqrt{\delta}$. Consequently, \eqref{eq:H2mu21} becomes
	\begin{equation}\label{eq:H2mu2}
	\begin{split}
		 \int_{A_{\epsilon',r}^+}   | \nabla \p_{\mu} u^\epsilon|^2\zeta_2^2G &\lesssim \sqrt{\delta} \sup_{[-r^2,-(\epsilon')^2]} \int_{\R^n_+}|\nabla u^\epsilon|^2\zeta_2^2G +  \sqrt{\delta} \int_{A_{\epsilon',r}^+} |D^2 u^\epsilon|^2\zeta_2^2G \\
		& \quad + \int_{\R^n_+}\p_{\mu} u^\epsilon(\cdot,-r^2)^2\zeta_1^2G(\cdot,-r^2)+ \int_{A_{\epsilon',r}^+}  | \nabla u^\epsilon|^2 \frac{|x|^2}{t^2}\zeta_2^2G\\
		&\quad + \frac{1}{\sqrt{\delta}}\int_{A_{\epsilon',r}^+}  [| \nabla u^\epsilon|^2(\zeta_1^2+ |\nabla\zeta_2|^2) +(f^{\epsilon} \zeta_2)^2]G.
		\end{split}
	\end{equation} 
From equation \eqref{eq:main}, we conclude
 \begin{align*}
	a^{nn}_{\epsilon}\p_{nn}u^\epsilon=\p_t u^\epsilon -\p_{\mu} (a^{ij}_{\epsilon}\p_j {u^\epsilon})-\p_n a^{nn}_{\epsilon}\p_n u^\epsilon -f^{\epsilon}.
\end{align*}
Thus, applications of Young's inequality together with the ellipticity  yield
\begin{align*}
 \int_{A_{\epsilon',r}^+} |\p_{nn}u^\epsilon|  ^2\zeta_2^2G \lesssim 	 \int_{A_{\epsilon',r}^+}|\p_t u^\epsilon|^2 \zeta_2^2G + \int_{A_{\epsilon',r}^+} | \nabla \p_{\mu} u^\epsilon|^2\zeta_2^2G \\
 +\int_{A_{\epsilon',r}^+} |\nabla a^{ij}_{\epsilon}|^2| \nabla u^\epsilon|^2\zeta_2^2G +\int_{A_{\epsilon',r}^+} (f^\epsilon)^2\zeta_2^2G.
\end{align*}
We now bound the third term on the right hand side of the above equation in the same way as in step 2 to get 
\begin{align}\label{eq:h2s3n}
	\int_{A_{\epsilon',r}^+} |\p_{nn}u^\epsilon|  ^2\zeta_2^2G &\lesssim 	 \int_{A_{\epsilon',r}^+}|\p_t u^\epsilon|^2 \zeta_2^2G + \int_{A_{\epsilon',r}^+} | \nabla \p_{\mu} u^\epsilon|^2\zeta_2^2G +\int_{A_{\epsilon',r}^+} (f^\epsilon)^2\zeta_2^2G \\
	&\quad + \delta \sup_t \int_{\R^n_+}|\nabla u^\epsilon|^2\zeta_2^2G  +  \delta \int_{A_{\epsilon',r}^+} |D^2 u^\epsilon|^2\zeta_2^2G\notag\\
	 &\quad + \delta \int_{A_{\epsilon',r}^+} | \nabla u^\epsilon|^2 |\nabla\zeta_2|^2G +  \delta\int_{A_{\epsilon',r}^+} | \nabla u^\epsilon|^2 \frac{|x|^2}{t^2}\zeta_2^2G.\notag
\end{align}
We now combine \eqref{eq:H2mu2} and \eqref{eq:h2s3n} to conclude
	\begin{align*}
	 &\quad \int_{A_{\epsilon', r}^+}   | D^2 u^\epsilon|^2\zeta_2^2G\\ &\lesssim \sqrt{\delta} \sup_{[-r^2,-(\epsilon')^2]} \int_{\R^n_+}|\nabla u^\epsilon|^2\zeta_2^2G+ \int_{A_{\epsilon',r}^+}| \nabla u^\epsilon|^2 \frac{|x|^2}{t^2}\zeta_2^2G+\int_{A_{\epsilon',r}^+}|\p_t u^\epsilon|^2 \zeta_2^2G\\
	&\quad + \int_{\R^n_+}|\nabla u^\epsilon(\cdot,-r^2)|^2\zeta_1^2G(\cdot,-r^2) \\
	&\quad + \sqrt{\delta} \int_{A_{\epsilon',r}^+} |D^2 u^\epsilon|^2\zeta_2^2G +\frac{1}{\sqrt{\delta}} \int [| \nabla u^\epsilon|^2(\zeta_1^2+ |\nabla\zeta_2|^2) +(f^{\epsilon} \zeta_2)^2]G.
\end{align*} 
We now apply \eqref{eq:s2f} with $\tilde{\rho}=r$ to control the first three terms on the right hand side of the above inequality. We find
	\begin{align*}
	&\int_{A_{\epsilon',r}^+}   | D^2 u^\epsilon|^2\zeta_2^2G\lesssim C_{R,\epsilon'}\epsilon+ \int_{\R^n_+}|\nabla  u^\epsilon|^2(\zeta_1^2+\zeta_2^2)G(\cdot,-r^2)\\
	&\quad + \frac{1}{\sqrt{\delta}}\int_{A_{\epsilon',r}^+} [| \nabla u^\epsilon|^2(\zeta_1^2+ |\nabla\zeta_2|^2) +(f^{\epsilon}\zeta_2)^2]G+ \sqrt{\delta} \int_{A_{\epsilon',r}^+} |D^2 u^\epsilon|^2\zeta_2^2G.
\end{align*} 
We choose $\delta>0$ further small enough depending on $n$. Consequently, we conclude
\begin{align*}
	  &\quad \int_{A_{\epsilon',r}^+}   | D^2 u^\epsilon|^2\zeta_2^2G\\
	 &\lesssim C_{R,\epsilon'}\epsilon+ \int_{\R^n_+}|\nabla  u^\epsilon|^2(\zeta_1^2+\zeta_2^2)G(\cdot,-r^2)+\int_{A_{\epsilon',r}^+} [| \nabla u^\epsilon|^2(\zeta_1^2+ |\nabla\zeta_2|^2) +(f^{\epsilon}\zeta_2)^2]G.
\end{align*}
Now, integrating over $r \in [\rho,\frac{7}{4}\rho]$ and noting that $0\leq \zeta_2\leq \zeta_1$ for $|t|\leq 1$ yields
\begin{align*}
	& \int_{A_{\epsilon',\rho}^+}   | D^2u^\epsilon|^2\zeta_2^2G \lesssim C_{R,\epsilon'}\epsilon+\int_{A_{\epsilon',7\rho/4}^+} [| \nabla u^\epsilon|^2(\zeta_1^2+ |\nabla\zeta_2|^2) +(f^{\epsilon} \zeta_2)^2]G.
\end{align*} 
Using \eqref{eq:s2f} and repeating the same procedure as above we get
\begin{align*}
& \quad \sup_{t\in [-r^2, -(\epsilon')^2]}\int_{\R^n_+}|\nabla u^\epsilon|^2\zeta_2^2G(\cdot, t)+ \int_{A_{\epsilon',\rho}^+}| \nabla u^\epsilon|^2 \frac{|x|^2}{t^2}\zeta_2^2G\\
	&\quad +\int_{A_{\epsilon',\rho}^+}|\p_t u^\epsilon|^2 \zeta_2^2G+ \int_{A_{\epsilon',\rho}^+}   | D^2u^\epsilon|^2\zeta_2^2G\\
	&\lesssim C_{R,\epsilon'}\epsilon+\int_{A_{\epsilon',7\rho/4}^+} [| \nabla u^\epsilon|^2(\zeta_1^2+ |\nabla\zeta_2|^2) +(f^{\epsilon} \zeta_2)^2]G.
\end{align*}
Using the gradient estimate in \emph{step 1}, cf. \eqref{eq:grad_est_eps}, and that 
\begin{equation*}
|\nabla\zeta_k|\lesssim |\zeta_{k-1}|+ |t|^{\frac{k-1}{2}}(1+R)^{s}\chi_{\{R-1\leq \frac{|x|}{\sqrt{|t|}}\leq R\}},\quad \forall k\in \{1,2,\cdots\},
\end{equation*}
 we get
\begin{align*}
\int_{A_{\epsilon',7\rho/4}^+} [| \nabla u^\epsilon|^2(\zeta_1^2+ |\nabla\zeta_2|^2)G&\lesssim \int_{A_{\epsilon',2\rho}^+} [|u^\epsilon|^2 (\zeta_0^2 + |\nabla \zeta_1|^2)+|f^\epsilon|^2\zeta_2^2] G + e^{-\frac{R^2}{8}}\int_{B_R^+\times [-(2\rho)^2,-(\epsilon')^2]} |\nabla u^\epsilon|^2\\
\end{align*}
Since $\int_{B_R^+\times [-(2\rho)^2, -(\epsilon')^2]} |\nabla u^\epsilon|^2\leq C_n \|f\|_{L^2(B_R^+\times [-1,0])}^2$ by the global energy bound, cf. \cite[Chapter 3]{DGPT}, combining together with the gradient estimate in \emph{step 1} we thus get
\begin{equation}\label{eq:H2-eps}
\begin{split}
& \quad \sup_{t\in [-\rho^2, -(\epsilon')^2]}\int_{\R^n_+}|t|^2|\nabla u^\epsilon|^2\zeta_0^2G(\cdot, t)+ \int_{A_{\epsilon',\rho}^+}|t|| \nabla u^\epsilon|^2 \zeta_0^2G\\
	&\quad +\int_{A_{\epsilon',\rho}^+}|t|^2\left(|\p_t u^\epsilon|^2 + |D^2u^\epsilon|^2\right) \zeta_0^2G\\
	&\lesssim C_{R,\epsilon'}\epsilon+C e^{-R^2/8} + \int_{A_{\epsilon',2\rho}^+}\left( |u^\epsilon|^2(\zeta_0^2 + |\nabla \zeta_1|^2) +|f^{\epsilon}|^2\zeta_2^2\right) G.
	\end{split}
\end{equation}
Now we pass first the limit $\epsilon \rightarrow 0$, then $R \rightarrow \infty$ and finally $\epsilon' \rightarrow 0$ in the above inequality, we obtain the desired $H^2$-estimate \eqref{eq:H2}. The proof for the lemma is thus complete.  
\end{proof}

The following lemma is an analogue of \cite[Lemma 5.3]{DGPT}. 
\begin{lem}\label{lem:energy_close}
Let $u_1\in \mathcal{G}^{A_1,f_1}_p(S_1^+)$ and  $u_2\in \mathcal{G}_p^{A_2,f_2}(S_1^+)$. Then there exists $\delta=\delta(n)>0$ such that if $\sup_{S_1^+}|A_i-I|\leq \delta$, $i=1,2$, then for $0<\rho<1/4$ and $\omega:=1+\frac{|x|}{\sqrt{-t}}$, we have
\begin{align*}
\int_{S_{\rho}^+} |t||\nabla (u_1-u_2)|^2G   &\lesssim \int_{S_{2\rho}^+} (| u_1- u_2|^2 + |t|^2| f_1- f_2|^2)G \\
&+ \|A_1-A_2\|^{2}_{L^{\infty}({S_{2\rho}^+})}\int_{S_{4\rho}^+} \omega^2(u_2^2+ t^2f_2^2) G.
\end{align*}
\end{lem}
\begin{proof}
The proof is similar to Step 1 in the proof of Lemma \ref{lem:H2}. We will here use the same notations. Let $u_i^{\epsilon}$ for $i=1,2$ be the approximations of $u_1$ and $u_2$. We here take $\eta=(u^\epsilon_1-u^\epsilon_2) \zeta_1^2G$ to obtain, for $i=1,2$ and $\epsilon'\ll 1$, 
	\begin{align*}
		&\int_{A^+_{\epsilon',r}} \p_tu^\epsilon_i (u^\epsilon_1-u^\epsilon_2) \zeta_1^2G + \int_{A^+_{\epsilon',r}} \langle A_i \nabla u^\epsilon_i, \nabla((u^\epsilon_1-u^\epsilon_2)\zeta_1^2G)\rangle \\ 
		&+\int_{A'_{\epsilon',r}}\beta_\epsilon(u^\epsilon_i)(u^\epsilon_1-u^\epsilon_2) \zeta_1^2G = \int_{A^+_{\epsilon',r}} f_i^\epsilon (u^\epsilon_1-u^\epsilon_2) \zeta_1^2 G.
	\end{align*}
Note that $\beta_{\epsilon}'\ge 0$ gives
$$(\beta_\epsilon(u^\epsilon_1)-\beta_\epsilon(u^\epsilon_2))(u^\epsilon_1-u^\epsilon_2) \ge 0.$$
Therefore it is easy to find
\begin{align*}
&\int_{A^+_{\epsilon',r}} \p_t(u^\epsilon_1-u^\epsilon_2) (u^\epsilon_1-u^\epsilon_2) \zeta_1^2G + \int_{A^+_{\epsilon',r}} \langle A_1 \nabla(u^\epsilon_1-u^\epsilon_2), \nabla((u^\epsilon_1-u^\epsilon_2) \zeta_1^2G)\rangle \\
&\le \int_{A^+_{\epsilon',r}} (f_1^\epsilon-f_2^\epsilon) (u^\epsilon_1-u^\epsilon_2) \zeta_1^2G +\int_{A^+_{\epsilon',r}} \langle (A_1-A_2) \nabla u^\epsilon_2, \nabla((u^\epsilon_1-u^\epsilon_2) \zeta_1^2G)\rangle.
\end{align*}
We now proceed as in Step 1 of proof of Lemma \ref{lem:H2} (more precisely, from \eqref{eq:grad_est_start} to \eqref{eq:grad_est_eps}) to obtain
 \begin{align*}
 \int_{A^+_{\epsilon',r}}  |\nabla(u^\epsilon_1-u^\epsilon_2)|^2 \zeta_1^2G 
&\lesssim \int_{A^+_{\epsilon',r}} (\zeta_0^2+|\nabla\zeta_1|^2)(u^\epsilon_1-u^\epsilon_2)^2G+\int_{A^+_{\epsilon',r}} (f_1^\epsilon-f_2^\epsilon)^2\zeta_2^2G \\
&+ \int_{A^+_{\epsilon',r}} \langle (A_1-A_2) \nabla u^\epsilon_2, \nabla((u^\epsilon_1-u^\epsilon_2) \zeta_1^2G)\rangle.
\end{align*}
Using Young's inequality in the last term, we get 
\begin{align*}
 \int_{A^+_{\epsilon',r}}  |\nabla(u^\epsilon_1-u^\epsilon_2)|^2 \zeta_1^2G 
&\lesssim \int_{A^+_{\epsilon',r}} (\zeta_0^2+|\nabla\zeta_1|^2)(u^\epsilon_1-u^\epsilon_2)^2G+\int (f_1^\epsilon-f_2^\epsilon)^2\zeta_2^2G \\
&+ \int_{A^+_{\epsilon',r}} |A_1-A_2|^{2}(\zeta_1^2+|x|^2\zeta_0^2) |\nabla u^\epsilon_2|^2 G.
\end{align*}
As in Step 1 of Lemma \ref{lem:H2}, passing to the limit $\epsilon\rightarrow 0$ and then $R\rightarrow \infty$ and $\epsilon'\rightarrow 0$, we conclude 
\begin{align*}
	\int_{S^+_{\rho}} |t| |\nabla(u_1-u_2)|^2 G 
	&\lesssim \int_{S^+_{2\rho}} [(u_1-u_2)^2+t^2(f_1-f_2)^2G\\
	&+ \|A_1-A_2\|^{2}_{L^{\infty}({S^+_{2\rho}})} \int_{S^+_{2\rho}} |t|\omega^2 |\nabla u_2|^2 G.
\end{align*}
Finally we use the energy inequality to find
\begin{align*}
	\int_{S^+_{\rho}} |t| |\nabla(u_1-u_2)|^2 G 
	&\lesssim \int_{S^+_{2\rho}} [(u_1-u_2)^2+t^2(f_1-f_2)^2G\\
	&+ \|A_1-A_2\|^{2}_{L^{\infty}({S^+_{2\rho}})} \int_{S^+_{4\rho}}  \omega^2 (u_2^2 +t^2f_2^2) G.
\end{align*}
This completes the proof of the Lemma \ref{lem:energy_close}.
\end{proof}

\section*{Acknowledgements}
V.A. has been supported by the Academy of Finland grant 347550. This work began during a visit by V.A. to the School of Mathematics at Monash University in Melbourne, in July 2022. He gratefully acknowledges the gracious hospitality of the Centre and the warm work environment. W.S. has been supported by the German Research Foundation (DFG) within the SFB 1481 (INST 222/1473-1).

\bibliography{final}
\bibliographystyle{abbrv}

\end{document}